\documentclass[reqno,tbtags,a4paper,12pt]{amsart}

\usepackage[in]{fullpage}
\usepackage{makecell}
\usepackage[matrix,arrow,curve]{xy}
\usepackage{tikz}
\usetikzlibrary{math}
\usetikzlibrary{shapes}

\usepackage{amsmath,amssymb,amsthm,amsfonts,amscd,mathtools,multirow,afterpage}
\usepackage{graphics,graphicx}

\usepackage{mathtools}

\newcommand{\Z}{\mathbb{Z}}
\newcommand{\Q}{\mathbb{Q}}

\newcommand{\RP}{\mathbb{RP}}
\newcommand{\CP}{\mathbb{CP}}
\newcommand{\HP}{\mathbb{HP}}

\newcommand{\OP}{\mathbb{OP}}
\newcommand{\F}{\mathbb{F}}

\newcommand{\CG}{\mathcal{G}}

\newcommand{\rC}{\mathrm{C}}
\newcommand{\rS}{\mathrm{S}}
\newcommand{\rA}{\mathrm{A}}

\newcommand{\Sym}{\mathop{\mathrm{Sym}}\nolimits}

\newcommand{\link}{\mathop{\mathrm{link}}\nolimits}
\newcommand{\rank}{\mathop{\mathrm{rank}}\nolimits}

\newcommand{\cert}{\mathop{\mathrm{cert}}\nolimits}

\newcommand{\tM}{\widetilde{M}\vphantom{M}}
\newcommand{\ttM}{\widetilde{\widetilde{M}}\vphantom{M}}

\newcommand{\pt}{\mathrm{pt}}

\newtheorem{theorem}{Theorem} [section]
\newtheorem{propos}[theorem] {Proposition}

\newtheorem{cor}[theorem] {Corollary}

\newtheorem{conj}[theorem] {Conjecture}
\newtheorem{problem}[theorem] {Problem}
\newtheorem{question}[theorem] {Question}
\theoremstyle{definition}
\newtheorem{observe}[theorem] {Observation}
\newtheorem{remark}[theorem]{Remark}
\newtheorem{defin}[theorem]{Definition}
\newtheorem{comp}[theorem]{Computer result}

\numberwithin{equation}{section}

\author{Alexander A. Gaifullin}

\address{Steklov Mathematical Institute of Russian Academy of Sciences, Moscow, Russia}
\address{Skolkovo Institute of Science and Technology, Moscow, Russia}
\address{Lomonosov Moscow State University, Moscow, Russia}
\address{Institute for the Information Transmission Problems of the Russian Academy of Sciences (Kharkevich Institute), Moscow, Russia}
\email{agaif@mi-ras.ru}

\thanks{}

\title{New examples and partial classification of 15-vertex triangulations of the quaternionic projective plane}

\date{}



\begin{document}
\begin{abstract}
 Brehm and K\"uhnel (1992) constructed three $15$-vertex combinatorial $8$-manifolds `like the quaternionic projective plane' with symmetry groups~$\rA_5$, $\rA_4$, and~$\rS_3$, respectively. Gorodkov (2016) proved that these three manifolds are in fact PL homeomorphic to~$\HP^2$. Note that $15$ is the minimal number of vertices of a combinatorial $8$-manifold that is not PL homeomorphic to~$S^8$. In the present paper we construct a lot of new $15$-vertex triangulations of~$\HP^2$. A surprising fact is that such examples are found for very different symmetry groups, including those not in any way related to the group~$\rA_5$. Namely, we find $19$~triangulations with symmetry group~$\rC_7$, one triangulation with symmetry group~$\rC_6\times\rC_2$, $14$~triangulations with symmetry group~$\rC_6$, $26$~triangulations with symmetry group~$\rC_5$, one new triangulation with symmetry group~$\rA_4$, and $11$~new triangulations with symmetry group~$\rS_3$. Further, we obtain the following classification result. We prove that, up to isomorphism, there are exactly $75$ triangulations of~$\HP^2$ with $15$ vertices and symmetry group of order at least~$4$: the three Brehm--K\"uhnel triangulations and the $72$ new triangulations listed above.
On the other hand, we show that there are plenty of triangulations with symmetry groups~$\rC_3$ and~$\rC_2$, as well as the trivial symmetry group.
\end{abstract}

\maketitle

\begin{flushright}\textit{To my Advisor Victor Buchstaber\\
with deep gratitude for everything I learned from him,\\
on the occasion of his 80th birthday}
\end{flushright}

\section{Introduction}\label{section_intro}

Throughout the paper, we denote by~$\rC_n$, $\rS_n$, and~$\rA_n$ the cyclic group of order~$n$,  the symmetric group of degree~$n$, and the alternating group of degree~$n$, respectively.

Recall that a simplicial complex~$K$ is called a \textit{combinatorial $d$-manifold} if the link of every vertex of~$K$ is PL homeomorphic to the standard $(d-1)$-sphere.

In 1992 Brehm and K\"uhnel~\cite{BrKu92} constructed three $15$-vertex combinatorial $8$-manifolds `like the quaternionic projective plane', denoted by~$M^8_{15}$, $\tM_{15}^8$, and~$\ttM_{15}^8$, with symmetry groups~$\rA_5$, $\rA_4$, and~$\rS_3$, respectively. Three more
examples of such combinatorial manifolds were constructed by Lutz\footnote{Unfortunately, the lists of simplices of the three Lutz triangulations are not available via the link given in~\cite{Lut05} and are apparently lost. Therefore, it is not clear how the Lutz triangulations are related to the triangulations that will be constructed in the present paper.}~\cite{Lut05}  by applying his program \texttt{BISTELLAR} to~$M^8_{15}$, $\tM_{15}^8$, and~$\ttM_{15}^8$.
Gorodkov~\cite{Gor16,Gor19} proved that the three Brehm--K\"uhnel combinatorial manifolds (and hence also the three combinatorial manifolds constructed by Lutz) are in fact PL triangulations of~$\HP^2$. Interest in such triangulations is caused by the following result.

\begin{theorem}[Brehm, K\"uhnel~\cite{BrKu87}]\label{thm_BK}
 Suppose that $K$ is a combinatorial $d$-manifold with $n$ vertices.
 \begin{enumerate}
  \item If $n<3d/2+3$, then $K$ is PL homeomorphic to the standard sphere~$S^d$.
  \item If $n=3d/2+3$, then either $K$  is PL homeomorphic to~$S^d$ or $d\in\{2,4,8,16\}$ and $K$ is a `manifold like a projective plane', that is, $K$ admits a PL Morse function with exactly $3$ critical points.
 \end{enumerate}
\end{theorem}

Thus, $n$-vertex (combinatorial) triangulations of $d$-manifolds not homeomorphic to~$S^d$, where $(d,n)$ is one of the pairs $(2,6)$, $(4,9)$, $(8,15)$, or~$(16,27)$, are of particular interest. In each of dimensions~$2$ and~$4$, such a triangulation is unique. Namely, these are the triangulation~$\RP^2_6$ of real projective plane obtained by taking the quotient of the boundary of regular icosahedron by the antipodal involution and  the triangulation~$\CP^2_9$ of complex projective plane constructed by K\"uhnel, see~\cite{KuBa83}. The uniqueness in the $4$-dimensional case was proved by K\"uhnel and Lassmann~\cite{KuLa83}, see also~\cite{BaDa01}. In the present paper we concentrate on the $8$-dimensional case. In dimension~$16$ the first examples of the $27$-vertex combinatorial manifolds that are not homeomorphic to~$S^{16}$ have been recently constructed by the author~\cite{Gai22}, see also~\cite{Gai23}. Though in the present paper we never consider the $16$-dimensional case, we will constantly use methods suggested in~\cite{Gai22} and~\cite{Gai23}.

For a finite simplicial complex~$K$, we denote by~$\Sym(K)$ the \textit{symmetry group} of~$K$, i.\,e., the group of all simplicial automorphisms of~$K$. Suppose that $K$ has $n$ vertices. Numbering the vertices of~$K$ from~$1$ to~$n$, we realize~$\Sym(K)$ as a subgroup of~$\rS_n$. Moreover, different numberings of vertices provide conjugate subgroups of~$\rS_n$.

The main result of the present paper is as follows.

{\sloppy
\begin{theorem}\label{thm_main}
\begin{enumerate}
 \item Up to conjugation, there are exactly $10$ subgroups $G\subset \rS_{15}$ for which there exists a $15$-vertex combinatorial triangulation~$K$ of~$\HP^2$ with $\Sym(K)=G$, see Table~\ref{table_main}.
 \item The number of different (up to isomorphism) $15$-vertex combinatorial triangulations of\/~$\HP^2$ with each symmetry group, except for~$\rC_3$, $\rC_2$, and the trivial group~$\rC_1$, is given in the last column of Table~\ref{table_main}.
 \item For the three symmetry groups~$\rC_3$, $\rC_2$, and~$\rC_1$, there are at least as many $15$-vertex combinatorial triangulations of~$\HP^2$ as the number in the last column of Table~\ref{table_main}.
 \end{enumerate}
\end{theorem}
}

\begin{table}
\caption{Numbers of $15$-vertex triangulations of~$\HP^2$ with given symmetry groups}\label{table_main}
\begin{tabular}{|c|c|c|c|}
\hline
 \textbf{Group} & \textbf{Order} & \textbf{Action on vertices} & \textbf{Number of}\\
 & & \textbf{(orbit lengths)} & \textbf{triangulations}\\
\hline
 $\rA_5$ & 60 & transitive & 1\\
 $\rA_4$ & 12 & 12, 3 & 2\\
 $\rC_6\times\rC_2$ & 12 & 12, 3 & 1 \\
 $\rC_7$ & 7 & 7, 7, 1 & 19 \\
 $\rS_{3}$ & 6 & 6, 3, 3, 3 & 12\\
 $\rC_6$ & 6 & 6, 6, 3 & 14 \\
 $\rC_5$ & 5 & free & 26 \\
 $\rC_3$ & 3 & free & $\ge 4617$ \\
 $\rC_2$ & 2 & with 3 fixed points & $\ge 3345$ \\
 $\rC_1$ & 1 & & $> 670000$ \\
\hline
\end{tabular}
\end{table}

\begin{remark}
 For each group~$G$ in Table~\ref{table_main}, its embedding into~$\rS_{15}$ is specified by pointing out the list of orbit lengths for the action of~$G$ on the set $\{1,\ldots,15\}$. It is easy to check that in each of the cases these data determine a subgroup $G\subset\rS_{15}$ uniquely up to conjugation.
\end{remark}

Perhaps the most surprising part of our result is the construction of new triangulations of~$\HP^2$. Unexpectedly, it was possible to construct triangulations with symmetry groups~$\rC_7$, $\rC_6\times\rC_2$, and~$\rC_6$, which are in no way related to the symmetry groups~$\rA_5$, $\rA_4$, and~$\rS_3$ of the three Brehm--K\"uhnel triangulations. The construction of new triangulations of~$\HP^2$ follows the approach developed by the author in~\cite{Gai22} in the $16$-dimensional case. To obtain the classification result (assertions~(1) and~(2) of Theorem~\ref{thm_main}) we also use the technique developed by the author in~\cite{Gai23}. This technique works not only for combinatorial manifolds but also for homology manifolds.

\begin{defin}\label{defin_hm}
Suppose that $R$ is a commutative ring.
A finite simplicial complex~$K$ is called an $R$-\textit{homology} $d$-\textit{manifold} if
\begin{enumerate}
 \item every simplex of~$K$ is contained in a $d$-simplex,
 \item for each nonempty simplex $\sigma\in K$ with $\dim\sigma<d$, there is an isomorphism of graded $R$-modules
 $$
 H_*\bigl(\link(\sigma,K); R\bigr)\cong H_*\bigl(S^{d-\dim\sigma-1};R\bigr).
 $$
\end{enumerate}
\end{defin}

In the present paper $R$ will always be either~$\Z$ or a finite field~$\F_p$, where $p$ is a prime. From the universal coefficient theorem it follows that a $\Z$-homology $d$-manifold is an $\F_p$-homology $d$-manifold for every prime~$p$.

An analogue of Theorem~\ref{thm_BK} for homology manifolds is as follows, see~\cite[Remark~1.7]{Gai23} for an explanation of how to extract this result from~\cite{Nov98}.

\begin{theorem}[Novik,~\cite{Nov98}]\label{thm_Novik}
 Suppose that $K$ is a $\Z$-homology $d$-manifold with $n$ vertices.
 \begin{enumerate}
  \item If $n<3d/2+3$, then $H_*(K;\Z)\cong H_*(S^d;\Z)$.
  \item If $n=3d/2+3$, then either $H_*(K;\Z)\cong H_*(S^d;\Z)$ or  the following assertions hold:
 \begin{itemize}
  \item $d\in\{2,4,8,16\}$,
  \item if $d=2$, then $K\cong\RP^2_6$,
  \item if $d\in\{4,8,16\}$, then
  \begin{equation*}
  H^*(K;\Z)\cong\Z[a]/(a^3),\qquad \deg a =d/2.
  \end{equation*}
 \end{itemize}
 \end{enumerate}
\end{theorem}

A natural conjecture (cf. Conjecture~0 in~\cite[\S 20]{ArMa91}) says that any $(3d/2+3)$-vertex $\Z$-homology $d$-manifold~$K$ with $H_*(K;\Z)\ncong H_*(S^d;\Z)$ is a combinatorial manifold. We will prove this conjecture in the case $d=8$ and $|\Sym(K)|>3$.

\begin{theorem}\label{thm_homology}
 Suppose that $K$ is a $\Z$-homology $8$-manifold such that  $H_*(K;\Z)$ is not isomorphic to~$H_*(S^8;\Z)$ and $|\Sym(K)|>3$. Then $K$ is isomorphic to one of the $75$ combinatorial manifolds listed in the first seven rows of Table~\ref{table_main}. In particular, $K$ is a combinatorial manifold PL homeomorphic to~$\HP^2$.
\end{theorem}

This paper is organized as follows. In Section~\ref{section_finite_order} we study the fixed point complexes~$K^{\rC_p}$ for simplicial actions of a cyclic group~$\rC_p$ of prime order~$p$ on a $15$-vertex $\Z$-homology $8$-manifold~$K$ with $H_*(K;\Z)\ncong H_*(S^8;\Z)$. Analogous results in the $16$-dimensional case were obtained in~\cite[Section~7]{Gai23}.

Sections~\ref{section_>=7}--\ref{section_5} contain constructions of new $15$-vertex combinatorial $8$-manifolds like a projective plane with different symmetry groups~$G$ of orders $|G|>3$. For $|G|\ge 7$, these combinatorial manifolds are constructed by directly applying the computer program \texttt{find} from~\cite{Gai-prog} that was originally written by the author  for finding $27$-vertex triangulations of $16$-manifolds like a projective plane. For $4\le |G|\le 6$, it seems that the program \texttt{find} run directly does not complete in a reasonable time, so we preface its use with the use of theoretical results from Section~\ref{section_finite_order}. In total, in Sections~\ref{section_>=7}--\ref{section_5} we construct exactly the $75$ triangulations~$K$ with $|\Sym(K)|>3$ from Table~\ref{table_main}.

In Section~\ref{section_classify} we combine the results from Sections~\ref{section_finite_order}--\ref{section_5} to obtain a complete classification of  $15$-vertex $\Z$-homology $8$-manifolds~$K$ such that $H_*(K;\Z)\ncong H_*(S^8;\Z)$ and $|\Sym(K)|>3$.

In Section~\ref{section_G}, we introduce and study the \textit{triple flip graph}~$\CG$ whose vertices are $15$-vertex combinatorial $8$-manifolds like a projective plane and whose edges are \textit{triple flips}, that is, moves of the form
$$
(\Delta_1*\partial\Delta_2)\cup(\Delta_2*\partial\Delta_3)\cup(\Delta_3*\partial\Delta_1)\rightsquigarrow
(\partial\Delta_1*\Delta_2)\cup(\partial\Delta_2*\Delta_3)\cup(\partial\Delta_3*\Delta_1),
$$
where $\Delta_1$, $\Delta_2$, and~$\Delta_3$ are $4$-simplices. Also we study the equivariant analogs of~$\CG$. As a result, we construct a lot of $15$-vertex combinatorial $8$-manifolds like a projective plane with symmetry groups~$\rC_3$, $\rC_2$, and~$\rC_1$.

The question of the topological type of the constructed combinatorial manifolds is addressed in Section~\ref{section_p1}. A result by Kramer~\cite{Kra03} allows us to reduce this question to the calculation of the first rational Pontryagin class. By my request, Denis Gorodkov has performed this calculation, using his program~\cite{Gor-prog}. It follows that all the combinatorial manifolds constructed in the present paper are PL homeomorphic to~$\HP^2$, see Proposition~\ref{propos_Gorodkov}.

Finally, Section~\ref{section_conclude} contains some concluding remarks and open problems.

The article does not contain a section explicitly entitled `proof of Theorem~\ref{thm_main}' or  `proof of Theorem~\ref{thm_homology}'. Theorems~\ref{thm_main} and~\ref{thm_homology} follow immediately from Propositions~\ref{propos_class}, \ref{propos_5_man_1}, and~\ref{propos_Gorodkov} and Corollary~\ref{cor_5_man_2,3} that we will prove throughout the paper.

\begin{remark}
 Lists of maximal simplices of the constructed $15$-vertex triangulations of~$\HP^2$ can be found in~\cite{Gai-prog}. Namely, for each group~$G$  from Table~\ref{table_main}, except for $\rA_5$, $\rC_3$, $\rC_2$, and~$\rC_1$, the file entitled \texttt{HP2\_15\_G.dat} contains the lists of maximal simplices of all triangulations~$K$ with $\Sym(K)\cong G$. In addition, the file \texttt{HP2\_15\_G-component.dat} contains the lists of maximal simplices of all~$K$ in the connected component~$\CG_0$ of the Brehm--K\"uhnel triangulations in the triple flip graph~$\CG$, see Section~\ref{section_G}. In these lists each simplex is encoded with a row of $15$ binary digits in which the $i$th digit indicates whether the $i$th vertex is in the simplex or not. For each triangulation we have the freedom to renumber its vertices. In each of the files \texttt{HP2\_15\_G.dat} we number the vertices so that all triangulations in the file have coinciding (rather than just conjugate) symmetry groups. On the other hand, in the file \texttt{HP2\_15\_G-component.dat} we give every triangulation with the numbering of its vertices produced by our program \texttt{triple\_flip\_graph}, see Section~\ref{section_G}. As a result, two triangulations, $\HP^2_{15}(\rS_3,1)$ and~$\HP^2_{15}(\rS_3,2)$, are given in the files \texttt{HP2\_15\_S3.dat} and \texttt{HP2\_15\_G-component.dat} with different orderings of their vertices.
\end{remark}

\section{Fixed point complexes for symmetries of prime orders}\label{section_finite_order}
Suppose that $K$ is a finite simplicial complex with vertex set~$V$ and $G\subseteq\Sym(K)$ is a subgroup. We use the language of abstract simplicial complexes so that simplices of~$K$ are subsets of~$V$. Let $|K|$ be the geometric realization of~$K$ and $|K|^G$ the set of all $G$-fixed points in~$|K|$. Recall that $|K|^G$ is the geometric realization of the simplicial complex~$K^G$ such that
\begin{itemize}
 \item vertices of~$K^G$ are the barycentres $b(\sigma_1),\ldots,b(\sigma_s)$, where $\sigma_1,\ldots,\sigma_s$ are all $G$-orbits in~$V$ that are simplices of~$K$,
 \item a subset $\bigl\{b(\sigma_{i_1}),\ldots,b(\sigma_{i_t})\bigr\}$ is a simplex of~$K^G$ if and only if $\sigma_{i_1}\cup\cdots\cup\sigma_{i_t}\in K$.
\end{itemize}
The complex~$K^G$ will be called the \textit{$G$-fixed point complex}, cf.~\cite[Section~2]{Gai23}.

\subsection{Preliminaries}
Suppose that $K$ is a $15$-vertex $\Z$-homology $8$-manifold such that the graded group $H_*(K;\Z)$ is not isomorphic to~$H_*(S^8;\Z)$. We are going to study the fixed point set complexes $K^{\langle g\rangle}$ for elements $g\in\Sym(K)$ of prime orders. We will use the same approach that was used in~\cite[Section~7]{Gai23} in the $16$-dimensional case.
It is based on the following three theorems (Theorems~\ref{thm_Smith}, \ref{thm_ArMa}, and~\ref{thm_Datta}), see~\cite[Sections~2]{Gai23} for more details and references.

\begin{theorem}[Smith,~\cite{Smi39}, cf.~\cite{Bor60}]\label{thm_Smith}
 Suppose that $K$ is an $\F_p$-homology manifold with a simplicial action of the group~$\rC_p$, where $p$ is a prime. Then every connected component of~$K^{\rC_p}$ is also an $\F_p$-homology manifold.
\end{theorem}

\begin{remark}\label{remark_smaller_dim}
 If $K$ is connected and the action of~$\rC_p$ is faithful, then we always have $\dim K^{\rC_p}<\dim K$, since a connected $\F_p$-homology manifold does not  contain a proper $\F_p$-homology submanifold of the same dimension.
\end{remark}

\begin{remark}
 Unlike the $16$-dimensional case, we will not need more complicated results by Bredon on the fixed point sets for the actions of~$\rC_p$ on homology projective spaces.
\end{remark}

The two remaining  theorems are related to the concept of complementarity, which is important for us.

\begin{defin}
 We say that a finite simplicial complex~$K$ on vertex set~$V$ \textit{satisfies complementarity} if, for each subset $\sigma\subseteq V$, exactly one of the two subsets~$\sigma$ and~$V\setminus \sigma$ is a simplex of~$K$.
\end{defin}

This property is sometimes called \textit{duality}, since it is equivalent to saying that the (combinatorial) Alexander dual of~$K$ coincides with $K$.
If $K$ has $n$ vertices and $\dim K=d$, then the complimentarity property implies that $K$ is \textit{$(n-d-2)$-neighborly}, that is, every $(n-d-2)$-element subset of~$V$ is a simplex of~$K$.

\begin{theorem}[Arnoux, Marin,~\cite{ArMa91}]\label{thm_ArMa}
 Suppose that $K$ is a simplicial complex such that the cohomology ring~$H^*(K;\F_2)$ contains a (graded) subring isomorphic to $\F_2[a]/(a^3)$, where $m=\deg a$ is even. Then $K$ has not less than $3m+3$ vertices. Moreover, if $K$ has exactly $3m+3$ vertices, then $K$ satisfies complimentarity.
\end{theorem}

Theorems~\ref{thm_Novik} and~\ref{thm_ArMa} immediately imply the following.
\begin{cor}\label{cor_complement}
 Suppose that $K$ is a $\Z$-homology $d$-manifold with $3d/2+3$ vertices such that $H_*(K;\Z)$ is not isomorphic to~$H_*(S^d;\Z)$. Then $K$ satisfies complementarity. In particular, $K$ is  $(d/2+1)$-neighborly.
\end{cor}

Let $f_k(K)$ denote the number of $k$-simplices of~$K$ and $\chi(K)$ the Euler characteristic of~$K$.
Brehm and K\"uhnel~\cite[Section~1]{BrKu92} showed that in the $8$-dimensional case the complimentarity condition, the Dehn--Sommerville equations for homology manifolds (due to Klee~\cite{Kle64}), and the condition $\chi(K)=3$ form a system of linear equations that completely determines the $f$-vector of~$K$. Their result together with Theorem~\ref{thm_Novik} and Corollary~\ref{cor_complement} implies the following assertion.

\begin{cor}\label{cor_fvect}
 Suppose that $K$ is a $\Z$-homology $d$-manifold with $3d/2+3$ vertices such that $H_*(K;\Z)$ is not isomorphic to~$H_*(S^d;\Z)$. Then the $f$-vector of~$K$ is as follows:
 \begin{equation}\label{eq_fvect_BK}
  \begin{aligned}
   f_0(K)&=15,&&& f_1(K)&=105,&&&f_2(K)&=455,\\
   f_3(K)&=1365,&&& f_4(K)&=3003,&&&f_5(K)&=4515,\\
   f_6(K)&=4230,&&& f_7(K)&=2205,&&&f_8(K)&=490.
  \end{aligned}
 \end{equation}
\end{cor}

\begin{defin}
 A finite simplicial complex~$K$ is called a \textit{weak $d$-pseudomanifold} (where $d>0$) if it satisfies the following two conditions:
 \begin{enumerate}
  \item every simplex of~$K$ is contained in a $d$-simplex of~$K$,
  \item every $(d-1)$-simplex of~$K$ is contained in exactly two $d$-simplices of~$K$.
  \end{enumerate}
  A weak $d$-pseudomanifold $K$ is called a \textit{$d$-pseudomanifold} if, in addition, it satisfies the condition
  \begin{itemize}
  \item[(3)] $K$ is \textit{strongly connected}, i.\,e., for any two $d$-simplices $\sigma,\tau\in K$, there exists a sequence of $d$-simplices $\sigma=\rho_1,\rho_2,\ldots,\rho_n=\tau$ such that $\dim(\rho_i\cap\rho_{i+1})=d-1$ for all~$i$.
  \end{itemize}
  We will conveniently say that a \textit{weak $0$-pseudomanifold} is the disjoint union of a finite number of points and a \textit{$0$-pseudomanifold} is a point.
\end{defin}

Note that every $\F_p$-homology $d$-manifold is a weak $d$-pseudomanifold and every connected $\F_p$-homology $d$-manifold is a $d$-pseudomanifold.

\begin{theorem}[Datta,~\cite{Dat98}; Bagchi, Datta,~\cite{BaDa04}]\label{thm_Datta}
Suppose that $K$ is a weak $d$-pseudo\-manifold that satisfies complementarity. Then one of the following assertions hold:
\begin{itemize}
 \item $K\cong\pt\sqcup\pt\sqcup\pt$  (the disjoint union of three points),
 \item $K\cong\RP^2_6$,
 \item $K\cong\CP^2_9$,
 \item $d\ge 7$ and the number of vertices of~$K$ is at least~$d+7$.
\end{itemize}
Moreover, the Euler characteristic $\chi(K)$ is odd.
\end{theorem}

\subsection{Main result on fixed point complexes}

For an element~$g$ of a group, we denote by~$\langle g\rangle$ the cyclic subgroup generated by~$g$.

\begin{propos}\label{propos_p}
Suppose that $K$ is a $15$-vertex $\Z$-homology $8$-manifold such that the graded group $H_*(K;\Z)$ is not isomorphic to~$H_*(S^8;\Z)$ and $V$ is the vertex set of~$K$.
Let $g\in\Sym(K)$ be an element of prime order~$p$. Suppose that the group $G=\langle g\rangle$ acts on~$V$ with $r$ orbits. Then all $G$-orbits in~$V$ are simplices of~$K$ and one of the following assertions is true:
 \begin{itemize}
  \item $r=6$ and $K^{G}\cong \RP^2_6$,
  \item $r=9$ and $K^{G}\cong \CP^2_9$,
  \item $K^{G}\cong\pt\sqcup\partial\Delta^{r-2}$.
 \end{itemize}
\end{propos}

\begin{proof}
Let $V=\sigma_1\cup\cdots\cup\sigma_r$ be the decomposition into $G$-orbits.
Assume that one of the $G$-orbits (say, $\sigma_r$) is not a simplex of~$K$. Then by complementarity  $\sigma_1\cup\cdots\cup\sigma_{r-1}\in K$. Hence $K^G$ is the $(r-2)$-simplex with the vertices $b(\sigma_1),\ldots,b(\sigma_{r-1})$. If $r\ge 3$, then we get a contradiction with the assertion of Theorem~\ref{thm_Smith}, since a simplex of positive dimension is not a homology manifold. (Note that our definition allows homology manifolds without boundary only.) If $r\le 2$, then we also arrive at a contradiction, since the group~$\rC_p$ with prime~$p$ cannot act on $15$ points with less than three orbits. So all the orbits $\sigma_1,\ldots,\sigma_r$ are simplices of~$K$. Then $K^G$ has exactly $r$ vertices $b(\sigma_1),\ldots,b(\sigma_r)$. In particular,  $K^G$ is non-empty. Moreover, from the complementarity property for~$K$ immediately follows the complimentarity property for~$K^G$.

Let $F_1,\ldots,F_s$ be the connected components of~$K^G$ and $I_1,\ldots,I_s$ the corresponding subsets of $\{1,\ldots,r\}$ so that $I_k$ is the set of all indices~$i$ such that $b(\sigma_i)$ belongs to~$F_k$. Then $I_k\cup I_l$ is not a simplex of~$K$ unless  $k=l$. If $s\ge 4$, then we get a contradiction with complementarity, since both $I_1\cup I_2$ and~$I_3\cup I_4$ are non-simplices of~$K$. So we need to consider the three cases of $s$ equal to~$1$, $2$, and~$3$.

Assume that $s=1$. Then by Theorem~\ref{thm_Smith} we see that $K^G$ is a connected $\F_p$-homology $d$-manifold and hence a $d$-pseudomanifold for some $d<\dim K=8$ (see Remark~\ref{remark_smaller_dim}). But $K^G$ satisfies complementarity. So by Theorem~\ref{thm_Datta} we obtain that $K^G\cong\RP^2_6$, or~$K^G\cong\CP^2_9$, or $K^G$ is a $7$-dimensional pseudomanifold with odd Euler characteristic. However, the last case is impossible, since by Poincar\'e duality any odd-dimensional $\F_p$-homology manifold has zero Euler characteristic.

Assume that $s=2$. From complementarity it follows that exactly one of the two sets~$I_1$ and~$I_2$ is a simplex of~$K$. We may assume that $I_1\in K$ and $I_2\notin K$. Then $F_1$ is a simplex. However, by Theorem~\ref{thm_Smith} we have that $F_1$ is an $\F_p$-homology manifold. Hence, $F_1$ is a point. Therefore, $I_1$ is a single $G$-orbit. We may assume that $I_1=\sigma_1$. For every $k=2,\ldots,r$, we know that $\{b(\sigma_1),b(\sigma_k)\}$ is not an edge of~$K^G$. Hence $\sigma_1\cup\sigma_k\notin K$. By complementarity we obtain that $$\sigma_2\cup\cdots\cup\widehat{\sigma}_k\cup\cdots\cup\sigma_r\in K$$ and hence
$$\bigl\{b(\sigma_2),\ldots,\widehat{b(\sigma_k)},\ldots,b(\sigma_r)\bigr\}\in K^G,$$
where hat denotes the omission of the corresponding set or vertex. On the other hand, since $\sigma_1\in K$, we see that $\sigma_2\cup\cdots\cup\sigma_r\notin K$ and hence $\bigl\{b(\sigma_2),\ldots,b(\sigma_r)\bigr\}\notin K^G$. Therefore, $F_2$ is the boundary of the $(r-2)$-simplex with vertices $b(\sigma_2),\ldots,b(\sigma_r)$ and thus $K^G\cong\pt\sqcup\partial\Delta^{r-2}$.

Finally, assume that $s=3$. Then from complementarity it follows that $I_1$, $I_2$, and~$I_3$ are simplices of~$K$. Hence every component~$F_k$ is a simplex. However, by Theorem~\ref{thm_Smith} every $F_k$ is an $\F_p$-homology manifold. Therefore, $F_k$ is a point. It follows that $I_k$ is a single $G$-orbit. Thus, $r=s=3$ and $K^G\cong \pt\sqcup\pt\sqcup\pt=\pt\sqcup\partial\Delta^1$.
\end{proof}

In the following corollaries we always assume that $K$ satisfies the conditions from Proposition~\ref{propos_p} and $V$ is the vertex set of~$K$.

\begin{cor}\label{cor_p>7}
 The group $\Sym(K)$ contains no elements of prime orders greater than~$7$.
\end{cor}

\begin{cor}\label{cor_2}
Suppose that $g\in \Sym(K)$ is an element of order~$2$. Then $g$ acts on~$V$ with exactly three fixed points and $K^{\langle g\rangle}\cong \CP^2_9$.
\end{cor}
\begin{proof}
The subgroup $G=\langle g\rangle\cong\rC_2$ acts on~$V$ with $r\ge 8$ orbits. By Corollary~\ref{cor_complement} the complex~$K$ is $5$-neighborly, so the union of any two $G$-orbits is a simplex of~$K$. Hence $K^G$ is connected. Therefore, $r=9$ and $K^G\cong \CP^2_9$.
\end{proof}

\begin{cor}\label{cor_3}
Suppose that $g\in \Sym(K)$ is an element of order~$3$. Then $g$ acts on~$V$ without fixed points and $K^{\langle g\rangle}\cong \pt\sqcup\partial\Delta^3$.
\end{cor}
\begin{proof}
Consider the subgroup $G=\langle g\rangle\cong\rC_3$. It has $3s$ fixed points in~$V$, where $0\le s\le 4$. Then the number of $G$-orbits in~$V$ is $2s+5$, so $K^G$ has $2s+5$ vertices. By Proposition~\ref{propos_p} we  have two cases.

\textsl{Case 1: $s=2$ and $K^G\cong\CP^2_9$.} By Corollary~\ref{cor_complement} the complex~$K$ is $5$-neighborly, so any five of the six $G$-fixed points in~$V$ form a simplex of~$K$ and hence a simplex of~$K^G$. Therefore, $K^G\cong\CP^2_9$ contains a subcomplex isomorphic to~$\partial\Delta^5$. We arrive at a contradiction, so this case is impossible.

\textsl{Case 2: $K^G\cong\pt\sqcup\partial\Delta^{2s+3}$.} Let $\alpha$ be the $G$-orbit in~$V$ corresponding to the separate point~$\pt$. If $s>0$, then there is a fixed point $v\in V$ that corresponds to a vertex of~$\partial\Delta^{2s+3}$. Then the subset $\sigma=V\setminus(\alpha\cup\{v\})$ corresponds to the facet~$\tau$ of~$\Delta^{2s+3}$ opposite to the vertex~$v$. Since $\tau\in K^G$, we obtain that $\sigma\in K$. We arrive at a contradiciton, since $|\sigma|\ge 11$. Therefore $s=0$.
\end{proof}

\begin{cor}\label{cor_9}
The order $|\Sym(K)|$ is not divisible by~$9$.
\end{cor}

\begin{proof}
 From Corollary~\ref{cor_3} it follows that the Sylow $3$-subgroup of~$\Sym(K)$ acts freely on~$V$. Hence its order does not exceed~$3$.
\end{proof}

\begin{cor}\label{cor_5}
 Suppose that $g\in \Sym(K)$ is an element of order~$5$. Then one of the following two assertions is true:
 \begin{enumerate}
  \item $g$ acts on~$V$ without fixed points and $K^{\langle g\rangle}\cong \pt\sqcup\pt\sqcup\pt$,
  \item $g$ acts on~$V$ with exactly $5$ fixed points and $K^{\langle g\rangle}\cong \pt\sqcup\partial\Delta^5$, where the separate point~$\pt$ corresponds to one of the two $\langle g \rangle$-orbits of length~$5$.
 \end{enumerate}
\end{cor}

\begin{proof}
 Since the group $G=\langle g \rangle\cong\rC_5$ acts faithfully on~$V$, we see that it has $10$, or~$5$, or no fixed points in~$V$. Consider the three cases.

 First, assume that $G$ acts on~$V$ with $10$ fixed points. Then $K^G$ has $11$ vertices. Hence  by Proposition~\ref{propos_p} we obtain that $K^G\cong\pt\sqcup\partial\Delta^9$, which is impossible, since by Remark~\ref{remark_smaller_dim} $\dim K^G$ must be strictly smaller than $\dim K=8$.

 Second, assume that $G$ acts on~$V$ with $5$ fixed points. Then $K^G$ has $7$ vertices. Hence  by Proposition~\ref{propos_p} we obtain that $K^G\cong\pt\sqcup\partial\Delta^5$. Let $\alpha$ and~$\beta$ be the two $G$-orbits of length~$5$ in~$V$. Since $\dim K=8$, we have $\alpha\cup\beta\notin K$. Hence $\{b(\alpha),b(\beta)\}\notin K^G$. Therefore, either $b(\alpha)$ or $b(\beta)$ is the separate point~$\pt$ of~$K^G$.

 Finally, assume that $G$ acts on~$V$ without fixed points. Then $K^G$ has $3$ vertices. Hence  by Proposition~\ref{propos_p} we have $K^G\cong\pt\sqcup\pt\sqcup\pt$.
\end{proof}

\begin{remark}
 In fact, an element $g\in\Sym(K)$ of order $5$ always acts on~$V$ without fixed points, so assertion~(1) is always true and assertion~(2) is never true. However, we are not ready to prove this at the moment. We will prove this in Subsection~\ref{subsection_5_ns}.
\end{remark}

\begin{cor}\label{cor_7}
Suppose that $g\in \Sym(K)$ is an element of order~$7$. Then $g$ acts on~$V$ with exactly one fixed point and $K^{\langle g\rangle}\cong \pt\sqcup\pt\sqcup\pt$.
\end{cor}

\begin{proof}
 We need only  to show that the group $G=\langle g\rangle\cong\rC_7$ cannot act on~$V$ with $8$ fixed point. Assume that this is the case. Let $\alpha$ be the $G$-orbit of length~$7$. Then $K^G$ has $9$ vertices, the barycentre~$b(\alpha)$ and the eight $G$-fixed vertices in~$V$.
By Corollary~\ref{cor_complement} the complex~$K$ is $5$-neighborly, so every two $G$-fixed vertices are connected by an edge in~$K$ and hence in~$K^G$. Further, $\alpha$ is a $6$-simplex of~$K$. Hence $\alpha$ is contained in a $7$-simplex of~$K$, that is, $\alpha\cup\{v\}\in K$ for some $G$-fixed vertex~$v$. Therefore, $b(\alpha)$ is connected with~$v$ by an edge in~$K^G$. Thus, $K^G$ is connected. By Proposition~\ref{propos_p} we obtain that $K^G\cong\CP^2_9$. Then the vertex $b(\alpha)$ of~$K^G$ is contained in a $4$-simplex of~$K^G$. Hence $\sigma=\alpha\cup\{v_1,\ldots,v_4\}$ is a simplex of~$K$ for certain $G$-fixed vertices~$v_1,\ldots,v_4$. We arrive at a contradiction, since $|\sigma|=11$.
\end{proof}

\subsection{K\"uhnel's $\CP^2_9$}\label{subsection_K}

Since K\"uhnel's triangulation~$\CP^2_9$ arises in Corollary~\ref{cor_2}, we will need an explicit description of it. The following nice construction of $\CP^2_9$ and its symmetry group is due to Bagchi and Datta, see~\cite{BaDa94}.

Let $\mathcal{P}$ be an affine plane over the field~$\F_3$. Then $|\mathcal{P}|=9$. Fix a decomposition
$$
\mathcal{P}=\ell_0\cup\ell_1\cup\ell_2,
$$
where $\ell_0$, $\ell_1$, and~$\ell_2$ are mutually parallel lines, and fix a cyclic order of these three lines. We conveniently consider the indices~$0$, $1$, and~$2$ as elements of~$\F_3$. The three lines $\ell_0$, $\ell_1$, and~$\ell_2$ will be called \textit{special}, and all other lines in~$\mathcal{P}$ \textit{non-special}.

The affine plane~$\mathcal{P}$ will serve as the set of vertices of~$\CP^2_9$. The $4$-simplices of~$\CP^2_9$ are exactly the following $36$ five-element subsets of~$\mathcal{P}$:
\begin{enumerate}
 \item The $27$ subsets of the form $m_1\cup m_2$, where $\{m_1,m_2\}$ is a pair of intersecting non-special lines in~$\mathcal{P}$.
 \item The $9$ subsets of the form $\ell_t\cup(\ell_{t+1}\setminus\{v\})$, where $t\in\F_3$ and $v\in\ell_{t+1}$.
\end{enumerate}

The simplicial complex~$\CP^2_9$ consisting of these $36$ four-simplices and all their faces is a combinatorial manifold PL homeomorphic to the complex projective plane~$\CP^2$. Moreover, $\CP^2_9$ satisfies complementarity.

Introduce affine coordinates~$(x,y)$ in~$\mathcal{P}$ so that the special lines are the lines~$\{y=0\}$, $\{y=1\}$, and~$\{y=2\}$ with this cyclic order. Then the symmetry group~$\Sym(\CP^2_9)$ has order~$54$ and consists of all affine transformations of the form
\begin{equation*}
\begin{pmatrix}
 x\\ y
\end{pmatrix}
\mapsto
\begin{pmatrix*}[r]
 \pm 1 & c\\
 0 & 1
\end{pmatrix*}
\begin{pmatrix}
 x\\ y
\end{pmatrix}+
\begin{pmatrix}
 a\\ b
\end{pmatrix},\qquad a,b,c\in\F_3.
\end{equation*}

We will also need the following proposition, see~\cite[Proposition~2.19]{Gai23}.

\begin{propos}\label{propos_3subset}
 Suppose that $m$ is a $3$-element subset of~$\mathcal{P}$. Then the following two assertions are equivalent to each other:
 \begin{enumerate}
  \item all $4$-element subsets of~$\mathcal{P}$ that contain~$m$ are simplices of~$\CP^2_9$,
  \item $m$ is a non-special line.
 \end{enumerate}
\end{propos}

\section{Triangulations with symmetry groups~$G$ of orders $|G|\ge 7$}\label{section_>=7}

In~\cite{Gai22} the author suggested an algorithm that, being given numbers $d$, $n$, and~$N$ and a subgroup $G\subset \rS_n$, produces the list of all $G$-invariant weak $d$-pseudomanifolds~$K$ with $n$ vertices and at least $N$ maximal simplices such that $K$ satisfies the following condition:
\begin{itemize}
 \item[$(*)$] for each subset $\sigma\subset V$, at most one of the two subsets $\sigma$ and~$V\setminus\sigma$ is a simplex of~$K$.
\end{itemize}
(We will refer to this condition as to condition~$(*)$ throughout the whole paper.)

This algorithm was implemented as a C++ program \texttt{find}, see~\cite{Gai-prog}. It completes in a reasonable amount of time, provided that the symmetry group~$G$ is large enough.  In~\cite{Gai22} this program was used to find $27$-vertex combinatorial $16$-manifolds that are not homeomorphic to~$S^{16}$. For this purpose, the program was launched for
$$
(d,n,N)=(16,27,100386)
$$
and different subgroups $G\subset\rS_{27}$. (The number of $100386$ top-dimensional simplices is provided by a $16$-dimensional analogue of Corollary~\ref{cor_fvect}.) It turned out that the program works rather well for $|G|\gtrsim 250$, but the running time grows uncontrollably when $|G|$ decreases.

In the present paper we are interested in the $8$-dimensional case. By Corollaries~\ref{cor_complement} and~\ref{cor_fvect} we know that any $15$-vertex $\Z$-homology $8$-manifold $K$ with $H_*(K;\Z)\ncong H_*(S^8;\Z)$ satisfies complementarity and, a fortiory, condition~$(*)$, and has exactly $490$ top-dimen\-sional simplices. So we would like to launch the program \texttt{find} for
$$
(d,n,N)=(8,15,490)
$$
and different subgroups $G\subset\rS_{15}$. It turns out that it completes in a reasonable time, provided that $|G|\ge 7$. The largest running time (about $46$ hours on one processor core of clock frequency 3.3~GHz) is achieved when $G$ is the subgroup~$\rC_7\subset \rS_{15}$ generated by the permutation
\begin{equation*}
   (1\ 2\ 3\ 4\ 5\ 6\ 7)(8\ 9\ 10\ 11\ 12\ 13\ 14).
\end{equation*}

Let us describe how we further process the results of the program \texttt{find}, using other programs from~\cite{Gai-prog}. Suppose that, being launched for $(d,n,N)=(8,15,490)$ and a subgroup $G\subset \rS_{15}$, the program \texttt{find} has produced a list $X_1,\ldots,X_r$ of weak $8$-pseudomanifolds. Then we proceed as follows.

\begin{enumerate}
 \item  We use the program \texttt{isomorphism\_groups} to divide the obtained list  into groups of isomorphic weak pseudomanifolds. As a result, we arrive at a sublist $Y_1,\ldots,Y_s$ consisting of pairwise non-isomorphic weak pseudomanifolds. The reason why the initial list  typically contains several isomorphic copies of each weak pseudomanifold is as follows. Let $N=N_{\rS_{15}}(G)$ be the normalizer of~$G$ in~$\rS_{15}$. Then the quotient group~$N/G$ acts on the set of $G$-invariant simplicial complexes. It follows that each pseudomanifold~$X$ with $\Sym(X)=G$ is multiplied by~$N/G$ into exactly $|N:G|$ isomorphic copies. (If $\Sym(X)$ is strictly greater than~$G$, then the number of isomorphic copies of~$X$ can be different.)
 \item For every~$i=1,\ldots,s$, we use the programs \texttt{check} and \texttt{fvect} to check that $Y_i$ is a combinatorial $8$-manifold and the $f$-vector of~$Y_i$ coincides with the expected $f$-vector~\eqref{eq_fvect_BK}. (Both assertions will indeed be true in all the cases that we will consider.) The latter assertion implies that $\chi(Y_i)=3$. Hence $Y_i$ is not homeomorphic to~$S^8$. By Theorem~\ref{thm_BK} we obtain that $Y_i$ is a manifold like the quaternionic projective plane. For a description of the algorithm implemented in the program \texttt{check}, see~\cite[Section~5]{Gai22}.
 \item For every~$i=1,\ldots,s$, we use the program \texttt{symm\_group} to compute the symmetry group~$\Sym(Y_i)$. Certainly, this group always contains~$G$. Nevertheless, it may happen that it is in fact strictly larger than~$G$.
\end{enumerate}

\begin{remark}
 In this section we do not study the question of whether the obtained $15$-vertex combinatorial $8$-manifolds like the quaternionic projective plane are PL homeomorphic to~$\HP^2$. (In Section~\ref{section_p1} we will see that indeed they are.) Nevertheless, it will be convenient to us to denote the obtained combinatorial manifolds by~$\HP^2_{15}(G,k)$,  where $G$ is the symmetry group, and $k$ is the ordinal number of the combinatorial manifold with such a symmetry group.
\end{remark}

The results of calculation for certain subgroups $G\subset\rS_{15}$ are as follows.

\subsection*{The group $\rA_5$ acting transitively} All subgroups of index~$15$ in~$\rA_5$ are isomorphic to~$\rC_2\times\rC_2$ and are conjugate to each other. Hence $\rA_5$ admits a unique up to conjugation realization as a transitive subgroup of~$\rS_{15}$. For this subgroup, the program \texttt{find} produces a list of $6$ simplicial complexes. All of them are isomorphic to each other and to the Brehm--K\"uhnel triangulation~$M^8_{15}$ with symmetry group~$\rA_5$. Further we denote this triangulation by~$\HP^2_{15}(\rA_5)$. (The fact that we get $6$ isomorphic copies of~$\HP^2_{15}(\rA_5)$ corresponds to the fact that the subgroup $\rA_5\subset\rS_{15}$ has index $6$ in the normalizer of it.)
The uniqueness of a vertex-transitive $15$-vertex combinatorial $8$-manifold not homeomorphic to the sphere was proved by Brehm (unpublished, see~\cite[Corollary~11]{KoLu05}). So in this case we do not obtain any new result.

\subsection*{The group $\rA_4$ acting with two orbits of lengths~$12$ and~$3$} Since the group $\rA_4$ contains a unique subgroup of index~$3$, it follows easily that such a subgroup $\rA_4\subset \rS_{15}$ is unique up to conjugation.
The program \texttt{find} produces a list of $216$ simplicial complexes, which are divided into $3$ groups, each consisting of $72$ isomorphic complexes. All of them are combinatorial $8$-manifolds like the quaternionic projective plane. One of the three combinatorial manifolds has a larger symmetry group $\rA_5\supset\rA_4$ and is isomorphic to~$\HP^2_{15}(\rA_5)$. The other two combinatorial manifolds have symmetry group~$\rA_4$. One of them is the Brehm--K\"uhnel triangulation~$\widetilde{M}\vphantom{M}^8_{15}$; we further denote it by~$\HP^2_{15}(\rA_4,1)$. The other one is new; we denote it by~$\HP^2_{15}(\rA_4,2)$.

\subsection*{The group $\rC_6\times\rC_2$ acting with two orbits of lengths~$12$ and~$3$} Since the group $\rC_6\times\rC_2$ contains a unique subgroup of index~$3$, it follows easily that such a subgroup $\rC_6\times\rC_2\subset \rS_{15}$ is unique up to conjugation. The program \texttt{find} produces a list of $36$ simplicial complexes; they are  all isomorphic to each other. The obtained complex, which we denote by~$\HP^2_{15}(\rC_6\times\rC_2)$, is a new $15$-vertex combinatorial $8$-manifold like the quaternionic projective plane; its symmetry group is exactly $\rC_6\times\rC_2$.

\subsection*{The group $\rC_7$ acting with one fixed point} All such subgroups in~$\rS_{15}$ are conjugate to each other, so we may assume that $\rC_7$ is generated by the permutation
\begin{equation*}
   (1\ 2\ 3\ 4\ 5\ 6\ 7)(8\ 9\ 10\ 11\ 12\ 13\ 14).
\end{equation*}
The program \texttt{find} produces a list of $1596$ simplicial complexes, which are divided into $19$ groups, each consisting of $84$ isomorphic complexes. We denote them by $$\HP^2_{15}(\rC_7,1),\ldots,\HP^2_{15}(\rC_7,19).$$ All of them are combinatorial manifolds like the quaternionic projective plane. Besides, the symmetry group of each~$\HP^2_{15}(\rC_7,k)$ is exactly~$\rC_7$.

\section{Triangulations with symmetry groups~$G$ of orders~$4$, $5$, and~$6$}\label{section_6}

On groups~$G$ of order less than $7$, the program \texttt{find} does not seem to finish running in a reasonable time, at least on a regular PC. Possibly, we could still try to perform the calculation by using a high-performance computer.
However, we prefer to reduce the complexity of the computational problem by proving that every $G$-invariant $15$-vertex $\Z$-homology $8$-manifold~$K$ with $H_*(K;\Z)\ncong H_*(S^8;\Z)$ contains certain pure $8$-dimensional subcomplex~$R$, which we call a \textit{mandatory subcomplex}. (A simplicial complex is said to be \textit{pure $d$-dimensional} if every its simplex is contained in a $d$-simplex.) The subcomplex~$R$ and the reason why it is contained in~$K$ will be different for different groups~$G$.
Once such a mandatory subcomplex~$R$ is found, we launch the same program \texttt{find} with the restriction that we force all top-dimensional simplices of~$R$ into~$K$ from the very beginning.
This significantly reduces the running time, since including a simplex in~$K$ immediately prohibits including a lot of other simplices in~$K$, see~\cite[Remark~4.3]{Gai22}. Being implemented in this way, the program completes in a reasonable amount of time in all interesting for us cases. The largest running time (about $49$ hours on one processor core of clock frequency 3.3~GHz) is achieved at $G=\rS_3$.

In this section we will consider the groups~$\rS_3$, $\rC_6$, $\rC_4$, $\rC_2\times\rC_2$, and the group~$\rC_5$ acting on the vertex set~$V$ with $5$ fixed points. In these cases, we will be able to produce a mandatory subcomplex~$R$ by applying Corollary~\ref{cor_2}, or Corollary~\ref{cor_3}, or Corollary~\ref{cor_5}.

Unfortunately, in the remaining case of the group~$\rC_5$ acting on~$V$ without fixed points, the results from Section~\ref{section_finite_order} do not help us to find a mandatory subcomplex. In this case the set~$V$ is the union of three $\rC_5$-orbits, which we denote by~$\alpha_1$, $\alpha_2$, and~$\alpha_3$. From the complimentarity property it follows that $K$ is $5$-neighborly. Hence, the orbits~$\alpha_1$, $\alpha_2$, and~$\alpha_3$ are $4$-simplices of~$K$. The search for a mandatory subcomplex in this case will be based on studying the structure of the links of the simplices~$\alpha_1$, $\alpha_2$, and~$\alpha_3$. Since  $K$ is a $\Z$-homology $8$-manifold, we obtain that each simplicial complex~$\link(\alpha_i,K)$ is a $\Z$-homology $3$-sphere. Moreover, since the simplex~$\alpha_i$ is $\rC_5$-invariant, the complex~$\link(\alpha_i,K)$ is also $\rC_5$-invariant, which in particular implies that it has either $5$ or $10$ vertices. In the former case, $\link(\alpha_i,K)$ is the boundary of a $4$-simplex, while the latter case is more complicated. In Section~\ref{section_3spheres} we will obtain an auxiliary classification result for $\rC_5$-invariant $10$-vertex homology $3$-spheres. In Section~\ref{section_5}, using this results, we will complete the consideration of the case of symmetry group~$\rC_5$ acting on~$V$ without fixed points.

\subsection{The group $\rS_3$}

\begin{propos}\label{propos_S3_aux}
Let $K$ be a $15$-vertex $\Z$-homology $8$-manifold such that $H_*(K;\Z)$ is not isomorphic to~$H_*(S^8;\Z)$. Suppose that the symmetry group~$\Sym(K)$ contains a subgroup~$G$ isomorphic to~$\rS_3$. Then we can number the vertices of~$K$ so that to achieve that
 \begin{itemize}
  \item $G$ is generated by the permutations
\begin{equation}\label{eq_S3_gen}
\begin{aligned}
 g&=(1\ 2\ 3)(4\ 5\ 6)(7\ 8\ 9)(10\ 11\ 12)(13\ 14\ 15),\\
 h&=(1\ 4)(2\ 6)(3\ 5)(7\ 8)(10\ 11)(13\ 14).
 \end{aligned}
\end{equation}
  \item $K$ contains the four $8$-simplices
 \begin{equation}\label{eq_S3_Z}
 \begin{aligned}
  &\{1,2,3,4,5,6,7,8,9\},
  &&\{1,2,3,4,5,6,10,11,12\},\\
  &\{1,2,3,7,8,9,10,11,12\},
  &&\{4,5,6,7,8,9,10,11,12\}.
 \end{aligned}
 \end{equation}
 \end{itemize}
\end{propos}

\begin{proof}
Let $V$ be the vertex set of~$K$. From Corollary~\ref{cor_3} it follows that the subgroup $\rC_3\subset G$ acts on~$V$ without fixed points. Hence, every $G$-orbit in~$V$ has length either~$3$ or~$6$. Each element of~$G$ of order~$2$ fixes exactly one vertex in every $G$-orbit of length~$3$ and fixes no vertices in every $G$-orbit of length~$6$. So from Corollary~\ref{cor_2} it follows that $G$ acts on~$V$ with one orbit of length~$6$ and three orbits of length~$3$. Up to isomorphism, the group~$\rS_3$ has a unique transitive action on $6$ points and a unique transitive action on $3$ points. It follows that all subgroups of~$\rS_{15}$ that are isomorphic to~$\rS_3$ and have one orbit of length~$6$ and three orbits of length~$3$ are conjugate to each other. Therefore, we can number the vertices of~$K$ so that to achieve that $G$ is generated by the permutations~$g$ and~$h$ given by~\eqref{eq_S3_gen}.

 By Corollary~\ref{cor_3} we see that  $K^{\langle g\rangle}\cong\pt\sqcup\partial\Delta^3$. The vertices of~$K^{\langle g\rangle}$ are the $5$ barycentres $b(\{1,2,3\})$, $b(\{4,5,6\})$, $b(\{7,8,9\})$, $b(\{10,11,12\})$, and~$b(\{13,14,15\})$. Since $\langle g\rangle$ is a normal subgroup of~$G$, the action of~$G$ on~$K$ induces the action of the quotient group~$G/\langle g\rangle\cong\rC_2$ on~$K^{\langle g\rangle}$. The latter action swaps  $b(\{1,2,3\})$ with~$b(\{4,5,6\})$ and fixes every of the three barycentres~$b(\{7,8,9\})$, $b(\{10,11,12\})$, and~$b(\{13,14,15\})$. Hence the separate point~$\pt$ of~$K^{\langle g\rangle}$ is neither $b(\{1,2,3\})$ nor~$b(\{4,5,6\})$. Renumbering the vertices of~$K$, we can achieve that $\pt=b(\{13,14,15\})$. Therefore, any three of the four barycentres $b(\{1,2,3\})$, $b(\{4,5,6\})$, $b(\{7,8,9\})$, and~$b(\{10,11,12\})$ form a simplex of~$K^{\langle g\rangle}$. Thus, the four sets~\eqref{eq_S3_Z} are simplices of~$K$.
\end{proof}

\begin{comp}
 Being run on the subgroup~$\rS_3\subset\rS_{15}$ generated by~$g$ and~$h$ and the mandatory subcomplex~$R$ consisting of the four simplices~\eqref{eq_S3_Z}, the program \texttt{find} produces a list of $156$ simplicial complexes, which are divided into $13$ groups, each consisting of $12$ isomorphic complexes. All the $13$ obtained simplicial complexes are combinatorial $8$-manifolds not homeomorphic to~$S^8$ and hence (by Theorem~\ref{thm_BK}) are manifolds like the quaternionic projective plane. One of them has in fact a larger symmetry group $\rA_5\supset\rS_3$ and is isomorphic to~$\HP^2_{15}(\rA_5)$. The other $12$ combinatorial manifolds have symmetry group exactly~$\rS_3$; we denote them by $\HP^2_{15}(\rS_3,1),\ldots,\HP^2_{15}(\rS_3,12)$. One of these $12$ combinatorial manifolds (we choose it to be $\HP^2_{15}(\rS_3,1)$) is isomorphic to the Brehm--K\"uhnel triangulation~$\widetilde{\widetilde{M}}\vphantom{M}_{15}^8$; the other $11$ triangulations are new. Thus, we obtain the following proposition.
\end{comp}

\begin{propos}\label{propos_S3}
 Up to isomorphism, there exist exactly twelve $15$-vertex $\Z$-homology $8$-manifolds $K=\HP^2_{15}(\rS_3,1),\ldots,\HP^2_{15}(\rS_3,12)$ such that $H_*(K;\Z)$ is not isomorphic to~$H_*(S^8;\Z)$ and $\Sym(K)\cong\rS_3$. All of them are combinatorial manifolds.
\end{propos}

 \begin{remark}
  Similarly to studying the fixed point complex~$K^{\langle g\rangle}$ for the element~$g$ of order~$3$, we could study the fixed point complex~$K^{\langle h\rangle}$ for the element~$h$ of order~$2$. By Corollary~\ref{cor_2} we have $K^{\langle h\rangle}\cong\CP^2_9$, so this approach seems to provide even a larger subcomplex~$R$. However, this approach would require consideration of several cases corresponding to different isomorphisms $K^{\langle h\rangle}\to\CP^2_9$. Since the subcomplex~$R$ given by~\eqref{eq_S3_Z} is already sufficient for the program to complete in a reasonable time, we prefer not to deal with the fixed point complex~$K^{\langle h\rangle}$.
 \end{remark}

\subsection{The group $\rC_6$}

\begin{propos}\label{propos_C6_aux}
 Let $K$ be a $15$-vertex $\Z$-homology $8$-manifold such that $H_*(K;\Z)$ is not isomorphic to~$H_*(S^8;\Z)$. Suppose that the symmetry group~$\Sym(K)$ contains a subgroup~$G$ isomorphic to~$\rC_6$. Then we can number the vertices of~$K$ so that to achieve that
 \begin{itemize}
  \item $G$ is generated by the permutation
\begin{equation}\label{eq_C6_f}
 f=(1\ 2\ 3\ 4\ 5\ 6)(7\ 8\ 9\ 10\ 11\ 12)(13\ 14\ 15).
\end{equation}
  \item $K$ contains the ten $8$-simplices
 \begin{equation}\label{eq_C6_Z}
 \begin{aligned}
  &\{1,2,3,4,5,6,7,9,11\},
  &&\{1,2,3,4,5,6,8,10,12\},\\
  &\{1,3,5,7,8,9,10,11,12\},
  &&\{2,4,6,7,8,9,10,11,12\},\\
  &\{1,2,4,5,7,8,10,11,13\},
  &&\{2,3,5,6,8,9,11,12,13\},\\
  &\{1,2,3,4,5,6,8,11,13\},
  &&\{1,2,3,4,5,6,9,12,13\},\\
  &\{2,5,7,8,9,10,11,12,13\},
  &&\{3,6,7,8,9,10,11,12,13\}.
 \end{aligned}
 \end{equation}
 \end{itemize}
\end{propos}

\begin{proof}
Let $V$ be the vertex set of~$K$.
From Corollaries~\ref{cor_2} and~\ref{cor_3} it follows that $G$ acts on~$V$ with two orbits of length~$6$ and one orbit of length~$3$. So we can number the vertices of~$K$ so that to achieve that $G$ is generated by the permutation~$f$ given by~\eqref{eq_C6_f}.
We put
\begin{align*}
g&=f^2=(1\ 3\ 5)(2\ 4\ 6)(7\ 9\ 11)(8\ 10\ 12)(13\ 15\ 14),\\
h&=f^3=(1\ 4)(2\ 5)(3\ 6)(7\ 10)(8\ 11)(9\ 12).
\end{align*}

First, literally as in the proof of Proposition~\ref{propos_S3_aux} we see that $K^{\langle g\rangle}\cong\pt\sqcup\partial\Delta^3$, the vertices of~$K^{\langle g\rangle}$ are the $5$ barycentres $b(\{1,3,5\})$, $b(\{2,4,6\})$, $b(\{7,9,11\})$, $b(\{8,10,12\})$, and~$b(\{13,14,15\})$.  Moreover, $b(\{13,14,15\})$ is the separate point~$\pt$ of~$K^{\langle g\rangle}$, since it is the only of the five barycentres that is fixed by~$f$. Therefore, any three of the four barycentres $b(\{1,3,5\})$, $b(\{2,4,6\})$, $b(\{7,9,11\})$, and~$b(\{8,10,12\})$ form a simplex of~$K^{\langle g\rangle}$. Thus, the first four of the ten sets~\eqref{eq_C6_Z} are simplices of~$K$.

Second, by Corollary~\ref{cor_2} we have $K^{\langle h\rangle}\cong\CP^2_9$. The $9$ vertices of~$K^{\langle h\rangle}$ are the $6$ barycentres $b(\{1,4\})$, $b(\{2,5\})$, $b(\{3,6\})$, $b(\{7,10\})$, $b(\{8,11\})$, and~$b(\{9,12\})$ and the three $h$-fixed vertices~$13$, $14$, and~$15$.
We identify $K^{\langle h\rangle}$ with $\CP^2_9$ along some isomorphism and identify the vertex set of~$\CP^2_9$ with the affine plane~$\mathcal{P}$ over~$\F_3$ as in Subsection~\ref{subsection_K}. These identifications provide a bijection $\varphi\colon V/\langle h\rangle\to \mathcal{P}$. We denote by $m\subset  \mathcal{P}$ the $3$-element subset consisting of the points that correspond to the three $h$-fixed vertices~$13$, $14$, and~$15$ of~$K$. Every $4$-element subset~$\sigma\subset \mathcal{P}$ that contains~$m$ corresponds to a $5$-element subset of~$V$, which is a simplex of~$K$, since by Corollary~\ref{cor_complement} the complex~$K$ is $5$-neighborly. Hence $\sigma$ is a simplex of~$\CP^2_9$. From Proposition~\ref{propos_3subset} it follows that $m$ is a non-special line.

Now, we can choose affine coordinates~$(x,y)$ in~$\mathcal{P}$ so that:
\begin{itemize}
 \item $m$ coincides with the line $\{x=0\}$,
 \item the special lines in~$\mathcal{P}$ are $\{y=0\}$, $\{y=1\}$, and $\{y=2\}$ with this cyclic order.
\end{itemize}

The automorphism~$f$ of~$K$ commutes with $h$ and hence induces an automorphism of~$K^{\langle h\rangle}\cong\CP^2_9$, which acts on vertices in the following way:
\begin{gather*}
 b(\{1,4\})\mapsto b(\{2,5\})\mapsto b(\{3,6\})\mapsto b(\{1,4\}),\\
 b(\{7,10\})\mapsto b(\{8,11\})\mapsto b(\{9,12\})\mapsto b(\{7,10\}),\\
 13\mapsto 14\mapsto 15\mapsto 13.
\end{gather*}
So the corresponding permutation~$\theta$ of points of~$\mathcal{P}$ provides an automorphism of~$\CP^2_9$. Hence $\theta$ is an affine transformation of~$\mathcal{P}$ that takes special lines to special lines preserving their cyclic order. Moreover, $\theta^3=\mathrm{id}$, the line $\{x=0\}$ is invariant under~$\theta$, and $\theta$ has no fixed points in~$\mathcal{P}$. It follows easily that
$$
\theta(x,y)=(x,y\pm 1)
$$
with the same sign~$\pm$ for all $(x,y)\in\mathcal{P}$. Hence the three triples
$$
\bigl(b(\{1,4\}),b(\{2,5\}),b(\{3,6\})\bigr),\qquad \bigl(b(\{7,10\}),b(\{8,11\}),b(\{9,12\})\bigr),\qquad (13,14,15)
$$
correspond to three non-special parallel lines in~$\mathcal{P}$ and the cyclic orders of points in these three triples are either all positive or all negative.

Note that the property that the group~$G$ is generated by permutation~\eqref{eq_C6_f} will not be violated if we renumber the vertices of~$K$ using any of the following five permutations:
\begin{gather*}
 (1\ 2\ 3\ 4\ 5\ 6),\qquad
 (7\ 8\ 9\ 10\ 11\ 12),\qquad
 (13\ 14\ 15),\\
 (1\ 7)(2\ 8)(3\ 9)(4\ 10)(5\ 11)(6\ 12),\qquad (2\ 6)(3\ 5)(8\ 12)(9\ 11)(14\ 15).
\end{gather*}
Renumbering the vertices using these five permutations, we can achieve that  the bijection $\varphi\colon V/\langle h\rangle\to\mathcal{P}$ is given by Table~\ref{table_bijection_C6}.
\begin{table}
 \caption{The bijection between points $(x,y)\in\mathcal{P}$ and $\langle h\rangle $-orbits in~$V$}\label{table_bijection_C6}
 \begin{tabular}{|c|ccc|}
 \hline
 \diaghead(1,-1){abc}{$y$}{$x$} & $0$ & $1$ & $2$\\
 \hline
 $0$ & $\{13\}\vphantom{\Bigl(}$ & $\{1,4\}$  & $\{7,10\}$\\
 $1$ & $\{14\}\vphantom{\Bigl(}$ & $\{2,5\}$  & $\{8,11\}$\\
 $2$ & $\{15\}\vphantom{\Bigl(}$ & $\{3,6\}$ & $\{9,12\}$\\
 \hline
 \end{tabular}
\end{table}

By the construction of the simplicial complex~$\CP^2_9$, the following $6$ five-element subsets of~$\mathcal{P}$ are $4$-simplices of~$\CP^2_9$:
\begin{align*}
 &\bigl\{(0,0),(1,0),(1,1),(2,0),(2,1)\bigr\},&
 &\bigl\{(0,0),(1,1),(1,2),(2,1),(2,2)\bigr\},\\
 &\bigl\{(0,0),(1,0),(1,1),(1,2),(2,1)\bigr\},&
 &\bigl\{(0,0),(1,0),(1,1),(1,2),(2,2)\bigr\},\\
 &\bigl\{(0,0),(1,1),(2,0),(2,1),(2,2)\bigr\},&
 &\bigl\{(0,0),(1,2),(2,0),(2,1),(2,2)\bigr\}.
\end{align*}
So the corresponding $6$ subsets of~$V/\langle h\rangle$ are simplices of~$K^{\langle h\rangle}$. Hence, the preimages of these $6$ subsets under the projection $V\to V/\langle h\rangle$ are simplices of~$K$. It is easy to check that they are exactly the last $6$ simplices in~\eqref{eq_C6_Z}.
\end{proof}

\begin{comp}
 Being run on the subgroup~$G=\rC_6\subset\rS_{15}$ generated by~$f$ and the mandatory subcomplex~$R$ consisting of the ten simplices~\eqref{eq_C6_Z}, the program \texttt{find} produces a list of $62$ simplicial complexes. These complexes are divided into $15$ groups, one of which consists of $6$, and each of the others of $4$ isomorphic complexes.   All the $15$ obtained simplicial complexes are combinatorial $8$-manifolds not homeomorphic to~$S^8$ and hence (by Theorem~\ref{thm_BK}) are manifolds like the quaternionic projective plane. One of them (those that belongs to the group consisting of $6$ isomorphic complexes) has in fact a larger symmetry group $\rC_6\times\rC_2$ and is isomorphic to~$\HP^2_{15}(\rC_6\times\rC_2)$. The other $14$ combinatorial manifolds have symmetry group exactly~$\rC_6$; we denote them by $\HP^2_{15}(\rC_6,1),\ldots,\HP^2_{15}(\rC_6,14)$. Thus, we obtain the following proposition.
\end{comp}

\begin{propos}\label{propos_C6}
 Up to isomorphism, there exist exactly fourteen $15$-vertex $\Z$-homology $8$-manifolds $K=\HP^2_{15}(\rC_6,1),\ldots,\HP^2_{15}(\rC_6,14)$ such that $H_*(K;\Z)$ is not isomorphic to~$H_*(S^8;\Z)$ and $\Sym(K)\cong\rC_6$. All of them are combinatorial manifolds.
\end{propos}

\subsection{The group~$\rC_5$ acting with $5$ fixed points}\label{subsection_5_ns}

We denote by $L_1*L_2$ the join of simplicial complexes~$L_1$ and~$L_2$.

\begin{propos}\label{propos_5_ns_aux}
 Let $K$ be a $15$-vertex $\Z$-homology $8$-manifold on vertex set~$V$ with $H_*(K;\Z)\not\cong H_*(S^8;\Z)$. Suppose that the symmetry group~$\Sym(K)$ contains a subgroup~$G$ isomorphic to~$\rC_5$ such that $G$ acts on~$V$ with $5$ fixed points. Then we can number the vertices of~$K$ so that to achieve that
 \begin{itemize}
  \item $G$ is generated by the permutation
\begin{equation}\label{eq_5_ns}
g=(1\ 2\ 3\ 4\ 5)(6\ 7\ 8\ 9\ 10).
\end{equation}
 \item $K$ contains the subcomplex
 \begin{equation}\label{eq_subcomp_5_ns}
 (\Delta_1*\partial\Delta_2)\cup(\Delta_2*\partial\Delta_3)\cup(\Delta_3*\partial\Delta_1),
 \end{equation}
 where
 \begin{equation}
 \Delta_1= \{1,2,3,4,5\},\quad
 \Delta_2= \{6,7,8,9,10\},\quad
 \Delta_3= \{11,12,13,14,15\},\label{eq_5_ns_Delta}
 \end{equation}
 and also the simplex
 \begin{equation}
  \label{eq_simp_5_ns}
  \{1,2,3,4,6,7,8,9,11\}.
 \end{equation}
 \end{itemize}
\end{propos}

\begin{proof}
Let $\Delta_1$ and $\Delta_2$ be the two $G$-orbits of length~$5$ in~$V$ and $\Delta_3$ the set of $G$-fixed points. By Corollary~\ref{cor_5} we have that $K^G\cong\pt\sqcup\partial\Delta^5$ and the separate point~$\pt$ of~$K^G$ is either~$b(\Delta_1)$ or~$b(\Delta_2)$. Swapping~$\Delta_1$ and~$\Delta_2$, we may achieve that $b(\Delta_1)=\pt$ and $b(\Delta_2)$ is a vertex of~$\partial\Delta^5$. Since the other $5$ vertices of~$\partial\Delta^5$ are the $5$ fixed points, we see that $\{b(\Delta_2)\}\cup(\Delta_3\setminus\{v\})\in K^G$ and hence $\Delta_2\cup(\Delta_3\setminus\{v\})\in K$ for all $v\in\Delta_3$. Therefore, the join $\Delta_2*\partial\Delta_3$ is a subcomplex of~$K$. Let us prove that $\Delta_1*\partial\Delta_2$ is also a subcomplex of~$K$. Indeed, since $\Delta_2\cup(\Delta_3\setminus\{v\})\in K$ for all $v\in\Delta_3$, we by complimentarity see that $\Delta_1\cup\{v\}\notin K$ for all $v\in\Delta_3$. Hence $\link(\Delta_1,K)$ is contained in~$\Delta_2$. However, $\link(\Delta_1,K)$ has the homology of a $3$-sphere. Therefore, $\link(\Delta_1,K)=\partial\Delta_2$ and thus $\Delta_1*\partial\Delta_2$ is a subcomplex of~$K$. Repeating the same argument, we obtain that $\Delta_3*\partial\Delta_1$ is also a subcomplex of~$K$. We can number the vertices of~$K$ so that to achieve that $G$ is generated by the permutation~$g$ given by~\eqref{eq_5_ns} and the simplices~$\Delta_1$, $\Delta_2$, and~$\Delta_3$ are given by~\eqref{eq_5_ns_Delta}. Finally, consider the $7$-simplex
$$
\rho=(\Delta_1\setminus\{5\})\cup(\Delta_2\setminus\{10\})=\{1,2,3,4,6,7,8,9\}.
$$
Since $\rho$ is contained in $\Delta_1*\partial\Delta_2$, we see that $\rho\in K$. But $K$ is a weak $8$-pseudomanifold. Hence $\rho$ is contained in exactly two $8$-simplices of~$K$. One of these two $8$-simplices is $\sigma=\Delta_1\cup(\Delta_2\setminus\{10\})$. Let $\tau$ be the other one. Since $\link(\Delta_2,K)=\partial\Delta_3$, we have $\tau\ne(\Delta_1\setminus\{5\})\cup\Delta_2$. Hence
$$
\tau=(\Delta_1\setminus\{5\})\cup(\Delta_2\setminus\{10\})\cup\{v\}
$$
for some $v\in\Delta_3$. Renumbering the vertices of~$\Delta_3$, we can achieve that $v=11$. Then $\tau$ is the simplex~\eqref{eq_simp_5_ns}.
\end{proof}

\begin{comp}
Being run on the subgroup~$G=\rC_5\subset\rS_{15}$ generated by~$g$ and the mandatory subcomplex~$R$ consisting of the subcomplex~\eqref{eq_subcomp_5_ns} and the simplex~\eqref{eq_simp_5_ns}, the program \texttt{find} produces an empty list of  simplicial complexes. Thus, we obtain the following proposition.
\end{comp}

\begin{propos}\label{propos_C5_ns}
  Suppose that $K$ is a $15$-vertex $\Z$-homology $8$-manifold on vertex set~$V$ such that $H_*(K;\Z)\ncong H_*(S^8;\Z)$. Then the symmetry group~$\Sym(K)$ does not contain an element of order~$5$ that acts on~$V$ with $5$ fixed points.
\end{propos}

\subsection{The group~$\rC_4$}
This case is dealt with in exactly the same way as the previous three. We will not do this here, since this has already been done by the author in~\cite{Gai23}. The result is as follows.

\begin{propos}[Proposition~6.1 in~\cite{Gai23}]\label{propos_4}
Suppose that $K$ is a $15$-vertex $\Z$-homology $8$-manifold such that $H_*(K;\Z)\ncong H_*(S^8;\Z)$. Then the symmetry group~$\Sym(K)$ contains no elements of order~$4$.
\end{propos}

\begin{remark}
In~\cite{Gai23} this result is formulated under weaker assumptions on~$K$, namely, for an arbitrary $15$-vertex $\F_2$-homology $8$-manifold that satisfies complimentarity. Proposition~\ref{propos_4} follows from Proposition~6.1 in~\cite{Gai23}, since by Corollary~\ref{cor_complement} any $15$-vertex  $\Z$-homology $8$-manifold~$K$ with $H_*(K;\Z)\ncong H_*(S^8;\Z)$ satisfies complementarity.
\end{remark}

\subsection{The group~$\rC_2\times\rC_2$}

\begin{propos}
 Let $K$ be a $15$-vertex $\Z$-homology $8$-manifold such that $H_*(K;\Z)$ is not isomorphic to~$H_*(S^8;\Z)$. Suppose that the symmetry group~$\Sym(K)$ contains a subgroup~$G$ isomorphic to~$\rC_2\times  \rC_2$. Then we can number the vertices of~$K$ so that to achieve that
 \begin{itemize}
  \item $G$ is generated by the permutations
\begin{equation}\label{eq_C2xC2_gen}
\begin{aligned}
 g&=(1\ 2)(3\ 4)(5\ 6)(7\ 8)(9\ 10)(11\ 12),\\
 h&=(1\ 3)(2\ 4)(5\ 7)(6\ 8)(9\ 11)(10\ 12).
 \end{aligned}
\end{equation}
   \item $K$ contains the twelve $8$-simplices
 \begin{equation}\label{eq_C2xC2_Z}
 \begin{aligned}
  &\{1,2,3,4,5,6,7,8,15\},
  &&\{1,2,3,4,5,6,7,8,13\},\\
  &\{5,6,7,8,9,10,11,12,13\},
  &&\{5,6,7,8,9,10,11,12,14\},\\
  &\{1,2,3,4,9,10,11,12,14\},
  &&\{1,2,3,4,9,10,11,12,15\},\\
  &\{1,2,3,4,5,6,9,10,14\},
  &&\{1,2,3,4,5,6,9,10,15\},\\
  &\{1,2,5,6,7,8,9,10,15\},
  &&\{1,2,5,6,7,8,9,10,13\},\\
  &\{1,2,5,6,9,10,11,12,13\},
  &&\{1,2,5,6,9,10,11,12,14\}.
 \end{aligned}
 \end{equation}
 \end{itemize}
\end{propos}

\begin{proof}
 Let $V$ be the vertex set of~$K$. Choose a nontrivial element $g\in G$. By Corollary~\ref{cor_2} the element $g$ acts on~$V$ with exactly three fixed points and $K^{\langle g\rangle}\cong\CP^2_9$. We identify the vertex set of~$\CP^2_9$ with the affine plane~$\mathcal{P}$ over~$\F_3$ as in Subsection~\ref{subsection_K}. Then some three points of~$\mathcal{P}$ correspond to the three $g$-fixed points in~$V$ and the other six points of~$\mathcal{P}$ correspond to the six $\langle g\rangle$-orbits of length~$2$ in~$V$. Literally as in the proof of Proposition~\ref{propos_C6_aux} we deduce that the three points of~$\mathcal{P}$ that correspond to $g$-fixed points in~$V$ form a non-special line~$m$. Hence, we can choose   affine coordinates~$(x,y)$ in~$\mathcal{P}$ so that:
\begin{itemize}
 \item $m$ coincides with the line $\{x=0\}$,
 \item the special lines in~$\mathcal{P}$ are $\{y=0\}$, $\{y=1\}$, and $\{y=2\}$ with this cyclic order.
\end{itemize}
We can number the vertices of~$K$ so that the bijection $\varphi\colon V/\langle g \rangle\to\mathcal{P}$ corresponding to the isomorphism $K^{\langle g\rangle}\cong\CP^2_9$ is given by Table~\ref{table_bijection_C2xC2}. Then
$$
g = (1\ 2)(3\ 4)(5\ 6)(7\ 8)(9\ 10)(11\ 12).
$$

\begin{table}
 \caption{The bijection between points $(x,y)\in\mathcal{P}$ and $\langle g\rangle $-orbits in~$V$}\label{table_bijection_C2xC2}
 \begin{tabular}{|c|ccc|}
 \hline
 \diaghead(1,-1){abc}{$y$}{$x$} & $0$ & $1$ & $2$\\
 \hline
 $0$ & $\{13\}\vphantom{\Bigl(}$ & $\{1,2\}$  & $\{3,4\}$\\
 $1$ & $\{14\}\vphantom{\Bigl(}$ & $\{5,6\}$  & $\{7,8\}$\\
 $2$ & $\{15\}\vphantom{\Bigl(}$ & $\{9,10\}$ & $\{11,12\}$\\
 \hline
 \end{tabular}
\end{table}

By the construction, the following $12$ subsets of~$\mathcal{P}$ are simplices of~$\CP^2_9$:
\begin{align*}
 &\bigl\{(0,y),(1,y),(2,y),(1,y+1),(2,y+1)\bigr\},&&y\in\F_3,\\
 &\bigl\{(0,y+2),(1,y),(2,y),(1,y+1),(2,y+1)\bigr\},&&y\in\F_3,\\
 &\bigl\{(0,y_1),(1,0),(1,1),(1,2),(2,y_2)\bigr\},&&y_1,y_2\in\F_3,\ y_1\ne y_2.
\end{align*}
Therefore, the preimages of these $12$ simplices under the map
$$V\to V/\langle g\rangle\xrightarrow{\varphi}\mathcal{P}$$
are simplices of~$K$. It is easy to check that these preimages are exactly the $12$ simplices~\eqref{eq_C2xC2_Z}.

Let now $h$ be a nontrivial element of~$G$ such that $h\ne g$. Then $h$ induces an involutive automorphism~$\eta$ of the complex $K^{\langle g\rangle}\cong\CP^2_9$. The corresponding permutation of vertices of~$\CP^2_9$ will be also denoted by~$\eta$.   Since $\eta\in\Sym(\CP^2_9)$, we obtain that $\eta$ is an involutive affine transformation of~$\mathcal{P}$ that takes special lines to special lines and preserves the cyclic order of them. Hence, every special line is invariant under~$\eta$. Moreover, the line~$m$ is also invariant under~$\eta$, since $\eta$ takes $g$-fixed points to $g$-fixed points. It follows easily that
\begin{equation*}
 \eta(x,y)=(-x,y)
\end{equation*}
for all $(x,y)\in\mathcal{P}$. Therefore, swapping vertices in pairs $\{3,4\}$, $\{7,8\}$, and~$\{11,12\}$, we can achieve that
$$
h =(1\ 3)(2\ 4)(5\ 7)(6\ 8)(9\ 11)(10\ 12),
$$
which completes the proof of the proposition.
\end{proof}

\begin{comp}
 Being run on the subgroup~$G=\rC_2\times\rC_2\subset\rS_{15}$ with generators~\eqref{eq_C2xC2_gen} and the mandatory subcomplex~$R$ consisting of the twelve simplices~\eqref{eq_C2xC2_Z}, the program \texttt{find} produces a list of $96$ simplicial complexes, which are divided into $4$ groups, each consisting of $24$ isomorphic complexes.  The simplicial complexes in these $4$ groups are isomorphic to~$\HP^2_{15}(\rA_5)$, $\HP^2_{15}(\rA_4,1)$, $\HP^2_{15}(\rA_4,2)$, and $\HP^2_{15}(\rC_6\times\rC_2)$, respectively. Therefore, we obtain the following proposition.
 \end{comp}

\begin{propos}\label{propos_C2xC2}
Suppose that $K$ is a $15$-vertex $\Z$-homology $8$-manifold such that $H_*(K;\Z)\ncong H_*(S^8;\Z)$ and the symmetry group~$\Sym(K)$ contains a subgroup isomorphic to~$\rC_2\times\rC_2$. Then $K$ is isomorphic to one of the four combinatorial manifolds~$\HP^2_{15}(\rA_5)$, $\HP^2_{15}(\rA_4,1)$, $\HP^2_{15}(\rA_4,2)$, and $\HP^2_{15}(\rC_6\times\rC_2)$.
\end{propos}

Thus, in this case we do not obtain any new triangulations.

\section{$\rC_5$-invariant $10$-vertex homology $3$-spheres}
\label{section_3spheres}

\begin{defin}
 Let $G$ be a group. A \textit{$G$-simplicial complex} is a simplicial complex with a simplicial action of~$G$ on it. Suppose that $K_1$ and $K_2$ are $G$-simplicial complexes. A \textit{weak $G$-isomorphism} is a pair~$(f,\varphi)$ such that $f\colon K_1\to K_2$ is an isomorphism, $\varphi$ is an automorphism of~$G$, and $f(gv)=\varphi(g)f(v)$ for all vertices~$v$ of~$K_1$.
\end{defin}

In this section we study $10$-vertex $\rC_5$-simplicial complexes~$L$ such that
\begin{itemize}
 \item $L$ is a \textit{$\Z$-homology $3$-sphere}, that is, a $\Z$-homology $3$-manifold such that $H_*(L;\Z)$ is isomorphic to~$H_*(S^3;\Z)$,
 \item $\rC_5$ acts on vertices of~$L$ with two orbits of length~$5$.
\end{itemize}
 Our aim is to classify such simplicial complexes (under some additional condition) up to weak $\rC_5$-isomorphism.

Note that every $\Z$-homology $3$-manifold, in particular every $\Z$-homology $3$-sphere, is a combinatorial $3$-manifold. Moreover, Lutz~\cite{Lut07} showed that every combinatorial $3$-manifold with $10$ vertices is homeomorphic to the $3$-sphere~$S^3$, or the sphere product~$S^2\times S^1$, or the $3$-dimensional Klein bottle ($=$ the total space of the twisted $S^2$-bundle over~$S^1$). Therefore, in fact every $\Z$-homology $3$-sphere with $10$ vertices is a combinatorial $3$-sphere. However, we will not use this fact.

\begin{propos}\label{propos_L}
 Up to weak $\rC_5$-isomorphism, there are exactly nine $10$-vertex $\rC_5$-simplicial complexes~$L$ such that
 \begin{enumerate}
  \item $L$ is a $\Z$-homology $3$-sphere,
  \item the vertex set~$W$ of~$L$ splits into two $\rC_5$-orbits of length~$5$,
  \item at least one of the two $\rC_5$-orbits does not contain any $2$-simplex of~$L$.
 \end{enumerate}
 These $9$ simplicial complexes $L_1,\ldots,L_9$ are listed in Table~\ref{table_L}.
 In this table the vertex set of each complex~$L_i$ is the set $W=\{1,\ldots,10\}$ on which the group~$\rC_5$ acts by the permutation
 \begin{equation}\label{eq_C5_gen}
  (1\ 2\ 3\ 4\ 5)(6\ 7\ 8\ 9\ 10).
 \end{equation}
 Each row of the table contains a list of representatives for $\rC_5$-orbits of $3$-simplices of~$L_i$.
 \end{propos}

\begin{table}
\caption{Simplicial $3$-spheres~$L_1,\ldots,L_{9}$}\label{table_L}
 \begin{tabular}{|l|l|}
  \hline
  & \textbf{Representatives of $\rC_5$-orbits of $3$-simplices}\\
  \hline
  $L_1$ & $\{1,2,3,6\},\ \{1,2,4,6\},\ \{1,3,4,6\},\ \{2,3,4,6\}$\\
  \hline
  $L_2$ & $\{1,2,6,7\},\ \{2,3,6,7\},\ \{3,4,6,7\},\ \{4,5,6,7\},\ \{1,5,6,7\}$\\
  \hline
  $L_3$ & $\{1,3,6,7\},\ \{3,5,6,7\},\ \{2,5,6,7\},\ \{2,4,6,7\},\ \{1,4,6,7\}$\\
  \hline
  $L_4$ & $\{1,2,6,7\},\ \{2,3,6,7\},\ \{1,3,6,7\},\ \{1,3,5,6\},\ \{2,3,5,6\}$\\
  \hline
  $L_5$ & $\{1,2,6,7\},\ \{2,3,6,7\},\ \{3,4,6,7\},\ \{1,4,6,7\},\ \{1,3,4,6\},\ \{1,3,5,6\}$\\
  \hline
  $L_6$ & $\{1,2,6,7\},\ \{2,3,6,7\},\ \{3,4,6,7\},\ \{1,4,6,7\},\ \{1,4,5,6\},\ \{3,4,5,6\}$\\
  \hline
  $L_7$ & $\{1,3,6,7\},\ \{2,3,6,7\},\ \{2,4,6,7\},\ \{1,4,6,7\},\ \{1,2,4,6\},\ \{2,3,5,6\}$\\
    \hline
  $L_8$ & $\{1,2,6,7\},\ \{2,3,6,7\},\ \{1,3,6,7\},\ \{2,3,6,8\},\ \{3,5,6,8\},\ \{2,5,6,8\}$\\
  \hline
  $L_9$ & $\{1,2,6,7\},\ \{2,3,6,7\},\ \{3,4,6,7\},\ \{1,4,6,7\},\ \{1,3,6,8\},\ \{3,5,6,8\},\ \{1,5,6,8\}$\\
    \hline
 \end{tabular}
\end{table}

Suppose that $K$ is a simplicial complex. We denote by $C_k(K;\F_2)$ the group of simplicial $k$-chains of~$K$ with coefficients in~$\F_2$ and by~$Z_k(K;\F_2)$ the group of $k$-cycles, i.\,e., the kernel of the boundary operator
$$
\partial\colon C_k(K;\F_2)\to C_{k-1}(K;\F_2).
$$
We denote by $\|\xi\|$ the number of $k$-simplices that enter a chain $\xi\in C_k(K;\F_2)$.

\begin{defin}
 A $1$-cycle $\xi\in Z_1(K;\F_2)$ will be called a \textit{simple cycle of length} $m\ge 3$ if
 \begin{equation}\label{eq_cycle}
  \xi = \{v_1,v_2\}+\{v_2,v_3\}+\cdots+\{v_{m-1},v_m\}+\{v_m,v_1\}
 \end{equation}
 for pairwise different vertices $v_1,\ldots,v_m$. We denote the cycle~\eqref{eq_cycle} by $(v_1,\ldots,v_m)$.
\end{defin}

\begin{proof}[Proof of Proposition~\ref{propos_L}]
Each simplicial complex $L_i$ has $10$ vertices, which form two $\rC_5$-orbits. There are several ways to show that $L_i$ is a $\Z$-homology $3$-sphere. First, one can show that $L_i$ is a combinatorial $3$-sphere using any standard software such as the program \texttt{BISTELLAR} by Lutz. Second, one can use the author's program \texttt{check} in~\cite{Gai-prog} to check that the suspension $\Sigma L_i$ is a combinatorial $4$-manifold, which again implies that $L_i$ is a combinatorial $3$-sphere. Third, one can check manually (without a computer) that the link of every vertex of~$L_i$ is a simplicial $2$-sphere and hence $L_i$ is a combinatorial $3$-manifold. Then one can check either manually or with a computer that $H_1(L_i;\Z)=0$. Anyway, $L_i$ is a $\Z$-homology $3$-sphere.
Thus, the simplicial complexes $L_1,\ldots,L_9$ satisfy conditions~(1)--(3) from Proposition~\ref{propos_L}.

Let now $L$ be an arbitrary simplicial complex that satisfy conditions~(1)--(3). We need to prove that $L$ is weakly $\rC_5$-isomorphic to one of the complexes~$L_1,\ldots,L_9$. Let $f_k$ denote the number of $k$-simplices of~$L$. Since $L$ is a $\Z$-homology $3$-sphere, we have
$$
f_0-f_1+f_2-f_3=0.
$$
Also we have $f_0=10$ and $f_2=2f_3$. Hence
\begin{equation}
 \label{eq_f1f3}
 f_3=f_1-10.
\end{equation}
Moreover, all the numbers $f_k$ are divisible by~$5$, since a nonempty simplex of~$L$ cannot be $\rC_5$-invariant.

Let $\alpha_1$ and~$\alpha_2$ be the two $\rC_5$-orbits in~$W$. We can number the elements of~$W$ so that $\alpha_1=\{1,2,3,4,5\}$, $\alpha_2=\{6,7,8,9,10\}$, and the group~$\rC_5$ is generated by permutation~\eqref{eq_C5_gen}. The simplicial complex~$L$ is contained in the $9$-simplex~$\Delta_W$ with vertex  set~$W$. For each~$s=1,2$, the $10$ edges of~$\Delta_W$ with both endpoints in~$\alpha_s$ form two $\rC_5$-orbits. Suppose that exactly $m_s$ of these two orbits enter the simplicial complex~$L$; then $0\le m_s\le 2$. Swapping the orbits~$\alpha_1$ and~$\alpha_2$, we can achieve that $m_1\ge m_2$. Consider four cases.
\smallskip

\textsl{Case 1: $m_2=0$.} Then at least two of the nine $\rC_5$-orbits of edges of~$\Delta_W$ are not simplices of~$L$. Hence $f_1\le 35$. Consider the simplicial complex $J=\link(6,L)$. No vertex of~$\alpha_2$ is a vertex of~$J$. Hence $J$ has at most $5$ vertices. On the other hand, $J$ is a simplicial $2$-sphere, so $J$ has at least $4$ vertices. Consider two subcases:

\smallskip
\textsl{Subcase 1a: $J$ has $4$ vertices.} Then $J$ is the boundary of a $3$-simplex. Renumbering cyclically the vertices of~$\alpha_1$, we can achieve that the vertices of~$J$ are $1$, $2$, $3$, and~$4$. Then $L$ contains~$L_1$ and hence coincides with it.

\smallskip
\textit{Subcase 1b: $J$ has $5$ vertices.} Then $J$ has six $2$-simplices. Hence the vertex~$6$ is contained in six $3$-simplices of~$L$. These six $3$-simplices lie in pairwise different $\rC_5$-orbits, so $f_3\ge 30$. Using~\eqref{eq_f1f3}, we arrive at a contradiction, since $f_1\le 35$. So this subcase is impossible.

\smallskip
\textsl{Case 2: $m_1=m_2=1$.} Then $L$ contains exactly one of the two edges~$\{1,2\}$ and~$\{1,3\}$  and exactly one of the two edges~$\{6,7\}$ and~$\{6,8\}$. Note that the property that the symmetry group~$\rC_5$ is generated by permutation~\eqref{eq_C5_gen} will not be violated if we renumber the vertices of~$L$ by means of the permutation
\begin{equation}
 \label{eq_C5_aut}
 (2\ 3\ 5\ 4)(7\ 8\ 10\ 9).
\end{equation}
(This permutation corresponds to the automorphism $g\mapsto g^2$ of the symmetry group~$\rC_5$.) Using such renumbering, we can achieve that $\{6,7\}$ is a simplex of~$L$.

Note that $L$ has no $2$-simplices that are contained in either~$\alpha_1$ or~$\alpha_2$, since every such simplex would contain edges belonging to two different $\rC_5$-orbits. Therefore, every $3$-simplex of~$L$ has two vertices in~$\alpha_1$ and two vertices in~$\alpha_2$. Let $Z$ be the modulo~$2$ fundamental cycle of~$\link(\{6,7\},L)$. Then $Z$ is a simple $1$-cycle with vertices in~$\alpha_1$. However, $Z$ may contain edges from only one of the two $\rC_5$-orbits with representatives~$\{1,2\}$ and~$\{1,3\}$. Therefore, $Z$ is either the $1$-cycle $(1,2,3,4,5)$ or the $1$-cycle $(1,3,5,2,4)$. Thus, $L$ contains one of the two complexes~$L_2$ or~$L_3$ and hence coincides with it.
\smallskip

\textsl{Case 3: $m_1=2$ and~$m_2=1$.} Likewise the previous case, we can use permutation~\eqref{eq_C5_aut} of vertices to achieve that $\{6,7\}\in L$ and $\{6,8\}\notin L$.  Let $Z$ be the modulo~$2$ fundamental cycle of~$\link(\{6,7\},L)$. The complex~$L$ has no $2$-simplices that are contained in~$\alpha_2$, so $Z$ is a simple $1$-cycle with vertices in~$\alpha_1$. Let $Y\in C_2(\alpha_1;\F_2)$ be the sum modulo~$2$ of all $2$-simplices that belong to~$\link(6,L)$ and are contained in~$\alpha_1$. Then the modulo~$2$ fundamental cycle of~$L$ has the form
\begin{equation}\label{eq_C5_X}
 X=\sum_{j=0}^4g^j\cdot\left(Z*\{6,7\}+Y*\{6\}\right),
\end{equation}
where
$$
g=(1\ 2\ 3\ 4\ 5)(6\ 7\ 8\ 9\ 10)
$$
is the generator of~$\rC_5$ and $*$ denotes the join. Then
$$
0=\partial X=\sum_{j=0}^4g^j\cdot\left(\left(Z+g^{-1}\cdot Z+\partial Y\right)*\{6\}+Y\right)
$$
and hence
\begin{equation}\label{eq_C5_dY}
\partial Y=Z+g^{-1}\cdot Z.
\end{equation}
Therefore, $\|\partial Y\|$ is even. Since the boundary of every $2$-simplex is the sum of three $1$-simplices, it follows that $\|Y\|$ is even, too.

From~\eqref{eq_C5_X} it follows that
$$
f_3=\|X\|=5\|Z\|+5\|Y\|.
$$
Combining this with~\eqref{eq_f1f3}, we obtain that
$$
f_1=5\|Z\|+5\|Y\|+10.
$$
Exactly $10$ edges of~$L$ have both endpoints in~$\alpha_1$ and exactly $5$ edges of~$L$ have both endpoints in~$\alpha_2$. Hence, any vertex in $\alpha_2$ is connected by edges of~$L$ with exactly~$q$ vertices in~$\alpha_1$, where
$$
q=\frac{f_1-15}{5}=\|Z\|+\|Y\|-1.
$$
On the other hand, the vertex~$6$ is connected by an edge of~$L$ with every vertex in the support of~$Z$ and every vertex in the support of~$g^{-1}\cdot Z$. Since the support of~$Z$ has at least three vertices, it follows easily that the union of the supports of~$Z$ and~$g^{-1}\cdot Z$ has at least $4$ vertices. Hence $q$ is either~$4$ or~$5$. Since $3\le \|Z\|\le 5$ and $\|Y\|$ is even, we see that $(\|Z\|,\|Y\|,q)$ is one of the three triples~$(3,2,4)$, $(4,2,5)$, and~$(5,0,4)$.

Suppose that $\|Y\|=0$.  Then from~\eqref{eq_C5_dY} it follows that the $1$-cycle~$Z$ is $\rC_5$-invariant and hence is one of the two $1$-cycles~$(1,2,3,4,5)$ and~$(1,3,5,2,4)$. Then $L$ contains one of the two complexes~$L_2$ and~$L_3$ and hence coincides with it. We arrive at a contradicition, since $m_1=m_2=1$ for both~$L_2$ and~$L_3$. Thus, $\|Y\|=2$ and $(\|Z\|,q)$ is either~$(3,4)$ or~$(4,5)$.

We still have freedom to permute cyclically the vertices $1, 2, 3,4,5$. Up to such cyclic permutations, there are exactly five simple $1$-cycles~$Z$ of length either~$3$ or~$4$, see Table~\ref{table_cycles}. The corresponding $1$-cycles $Z+g^{-1}\cdot Z$ are also shown in the  table. Consider the five subcases separately.
\smallskip

\begin{table}
 \caption{Simple $1$-cycles~$Z$ of lengths~$3$ and~$4$ and the corresponding $1$-cycles~$Z+g^{-1}\cdot Z$  and~$Z+g^{-2}\cdot Z$. (The latter are not depicted for~$Z$ of length~$4$, since we do not need them.) }
 \label{table_cycles}
\begin{tabular}{|c|c|c|c|c|}
 \hline
 No. & $\|Z\|$ & $Z$ & $Z+g^{-1}\cdot Z\vphantom{\Bigl(}$ & $Z+g^{-2}\cdot Z$ \\
 \hline
 \begin{tikzpicture}[scale =.8]
  \draw[color=white] (0,-1.3)--(0,1.3);
  \draw(0,0.15) node {1};
 \end{tikzpicture}
 &
 \begin{tikzpicture}[scale =.8]
  \draw[color=white] (0,-1.3)--(0,1.3);
  \draw(0,0.15) node {3};
 \end{tikzpicture}
 &
 \begin{tikzpicture}[scale =.8]
   \draw[color=white] (-2,-1.2)--(2,1.75);
   \fill (0,1) circle (2pt) node [above,yshift=1pt] {\footnotesize{3}};
   \fill (0.951,0.309) circle (2pt) node [right,yshift=2pt] {\footnotesize{4}};
   \fill (-0.951,0.309) circle (2pt) node [left,yshift=2pt] {\footnotesize{2}};
   \fill (0.588,-0.809) circle (2pt) node [below,xshift=3pt] {\footnotesize{5}};
   \fill (-0.588,-0.809) circle (2pt) node [below,xshift=-3pt] {\footnotesize{1}};
   \draw (-0.588,-0.809)--(-0.951,0.309)--(0,1)--(-0.588,-0.809);
 \end{tikzpicture}
 &
 \begin{tikzpicture}[scale =.8]
   \draw[color=white] (-2,-1.2)--(2,1.75);
   \fill (0,1) circle (2pt) node [above,yshift=1pt] {\footnotesize{3}};
   \fill (0.951,0.309) circle (2pt) node [right,yshift=2pt] {\footnotesize{4}};
   \fill (-0.951,0.309) circle (2pt) node [left,yshift=2pt] {\footnotesize{2}};
   \fill (0.588,-0.809) circle (2pt) node [below,xshift=3pt] {\footnotesize{5}};
   \fill (-0.588,-0.809) circle (2pt) node [below,xshift=-3pt] {\footnotesize{1}};
   \draw (-0.588,-0.809)--(0.588,-0.809)--(-0.951,0.309)--(0,1)--(-0.588,-0.809);
 \end{tikzpicture}
 &
 \begin{tikzpicture}[scale =.8]
   \draw[color=white] (-2,-1.2)--(2,1.75);
   \fill (0,1) circle (2pt) node [above,yshift=1pt] {\footnotesize{3}};
   \fill (0.951,0.309) circle (2pt) node [right,yshift=2pt] {\footnotesize{4}};
   \fill (-0.951,0.309) circle (2pt) node [left,yshift=2pt] {\footnotesize{2}};
   \fill (0.588,-0.809) circle (2pt) node [below,xshift=3pt] {\footnotesize{5}};
   \fill (-0.588,-0.809) circle (2pt) node [below,xshift=-3pt] {\footnotesize{1}};
   \draw (-0.588,-0.809)--(-0.951,0.309)--(0,1)--(-0.588,-0.809)--(0.951,0.309)--(0.588,-0.809)--(-0.588,-0.809);
 \end{tikzpicture}\\
 \hline
 \begin{tikzpicture}[scale =.8]
  \draw[color=white] (0,-1.3)--(0,1.3);
  \draw(0,0.15) node {2};
 \end{tikzpicture}
 &
 \begin{tikzpicture}[scale =.8]
  \draw[color=white] (0,-1.3)--(0,1.3);
  \draw(0,0.15) node {3};
 \end{tikzpicture}
 &
 \begin{tikzpicture}[scale =.8]
   \draw[color=white] (-2,-1.2)--(2,1.75);
   \fill (0,1) circle (2pt) node [above,yshift=1pt] {\footnotesize{3}};
   \fill (0.951,0.309) circle (2pt) node [right,yshift=2pt] {\footnotesize{4}};
   \fill (-0.951,0.309) circle (2pt) node [left,yshift=2pt] {\footnotesize{2}};
   \fill (0.588,-0.809) circle (2pt) node [below,xshift=3pt] {\footnotesize{5}};
   \fill (-0.588,-0.809) circle (2pt) node [below,xshift=-3pt] {\footnotesize{1}};
   \draw (-0.588,-0.809)--(-0.951,0.309)--(0.951,0.309)--(-0.588,-0.809);
 \end{tikzpicture}
 &
 \begin{tikzpicture}[scale =.8]
   \draw[color=white] (-2,-1.2)--(2,1.75);
   \fill (0,1) circle (2pt) node [above,yshift=1pt] {\footnotesize{3}};
   \fill (0.951,0.309) circle (2pt) node [right,yshift=2pt] {\footnotesize{4}};
   \fill (-0.951,0.309) circle (2pt) node [left,yshift=2pt] {\footnotesize{2}};
   \fill (0.588,-0.809) circle (2pt) node [below,xshift=3pt] {\footnotesize{5}};
   \fill (-0.588,-0.809) circle (2pt) node [below,xshift=-3pt] {\footnotesize{1}};
   \draw (-0.588,-0.809)--(-0.951,0.309)--(0.951,0.309)--(-0.588,-0.809)--(0,1)--(0.588,-0.809)--(-0.588,-0.809);
 \end{tikzpicture}
 &
 \begin{tikzpicture}[scale =.8]
   \draw[color=white] (-2,-1.2)--(2,1.75);
   \fill (0,1) circle (2pt) node [above,yshift=1pt] {\footnotesize{3}};
   \fill (0.951,0.309) circle (2pt) node [right,yshift=2pt] {\footnotesize{4}};
   \fill (-0.951,0.309) circle (2pt) node [left,yshift=2pt] {\footnotesize{2}};
   \fill (0.588,-0.809) circle (2pt) node [below,xshift=3pt] {\footnotesize{5}};
   \fill (-0.588,-0.809) circle (2pt) node [below,xshift=-3pt] {\footnotesize{1}};
   \draw (-0.588,-0.809)--(-0.951,0.309)--(0.588,-0.809)--(0.951,0.309)--(-0.588,-0.809);
 \end{tikzpicture}\\
 \hline
 \begin{tikzpicture}[scale =.8]
  \draw[color=white] (0,-1.3)--(0,1.3);
  \draw(0,0.15) node {3};
 \end{tikzpicture}
&
 \begin{tikzpicture}[scale =.8]
  \draw[color=white] (0,-1.3)--(0,1.3);
  \draw(0,0.15) node {4};
 \end{tikzpicture}
 &
 \begin{tikzpicture}[scale =.8]
   \draw[color=white] (-2,-1.2)--(2,1.75);
   \fill (0,1) circle (2pt) node [above,yshift=1pt] {\footnotesize{3}};
   \fill (0.951,0.309) circle (2pt) node [right,yshift=2pt] {\footnotesize{4}};
   \fill (-0.951,0.309) circle (2pt) node [left,yshift=2pt] {\footnotesize{2}};
   \fill (0.588,-0.809) circle (2pt) node [below,xshift=3pt] {\footnotesize{5}};
   \fill (-0.588,-0.809) circle (2pt) node [below,xshift=-3pt] {\footnotesize{1}};
   \draw (-0.588,-0.809)--(-0.951,0.309)--(0,1)--(0.951,0.309)--(-0.588,-0.809);
 \end{tikzpicture}
 &
 \begin{tikzpicture}[scale =.8]
   \draw[color=white] (-2,-1.2)--(2,1.75);
   \fill (0,1) circle (2pt) node [above,yshift=1pt] {\footnotesize{3}};
   \fill (0.951,0.309) circle (2pt) node [right,yshift=2pt] {\footnotesize{4}};
   \fill (-0.951,0.309) circle (2pt) node [left,yshift=2pt] {\footnotesize{2}};
   \fill (0.588,-0.809) circle (2pt) node [below,xshift=3pt] {\footnotesize{5}};
   \fill (-0.588,-0.809) circle (2pt) node [below,xshift=-3pt] {\footnotesize{1}};
   \draw (-0.588,-0.809)--(0.951,0.309)--(0,1)--(0.588,-0.809)--(-0.588,-0.809);
 \end{tikzpicture}
 &
 \\
 \hline
 \begin{tikzpicture}[scale =.8]
  \draw[color=white] (0,-1.3)--(0,1.3);
  \draw(0,0.15) node {4};
 \end{tikzpicture}
&
 \begin{tikzpicture}[scale =.8]
  \draw[color=white] (0,-1.3)--(0,1.3);
  \draw(0,0.15) node {4};
 \end{tikzpicture}
 &
 \begin{tikzpicture}[scale =.8]
   \draw[color=white] (-2,-1.2)--(2,1.75);
   \fill (0,1) circle (2pt) node [above,yshift=1pt] {\footnotesize{3}};
   \fill (0.951,0.309) circle (2pt) node [right,yshift=2pt] {\footnotesize{4}};
   \fill (-0.951,0.309) circle (2pt) node [left,yshift=2pt] {\footnotesize{2}};
   \fill (0.588,-0.809) circle (2pt) node [below,xshift=3pt] {\footnotesize{5}};
   \fill (-0.588,-0.809) circle (2pt) node [below,xshift=-3pt] {\footnotesize{1}};
   \draw (-0.588,-0.809)--(-0.951,0.309)--(0.951,0.309)--(0,1)--(-0.588,-0.809);
 \end{tikzpicture}
 &
 \begin{tikzpicture}[scale =.8]
   \draw[color=white] (-2,-1.2)--(2,1.75);
   \fill (0,1) circle (2pt) node [above,yshift=1pt] {\footnotesize{3}};
   \fill (0.951,0.309) circle (2pt) node [right,yshift=2pt] {\footnotesize{4}};
   \fill (-0.951,0.309) circle (2pt) node [left,yshift=2pt] {\footnotesize{2}};
   \fill (0.588,-0.809) circle (2pt) node [below,xshift=3pt] {\footnotesize{5}};
   \fill (-0.588,-0.809) circle (2pt) node [below,xshift=-3pt] {\footnotesize{1}};
   \draw (-0.588,-0.809)--(-0.951,0.309)--(0.588,-0.809)--(-0.588,-0.809);
   \draw (0,1)--(-0.951,0.309)--(0.951,0.309)--(0,1);
 \end{tikzpicture}
 &
 \\
 \hline
 \begin{tikzpicture}[scale =.8]
  \draw[color=white] (0,-1.3)--(0,1.3);
  \draw(0,0.15) node {5};
 \end{tikzpicture}
&
 \begin{tikzpicture}[scale =.8]
  \draw[color=white] (0,-1.3)--(0,1.3);
  \draw(0,0.15) node {4};
 \end{tikzpicture}
 &
 \begin{tikzpicture}[scale =.8]
   \draw[color=white] (-2,-1.2)--(2,1.75);
   \fill (0,1) circle (2pt) node [above,yshift=1pt] {\footnotesize{3}};
   \fill (0.951,0.309) circle (2pt) node [right,yshift=2pt] {\footnotesize{4}};
   \fill (-0.951,0.309) circle (2pt) node [left,yshift=2pt] {\footnotesize{2}};
   \fill (0.588,-0.809) circle (2pt) node [below,xshift=3pt] {\footnotesize{5}};
   \fill (-0.588,-0.809) circle (2pt) node [below,xshift=-3pt] {\footnotesize{1}};
   \draw (-0.588,-0.809)--(0,1)--(-0.951,0.309)--(0.951,0.309)--(-0.588,-0.809);
 \end{tikzpicture}
 &
 \begin{tikzpicture}[scale =.8]
   \draw[color=white] (-2,-1.2)--(2,1.75);
   \fill (0,1) circle (2pt) node [above,yshift=1pt] {\footnotesize{3}};
   \fill (0.951,0.309) circle (2pt) node [right,yshift=2pt] {\footnotesize{4}};
   \fill (-0.951,0.309) circle (2pt) node [left,yshift=2pt] {\footnotesize{2}};
   \fill (0.588,-0.809) circle (2pt) node [below,xshift=3pt] {\footnotesize{5}};
   \fill (-0.588,-0.809) circle (2pt) node [below,xshift=-3pt] {\footnotesize{1}};
   \draw (-0.588,-0.809)--(-0.951,0.309)--(0.951,0.309)--(-0.588,-0.809);
   \draw (0,1)--(-0.951,0.309)--(0.588,-0.809)--(0,1);
 \end{tikzpicture}
 &
 \\
 \hline
\end{tabular}
\end{table}

 \textsl{Subcase 3a: $Z$ is the $1$-cycle in row~1 of Table~\ref{table_cycles}.} Then there are exactly two $2$-chains $Y$ such that $\partial Y=Z+g^{-1}\cdot Z$ and $\|Y\|=2$, namely,
 \begin{align*}
  Y_1&=\{1,3,5\}+\{2,3,5\},\\
  Y_2&=\{1,2,3\}+\{1,2,5\}.
 \end{align*}
 If $Y=Y_1$, then we obtain that $L=L_4$. If $Y=Y_2$, then the three $3$-simplices $\{1,2,6,7\}$, $\{1,2,6,10\}$, and $\{1,2,3,6\}$ belong to~$L$ and have a common two-face~$\{1,2,6\}$, so $L$ is not a pseudo-manifold. Thus, the case $Y=Y_2$ is impossible.
 \smallskip

 \textsl{Subcase 3b: $Z$ is the $1$-cycle in row~2 of Table~\ref{table_cycles}.} Then $\|Z\|=3$. On the other hand, the vertex~$6$ is connected by an edge of~$L$ with every vertex of~$\alpha_1$, so $q=5$. We arrive at a contradiciton, since $(\|Z\|,q)$ must be either $(3,4)$ or~$(4,5)$. Hence this subcase is impossible.
 \smallskip

 \textsl{Subcase 3c:   $Z$ is the $1$-cycle in row~3 of Table~\ref{table_cycles}.} Then there are exactly two $2$-chains $Y$ such that $\partial Y=Z+g^{-1}\cdot Z$ and $\|Y\|=2$, namely,
 \begin{align*}
  Y_1&=\{1,3,4\}+\{1,3,5\},\\
  Y_2&=\{1,4,5\}+\{3,4,5\}.
 \end{align*}
 If $Y=Y_1$, then we obtain that $L=L_5$, and if $Y=Y_2$, then we obtain that $L=L_6$.
 \smallskip

 \textsl{Subcase 3d: $Z$ is the $1$-cycle in row~4 of Table~\ref{table_cycles}.} Then there is exactly one $2$-chain~$Y$ such that $\partial Y=Z+g^{-1}\cdot Z$ and $\|Y\|=2$, namely,
 \begin{equation*}
  Y=\{1,2,5\}+\{2,3,4\}.
 \end{equation*}
 Hence $L$ contains the six $\rC_5$-orbits of $3$-simplices with representatives
 \begin{equation*}
 \{1,2,6,7\},\ \{2,4,6,7\},\ \{3,4,6,7\},\ \{1,3,6,7\},\ \{1,2,5,6\},\ \{2,3,4,6\}.
\end{equation*}
 It is easy to check that the simplicial complex~$L^*_1$ consisting of all $3$-simplices in these six $\rC_5$-orbits and all their faces is a $3$-pseudomanifold. Therefore, $L=L_1^*$. However, $L_1^*$ is not in Table~\ref{table_L}. This is because $L_1^*$ is not orientable and hence is not a $\Z$-homology $3$-sphere. Indeed, consider two $3$-simplices $\sigma=\{1,2,5,6\}$ and $\tau=\{1,2,6,7\}$ of~$L_1^*$. Then $\tau$ has a common facet with each of the two $3$-simplices~$\sigma$ and~$g\cdot\sigma=\{1,2,3,7\}$. Choose the orientation of~$\sigma$ given by the ordering $1,2,5,6$ of its vertices. Let us travel in $|L_1^*|$ from~$b(\sigma)$ to~$b(g\cdot\sigma)$ along a path~$\gamma$ that goes through~$\tau$, crossing consecutively the $2$-simplices $\sigma\cap\tau$ and $(g\cdot\sigma)\cap\tau$.  It is easy to see that we arrive at the orientation of~$g\cdot\sigma$ given by the ordering $1,3,2,7$ of its vertices. This orientation is opposite to the one obtained from the chosen orientation of~$\sigma$ by the action of~$g$. Hence, concatenating the paths~$\gamma$, $g\cdot\gamma$, $\ldots$, $g^4\cdot\gamma$, we get a closed curve such that the orientation of~$L_1^*$ is reversed when passing along this curve. Therefore, $L_1^*$ is not a $\Z$-homology $3$-sphere. Thus, this subcase is impossible. Though it is not improtant for us, note that in fact $L_1^*$ is homeomorphic to the $3$-dimensional Klein bottle.
 \smallskip

 \textsl{Subcase 3e: $Z$ is the $1$-cycle in row~5 of Table~\ref{table_cycles}.} Then there is exactly one $2$-chain~$Y$ such that $\partial Y=Z+g^{-1}\cdot Z$ and $\|Y\|=2$, namely,
 \begin{equation*}
  Y=\{1,2,4\}+\{2,3,5\}.
 \end{equation*}
 It follows that $L=L_7$.
\smallskip

\textsl{Case 4: $m_1=m_2=2$.} Then every two vertices of~$L$ that belong to the same $\rC_5$-orbit are connected by an edge. Let $L_{\alpha_1}$ and~$L_{\alpha_2}$ be the full subcomplexes of~$L$ spanned by~$\alpha_1$ and~$\alpha_2$, respectively. Then $|L_{\alpha_2}|$ is a deformation retract of $|L|\setminus |L_{\alpha_1}|$. Since $L$ is a $\Z$-homology $3$-sphere, the Alexander duality implies that
$$
\rank H_1(L_{\alpha_1};\Z)=\rank H_1(L_{\alpha_2};\Z).
$$
From condition~(3) in Proposition~\ref{propos_L} it follows that at least one of the two complexes~$L_{\alpha_1}$ and~$L_{\alpha_2}$ (say, $L_{\alpha_1}$) contains no $2$-simplices and so is isomorphic to the complete graph~$K_5$ on~$5$ vertices. Then $\rank H_1(L_{\alpha_1};\Z)=6$. If the complex~$L_{\alpha_2}$ contained a $2$-simplex, then $\rank H_1(L_{\alpha_2};\Z)$ would be strictly smaller than~$6$. Therefore, neither~$L_{\alpha_1}$ nor~$L_{\alpha_2}$ contains a $2$-simplex. Thus, every $3$-simplex of~$L$ has exactly two vertices in~$\alpha_1$ and exactly two vertices in~$\alpha_2$.

Let $Z_1$ and $Z_2$ be the modulo~$2$ fundamental cycles of~$\link(\{6,7\},L)$ and~$\link(\{6,8\},L)$, respectively. Then the supports of~$Z_1$ and~$Z_2$ are contained in~$\alpha_1$. Renumbering the vertices of~$L$ by means of permutation~\eqref{eq_C5_aut}, we can achieve that $\|Z_1\|\ge \|Z_2\|$. The modulo~$2$ fundamental cycle of~$L$ has the form
\begin{equation}\label{eq_C5_X_2}
 X=\sum_{j=0}^4g^j\cdot\left(Z_1*\{6,7\}+Z_2*\{6,8\}\right).
\end{equation}
Hence,
\begin{equation*}
0=\partial X=\sum_{j=0}^4g^j\cdot\left(\left(Z_1+g^{-1}\cdot Z_1+Z_2+g^{-2}\cdot Z_2\right)  *\{6\}\right).
\end{equation*}
Therefore,
\begin{equation}\label{eq_Z1Z2}
 Z_1+g^{-1}\cdot Z_1=Z_2+g^{-2}\cdot Z_2.
\end{equation}
Moreover, from~\eqref{eq_C5_X_2} it follows that
$$
f_3=\|X\|=5\|Z_1\|+5\|Z_2\|.
$$
On the other hand, from~\eqref{eq_f1f3} it follows that
$$
f_3=f_1-10\le \binom{10}{2}-10=35.
$$
Hence $\|Z_1\|+\|Z_2\|\le 7$. Therefore, $\|Z_2\|=3$ and $\|Z_1\|$ is either~$3$ or~$4$.

Permuting cyclically the vertices $1,2,3,4,5$, we can achieve that $Z_1$ is one of the four simple $1$-cycles from Table~\ref{table_cycles}. On the other hand, $Z_2$ is obtained by a rotation from one of the two $1$-cycles of length~$3$ from Table~\ref{table_cycles}. Equality~\eqref{eq_Z1Z2} can only occur in the following three subcases.
\smallskip

\textsl{Subcase 4a:} $Z_1$ is the $1$-cycle in row~$1$ of Table~\ref{table_cycles} and $Z_2$ is the $1$-cycle~$(2,3,5)$, which can be obtained from the $1$-cycle in row~$2$ of Table~\ref{table_cycles} by the permutation $(1\ 2\ 3\ 4\ 5)$. Then $L=L_8$.
\smallskip

\textsl{Subcase 4b:} $Z_1$ is the $1$-cycle in row~$3$ of Table~\ref{table_cycles} and $Z_2$ is the $1$-cycle~$(1,3,5)$, which can be obtained from the $1$-cycle in row~$2$ of Table~\ref{table_cycles} by the permutation $(1\ 5\ 4\ 3\ 2)$. Then $L=L_9$.
\smallskip

\textsl{Subcase 4c:} $Z_1$ is the $1$-cycle in row~$4$ of Table~\ref{table_cycles} and $Z_2$ is the $1$-cycle~$(2,3,4)$, which can be obtained from the $1$-cycle in row~$1$ of Table~\ref{table_cycles} by the permutation $(1\ 2\ 3\ 4\ 5)$. Then $L$ is the $3$-pseudomanifold $L_2^*$ consisting of all $3$-simplices in $\rC_5$-orbits with representatives
$$
\{1,2,6,7\},\ \{2,4,6,7\},\ \{3,4,6,7\},\ \{1,3,6,7\},\ \{2,3,6,8\},\ \{3,4,6,8\},\ \{2,4,6,8\}
$$
and all their faces. As in Subcase~3d above, let us show that $L_2^*$ is not orientable and hence is not a $\Z$-homology $3$-sphere. The pseudomanifold~$L_2^*$ contains the simplices
$$
\rho=\{1,2,6,7\},\qquad \sigma=\{1,2,7,10\},\qquad \tau=\{2,3,7,10\}.
$$
Choose the orientation of~$\rho$ given by the ordering $1,2,6,7$ of its vertices. Let us travel in $|L_2^*|$ from~$b(\rho)$ to~$b(g\cdot\rho)$ along a path~$\gamma$ that goes through~$\sigma$ and $\tau$, crossing consecutively the $2$-simplices $\rho\cap\sigma$, $\sigma\cap\tau$, and $\tau\cap (g\cdot\rho)$. Then we arrive at the orientation of~$g\cdot\rho$ given by the ordering $2,3,8,7$ of its vertices. This orientation is opposite to the one obtained from the chosen orientation of~$\rho$ by the action of~$g$. Hence, concatenating the paths~$\gamma$, $g\cdot\gamma$, $\ldots$, $g^4\cdot\gamma$, we get a closed curve such that the orientation of~$L_2^*$ is reversed when passing along this curve. Therefore, $L_2^*$ is not a $\Z$-homology $3$-sphere. Thus, this subcase is impossible. As in case Subcase~3d, in fact $L_2^*$ is homeomorphic to the $3$-dimensional Klein bottle.

\smallskip

In all cases we have shown that $L$ is weakly $\rC_5$-isomorphic to one of the simplicial complexes $L_1,\ldots,L_9$. The proposition follows.
\end{proof}

\section{Triangulations with symmetry group~$\rC_5$}\label{section_5}

Suppose that $K$ is a $15$-vertex $\Z$-homology $8$-manifold such that $H_*(K;\Z)\ncong H_*(S^8;\Z)$. From Corollary~\ref{cor_5} and Proposition~\ref{propos_C5_ns} it follows that, if the group~$\Sym(K)$ contains a subgroup isomorphic to~$\rC_5$, then this subgroups acts on the vertices of~$K$ without fixed points.

\begin{propos}\label{propos_C5_aux}
 Suppose that $K$ is a $\rC_5$-invariant $15$-vertex $\Z$-homology $8$-manifold with $H_*(K;\Z)\ncong H_*(S^8;\Z)$, where $\rC_5$ acts on vertices of~$K$ without fixed points. Then $K$ contains a $\rC_5$-invariant $4$-simplex~$\sigma$ such that $\link(\sigma,K)$  either is the boundary of a $4$-simplex or is weakly $\rC_5$-isomorphic to one the $9$ simplicial complexes $L_1,\ldots,L_9$ from Table~\ref{table_L}.
\end{propos}

\begin{proof}
 Let $\alpha_1$, $\alpha_2$, and~$\alpha_3$ be the three $\rC_5$-orbits of vertices of~$K$. By Corollary~\ref{cor_complement} the complex~$K$ is $5$-neighborly, so $\alpha_1,\alpha_2,\alpha_3\in K$. Each simplicial complex~$\link(\alpha_i,K)$ is a $\rC_5$-invariant $\Z$-homology $3$-sphere with either~$5$ or~$10$ vertices. If $\link(\alpha_i,K)$ has $5$ vertices, then it is the boundary of a $4$-simplex. Suppose that every complex~$\link(\alpha_i,K)$ has $10$ vertices. Then it satisfies conditions~(1) and~(2) from Proposition~\ref{propos_L}. We need to prove that at least one of the three complexes~$\link(\alpha_i,K)$ satisfies condition~(3) as well; then by Proposition~\ref{propos_L} it is weakly $\rC_5$-isomorphic to one of the simplicial complexes $L_1,\ldots,L_9$. Assume that neither~$\link(\alpha_1,K)$ nor~$\link(\alpha_2,K)$ satisfies condition~(3). Then there exist $2$-simplices $\rho_1\in\link(\alpha_1,K)$ and $\rho_2\in\link(\alpha_2,K)$ such that $\rho_1\subset\alpha_3$ and $\rho_2\subset\alpha_3$. Since $\rho_1$ and~$\rho_2$ are $3$-element subsets of the $5$-element $\rC_5$-orbit~$\alpha_3$ it follows easily that there is an element $g\in\rC_5$ such that the union~$\rho_1\cup g\rho_2$ is the whole set~$\alpha_3$. Then $\alpha_1\cup\rho_1$ and $\alpha_2\cup g\rho_2$ are simplices of~$K$ whose union contains all vertices of~$K$. We arrive at a contradicition with complementarity, which completes the proof of the proposition.
\end{proof}

From Proposition~\ref{propos_C5_aux} it follows that any $\rC_5$-invariant $15$-vertex $\Z$-homology $8$-mani\-fold~$K$ with $H_*(K;\Z)\ncong H_*(S^8;\Z)$ contains a $\rC_5$-invariant subcomplex that is weakly $\rC_5$-isomorphic to either $\Delta^4*\partial\Delta^4$ or $\Delta^4*L_i$ for some~$i$. So we can get the list of all such homology manifolds by launching the program \texttt{find} for the subcomplex $R=\Delta^4*\partial\Delta^4$ and each of the subcomplexes $R=\Delta^4*L_i$ separately, and then taking the union of the obtained lists. In each of the cases the program completes in a reasonable amount of time. The result of the computation is as follows.
\begin{itemize}
 \item For $R=\Delta^4*\partial\Delta^4$ and each $R=\Delta^4*L_i$, where $i=1,3,7,8,9$, the program returns an empty list, so there are no $\rC_5$-invariant $15$-vertex $\Z$-homology manifolds with $H_*(K;\Z)\ncong H_*(S^8;\Z)$ that contain any of these subcomplexes.
 \item For $R=\Delta^4*L_2$, the program produces a list of $1850$ simplicial complexes, which are divided into $19$ groups, one of which consists of~$50$, and each of the others of $100$  isomorphic complexes. Representatives of the group of~$50$ isomorphic complexes in fact have a larger symmetry group~$\rA_5$ and are isomorphic to~$\HP^2_{15}(\rA_5)$. Representatives of  the other $18$ groups have symmetry group exactly~$\rC_5$.
 \item For $R=\Delta^4*L_4$, the program produces a list of $120$ simplicial complexes, which are divided into $12$ groups, each consisting of~$10$ isomorphic complexes with symmetry group  exactly~$\rC_5$.
 \item For $R=\Delta^4*L_5$, the program produces a list of $40$ simplicial complexes, which are divided into $4$ groups, each consisting of~$10$ isomorphic complexes with symmetry group  exactly~$\rC_5$.
 \item For $R=\Delta^4*L_6$, the program produces a list of $265$ simplicial complexes, which are divided into $19$ groups, one of which consists of~$5$, ten of~$10$, and eight of~$20$ isomorphic complexes. Representatives of the group of~$5$ isomorphic complexes in fact have a larger symmetry group~$\rA_5$ and are isomorphic to~$\HP^2_{15}(\rA_5)$. Representatives of  the other $18$ groups have symmetry group exactly~$\rC_5$.
\end{itemize}

Taking the union of the obtained lists and eliminating repetitions (up to isomorphism), we arrive at a list of  $26$ pairwise non-isomorphic simplicial complexes with symmetry group exactly~$\rC_5$. From Proposition~\ref{propos_C5_aux} it follows that all $\rC_5$-invariant $15$-vertex $\Z$-homology $8$-manifold with $H_*(K;\Z)\ncong H_*(S^8;\Z)$ (except for $\HP^2_{15}(\rA_5)$) are in the list. On the other hand, using the programs \texttt{check} and \texttt{fvect}, we check that the obtained $26$ simplicial complexes are combinatorial $8$-manifolds not homeomorphic to~$S^8$ and hence (by Theorem~\ref{thm_BK}) are manifolds like the quaternionic projective plane. We denote them by $\HP^2_{15}(\rC_5,1),\ldots,\HP^2_{15}(\rC_5,26)$. Thus, we obtain the following proposition.

\begin{propos}\label{propos_C5}
 Up to isomorphism, there exist exactly $27$ \ $15$-vertex $\Z$-homology $8$-manifolds~$K$ such that $H_*(K;\Z)$ is not isomorphic to~$H_*(S^8;\Z)$ and $\Sym(K)$ contains an element of order~$5$. All of them are combinatorial manifolds. One of these combinatorial manifolds is the Brehm--K\"uhnel triangulation~$\HP^2_{15}(\rA_5)$; the symmetry group of each of the other $26$ combinatorial manifolds $\HP^2_{15}(\rC_5,1),\ldots,\HP^2_{15}(\rC_5,26)$ is isomorphic to~$\rC_5$.
\end{propos}

\begin{remark}\label{remark_3links}
 Once the list of triangulations with symmetry group~$\rC_5$ is found, it is interesting to see what the links of the three $\rC_5$-orbits~$\alpha_1$, $\alpha_2$, and~$\alpha_3$ are in each of the combinatorial manifolds $\HP^2_{15}(\rC_5,1),\ldots,\HP^2_{15}(\rC_5,26)$, and~$\HP^2_{15}(\rA_5)$. From Proposition~\ref{propos_C5_aux} and the computer results mentioned above it follows that in each of the combinatorial manifolds the link of at least one of the three orbits is in the list~$L_2$, $L_4$, $L_5$, $L_6$. In fact, it turns out that, for some $8$ combinatorial manifolds ($\HP^2_{15}(\rC_5,19),\ldots,\HP^2_{15}(\rC_5,26)$ in our notation), the links of all the three orbits are in the list~$L_2$, $L_4$, $L_5$, $L_6$, and for each of  the other $19$ combinatorial manifolds $\HP^2_{15}(\rC_5,1),\ldots,\HP^2_{15}(\rC_5,18)$, and~$\HP^2_{15}(\rA_5)$, the links of exactly two orbits are in this list and the link of the third orbit is weakly $\rC_5$-isomorphic to one of the four complexes~$M_1,\ldots,M_4$ from Table~\ref{table_M}. It is convenient for us to postpone specifying explicit link types for each of the combinatorial manifolds to Section~\ref{section_G}, see Fig.~\ref{fig_C5}.
\end{remark}

\begin{table}
\caption{Simplicial $3$-spheres~$M_1,\ldots,M_{4}$}\label{table_M}
 \begin{tabular}{|l|l|}
  \hline
  & \textbf{Representatives of $\rC_5$-orbits of $3$-simplices}\\
  \hline
  $M_1$ & $\{1,2,3,6\},\ \{1,2,5,6\},\ \{1,3,6,7\},\ \{2,3,6,8\},\ \{1,6,7,9\},\ \{3,6,7,9\}$\\
  \hline
  $M_2$ & $\{1,2,3,6\},\ \{1,2,5,6\},\ \{2,4,6,7\},\ \{2,3,6,8\},\ \{2,4,6,8\},\ \{3,6,7,9\},\ \{5,6,7,9\}$\\
  \hline
  $M_3$ & $\{1,2,3,6\},\ \{1,2,5,6\},\ \{1,3,6,7\},\ \{1,5,6,7\},\ \{4,5,6,8\},\ \{4,6,7,8\},\ \{5,6,7,8\}$\\
  \hline
  $M_4$ & $\{1,2,3,6\},\ \{1,4,5,6\},\ \{1,2,6,7\},\ \{1,3,6,8\},\ \{4,5,6,8\},\ \{1,6,7,9\},\ \{2,6,7,9\}$\\
  \hline
 \end{tabular}
\end{table}

\section{Classification in the case of symmetry group of order at least four}\label{section_classify}

{\sloppy
\begin{propos}\label{propos_class}
Suppose that $K$ is a $15$-vertex $\Z$-homology $8$-manifold such that $H_*(K;\Z)\ncong H_*(S^8;\Z)$ and $|\Sym(K)|\ge 4$. Then $K$ is isomorphic to one of the following $75$ combinatorial $8$-manifolds from Sections~\ref{section_>=7}, \ref{section_6}, and~\ref{section_5}:
\begin{itemize}
 \item $\HP^2_{15}(\rA_5)$, $\HP^2_{15}(\rA_4,1)$, $\HP^2_{15}(\rA_4,2)$, $\HP^2_{15}(\rC_6\times\rC_2)$,
 \item $\HP^2_{15}(\rC_7,1),\ldots,\HP^2_{15}(\rC_7,19)$,
 \item $\HP^2_{15}(\rS_3,1),\ldots,\HP^2_{15}(\rS_3,12)$,
 \item $\HP^2_{15}(\rC_6,1),\ldots,\HP^2_{15}(\rC_6,14)$,
 \item $\HP^2_{15}(\rC_5,1),\ldots,\HP^2_{15}(\rC_5,26)$.
\end{itemize}
\end{propos}}

\begin{proof}
 First, assume that $|\Sym(K)|$ is divisible by~$4$. Then $\Sym(K)$ contains a subgroup isomorphic to either~$\rC_4$ or~$\rC_2\times\rC_2$. Hence, by Propositions~\ref{propos_4} and~\ref{propos_C2xC2}, $K$ is isomorphic to one of the four combinatorial manifolds~$\HP^2_{15}(\rA_5)$, $\HP^2_{15}(\rA_4,1)$, $\HP^2_{15}(\rA_4,2)$, and~$\HP^2_{15}(\rC_6\times\rC_2)$.

Second, assume that $|\Sym(K)|$ is divisible by~$7$, that is, $\Sym(K)$ contains an element of order~$7$. By Corollary~\ref{cor_7} this element fixes exactly one vertex of~$K$. We showed in Section~\ref{section_>=7} that $K$ is then isomorphic to one of the combinatorial manifolds $\HP^2_{15}(\rC_7,k)$, where $k=1,\ldots,19$.

 Third, assume that $|\Sym(K)|$ is divisible by~$5$, that is, $\Sym(K)$ contains an element of order~$5$. Then, by Proposition~\ref{propos_C5}, $K$ is isomorphic to either~$\HP^2_{15}(\rA_5)$ or one of the combinatorial manifolds $\HP^2_{15}(\rC_5,k)$, where $k=1,\ldots,26$.

Finally, assume that $|\Sym(K)|$ is divisible by none of the numbers~$4$, $5$, and~$7$. By Corollaries~\ref{cor_9} and~\ref{cor_p>7} we have that $|\Sym(K)|$ is also not divisible by~$9$ or any prime greater than~$7$. Since $|\Sym(K)|\ge 4$, it follows that $|\Sym(K)|=6$. Hence $\Sym(K)$ is isomorphic to either~$\rS_3$ or~$\rC_6$. By Propositions~\ref{propos_S3} and~\ref{propos_C6} we obtain that $K$ is isomorphic to one of the combinatorial manifolds~$\HP^2_{15}(\rS_3,k)$, where $k=1,\ldots,12$, or~$\HP^2_{15}(\rC_6,k)$, where $k=1,\ldots,14$.
 The proposition follows.
 \end{proof}

 \section{Triple flip graphs}\label{section_G}

\subsection{The triple flip graph~$\CG$} Let $K$ be a $15$-vertex combinatorial $8$-manifold like the quaternionic projective plane and $V$ the vertex set of~$K$.

 \begin{defin}
 A triple $(\Delta_1,\Delta_2,\Delta_3)$ of $4$-simplices of~$K$ is called a \textit{distinguished triple} if $V=\Delta_1\sqcup\Delta_2\sqcup\Delta_3$ and
 $$
 \link(\Delta_1,K)=\partial\Delta_2,\qquad
 \link(\Delta_2,K)=\partial\Delta_3,\qquad
 \link(\Delta_3,K)=\partial\Delta_1.
 $$
 The corresponding subcomplex
 \begin{equation}\label{eq_disting}
 J=(\Delta_1*\partial\Delta_2)\cup(\Delta_2*\partial\Delta_3)\cup(\Delta_3*\partial\Delta_1) \end{equation}
 of~$K$ is called a \textit{distinguished subcomplex}.
 \end{defin}

 Distinguished triples that differ in cyclic permutations of the simplices~$\Delta_1$, $\Delta_2$, and $\Delta_3$ are considered to be the same distinguished triple.

 \begin{defin}
  A \textit{triple flip} associated with a distinguished triple~$(\Delta_1,\Delta_2,\Delta_3)$ is the operation consisting in replacement of the distinguished subcomplex~\eqref{eq_disting} with  the subcomplex
  $$
  (\partial\Delta_1*\Delta_2)\cup(\partial\Delta_2*\Delta_3)\cup(\partial\Delta_3*\Delta_1).
  $$
  Two triple flips $\mathfrak{t}\colon K_1\rightsquigarrow K_2$ and $\mathfrak{t'}\colon K'_1\rightsquigarrow K'_2$ are said to be \textit{isomorphic} if there exists an isomorphism $K_1\to K_1'$ that takes the distinguished triple associated with~$\mathfrak{t}$ to the distinguished triple associated with~$\mathfrak{t}'$. A triple flip~$\mathfrak{t}$ is called \textit{self-inverse} if it is isomorphic to its inverse triple flip~$\mathfrak{t}^{-1}$.
 \end{defin}

 Note that a triple flip $\mathfrak{t}\colon K_1\rightsquigarrow K_2$ such that $K_1$ is isomorphic to $K_2$ may or may not be self-inverse.

 Brehm and K\"uhnel~\cite[Lemma~1]{BrKu92} showed that a triple flip always yields a well-defined simplicial complex, which is a combinatorial manifold PL homeomorphic to~$K$. So triple flips can be used to move between different $15$-vertex combinatorial $8$-manifolds like the quaternionic projective plane and to produce new examples of them, starting from known ones.

 \begin{remark}
 As we have mentioned in the Introduction, Lutz~\cite{Lut05} used the program \texttt{BISTELLAR}, which performs usual bistellar flips, to construct three new $15$-vertex triangulations of~$\HP^2$, starting from the three triangulations due to Brehm and K\"uhnel. However, bistellar flips cannot transform one $15$-vertex triangulation of~$\HP^2$ to another without increasing the number of vertices in the process. So, using bistellar flips, one needs first to add a new vertex at the barycentre of a $8$-simplex of a $15$-vertex triangulation, and then perform bistellar flips with the obtained $16$-vertex triangulation, hoping that the program will reduce the number of vertices. The advantage of triple flips is that they do not take us out of the class of $15$-vertex triangulations.
 \end{remark}

\begin{defin}
 We conveniently construct the following \textit{triple flip graph}~$\CG$:
 \begin{itemize}
  \item vertices of~$\CG$ are isomorphism classes of $15$-vertex combinatorial $8$-manifolds like the quaternionic projective plane,
  \item edges of~$\CG$ are isomorphism classes of triple flips.
 \end{itemize}
\end{defin}

\begin{remark}
 Actually, $\CG$ is a multigraph rather than a graph, since it may contain multiple edges and loops. A loop in~$\CG$ corresponds to a triple flip $\mathfrak{t}\colon K_1\rightsquigarrow K_2$ such that $K_1$ is isomorphic to~$K_2$. Loops in~$\mathcal{G}$ corresponding to self-inverse and non-self-inverse triple flips~$\mathfrak{t}$ will be called \textit{self-inverse} and \textit{non-self-inverse}, respectively. The two types of loops contribute differently to the degrees of  vertices of~$\CG$. Namely, we conveniently agree that a self-inverse loop contributes~$1$ and a non-self-inverse loop contributes~$2$ to the degree of the corresponding vertex. With this agreement, the degree of a vertex~$K$ of~$\CG$ is the number of $\Sym(K)$-orbits of distinguished triples in~$K$. In all figures throughout the paper, we use full lines for non-self-inverse loops and dashed lines for self-inverse loops.
\end{remark}

For a $15$-vertex combinatorial $8$-manifold~$K$ like the quaternionic projective plane, there is an effective way to find all distinguished triples in it, see~\cite[Proposition~8.1]{Gai22}. This gives an algorithm that computes the connected component of the given combinatorial manifold~$K$ in~$\CG$. We have implemented this algorithm as a C++ program, see the program \texttt{triple\_flip\_graph} in~\cite{Gai-prog}. The first result obtained by using this program is the following description of the connected component in~$\CG$ that contains the three Brehm--K\"uhnel triangulations of~$\HP^2$. (In~\cite{BrKu92} it was shown that these three triangulations can be obtained from each other by triple flips, so they lie in the same connected component.)

\begin{propos}
The connected component~$\CG_0$ of the three Brehm--K\"uhnel triangulations~$\HP^2_{15}(\rA_5)$, $\HP^2_{15}(\rA_4,1)$, and~$\HP^2_{15}(\rS_3,1)$ in the graph~$\CG$ is shown in Fig.~\ref{fig_G_component}. It has $22$ vertices; the symmetry groups of the corresponding $15$-vertex triangulations of~$\HP^2$ are as follows:
\begin{itemize}
 \item $1$ triangulation~$\HP^2_{15}(\rA_5)$ has symmetry group~$\rA_5$,
 \item $2$ triangulations~$\HP^2_{15}(\rA_4,1)$ and\/~$\HP^2_{15}(\rA_4,2)$ have symmetry group~$\rA_4$,
 \item $2$ triangulations~$\HP^2_{15}(\rS_3,1)$ and\/~$\HP^2_{15}(\rS_3,2)$ have symmetry group~$\rS_3$,
 \item {\sloppy $3$ triangulations~$\HP^2_{15}(\rC_3,1)$, $\HP^2_{15}(\rC_3,2)$, and~$\HP^2_{15}(\rC_3,3)$ have symmetry group~$\rC_3$,

 }
 \item $7$ triangulations~$\HP^2_{15}(\rC_2,1), \ldots,\HP^2_{15}(\rC_2,7)$ have symmetry group~$\rC_2$,
 \item and $7$ triangulations~$\HP^2_{15}(\rC_1,1), \ldots,\HP^2_{15}(\rC_1,7)$ have trivial symmetry group~$\rC_1$.
\end{itemize}
\end{propos}

\begin{figure}[t]
  \tiny
 \begin{tikzpicture}
 \tikzmath{\minrad=.65cm; \innsep=0pt;\step=2.1;\dx=.8;\dy=.35;\dz=.6;}
  \draw (0, 5*\step) node (1) [circle, inner sep = \innsep, minimum size=\minrad, draw] {$\rA_5$};
  \draw (0, 4*\step) node (2) [circle, inner sep = \innsep, minimum size=\minrad, draw] {$\rA_4,\!1$};
  \draw (\step, 4*\step) node (3) [circle, inner sep = \innsep, minimum size=\minrad, draw] {$\rC_3,\!1$};
  \draw (2*\step, 4*\step) node (5) [circle, inner sep = \innsep, minimum size=\minrad, draw] {$\rC_2,\!1$};
  \draw (3*\step, 4*\step) node (7) [circle, inner sep = \innsep, minimum size=\minrad, draw] {$\rC_3,\!2$};
  \draw (4*\step, 4*\step) node (13) [circle, inner sep = \innsep, minimum size=\minrad, draw] {$\rA_4,\!2$};

  \draw (0, 3*\step) node (4) [circle, inner sep = \innsep, minimum size=\minrad, draw] {$\rS_3,\!1$};
  \draw (\step, 3*\step) node (6) [circle, inner sep = \innsep, minimum size=\minrad, draw] {$\rC_1,\!1$};
  \draw (2*\step,3*\step) node (8) [circle, inner sep = \innsep, minimum size=\minrad, draw] {$\rC_1,\!2$};
  \draw (3*\step, 3*\step) node (12) [circle, inner sep = \innsep, minimum size=\minrad, draw] {$\rC_3,\!3$};
  \draw (4*\step, 3*\step) node (17) [circle, inner sep = \innsep, minimum size=\minrad, draw] {$\rC_1,\!3$};
  \draw (5*\step, 3*\step) node (21) [circle, inner sep = \innsep, minimum size=\minrad, draw] {$\rC_1,\!4$};
  \draw (6*\step, 3*\step) node (22) [circle, inner sep = \innsep, minimum size=\minrad, draw] {$\rS_3,\!2$};

  \draw (1.5*\step, 2*\step+\dz) node (10) [circle, inner sep = \innsep, minimum size=\minrad, draw] {$\rC_2,\!2$};
\draw (3*\step, 2*\step+\dz) node (16) [circle, inner sep = \innsep, minimum size=\minrad, draw] {$\rC_1,\!5$};
\draw (4.5*\step, 2*\step+\dz) node (18) [circle, inner sep = \innsep, minimum size=\minrad, draw] {$\rC_2,\!3$};

  \draw (\step-\dx, .5*\step-\dy) node (11) [circle, inner sep = \innsep, minimum size=\minrad, draw] {$\rC_2,\!4$};
  \draw (3*\step-\dx, .5*\step-\dy) node (15) [circle, inner sep = \innsep, minimum size=\minrad, draw] {$\rC_1,\!6$};
  \draw (5*\step-\dx, .5*\step-\dy) node (20) [circle, inner sep = \innsep, minimum size=\minrad, draw] {$\rC_2,\!6$};

  \draw (\step+\dx, .5*\step+\dy) node (9) [circle, inner sep = \innsep, minimum size=\minrad, draw] {$\rC_2,\!5$};
  \draw (3*\step+\dx, .5*\step+\dy) node (14) [circle, inner sep = \innsep, minimum size=\minrad, draw] {$\rC_1,\!7$};
  \draw (5*\step+\dx, .5*\step+\dy) node (19) [circle, inner sep = \innsep, minimum size=\minrad, draw] {$\rC_2,\!7$};

  \foreach \from/\to in {1/2,2/3,2/4,3/5,3/6,4/6,5/7,5/8,6/9,6/10,6/11,7/12,7/13,8/10,8/12,8/14, 8/15,8/16,9/11,9/14,9/16,10/14,10/15,11/15,11/16,12/17,14/15,14/17,14/18,14/19, 15/17,15/18,15/20,16/17,16/19,16/20,18/21,19/21,20/21,21/22}
    \draw (\from) -- (\to);
    \draw [double, double distance=4pt] (6)--(8);
    \draw [double, double distance=4pt] (17)--(21);
    \draw [dashed] (4) .. controls (-.5*\step,2.8*\step) and (-0.2*\step,2.5*\step) .. (4);
    \draw (6) .. controls (.8*\step,3.5*\step) and (.5*\step,3.2*\step) .. (6);
    \draw [dashed] (18) .. controls (4.5*\step-.2*\step, 2*\step+\dz+.5*\step) and (4.5*\step-.5*\step, 2*\step+\dz+.2*\step) .. (18);
 \end{tikzpicture}
 \caption{The connected component of~$\HP^2_{15}(\rA_5)$ in~$\CG$}\label{fig_G_component}
\end{figure}

\begin{remark}
 To save space in Fig.~\ref{fig_G_component}, we indicate the vertex corresponding to the triangulation~$\HP^2_{15}(G,k)$ by simply writing a pair~$G,k$.
\end{remark}

An amazing phenomenon is that other connected components of~$\CG$ seem to be arranged completely differently.

\begin{observe}
 Let $K$ be any of the $75$ combinatorial manifolds from Proposition~\ref{propos_class}, except for the $5$ combinatorial manifolds~$\HP^2_{15}(\rA_5)$, $\HP^2_{15}(\rA_4,1)$,  $\HP^2_{15}(\rA_4,2)$, $\HP^2_{15}(\rS_3,1)$, and~$\HP^2_{15}(\rS_3,2)$ that lie in~$\CG_0$. Let us start computing the connected component of~$K$ in~$\CG$ using the program \textnormal{\texttt{triple\_flip\_graph}}. Then the program produces a huge number (at least hundreds of thousands) of different triangulations in the connected component of~$K$ and does not finish the job in a reasonable time.
 \end{observe}

\subsection{Equivariant triple flip graphs} Suppose that $K$ is a $15$-vertex combinatorial manifold like the quaternionic projective plane, and $J_1,\ldots,J_k$ are distinguished subcomplexes of~$K$, that is, subcomplexes of the form~\eqref{eq_disting}. Recall that if $J_1,\ldots,J_k$ have pairwise disjoint interiors (i.\,e., no $8$-simplex belongs to two different subcomplexes~$J_i$), then we can perform the triple flips corresponding to $J_1,\ldots,J_k$ in any order, and the resulting combinatorial manifold is independent of this order, see~\cite[Corollary~7.4]{Gai22}.

\begin{defin}
 Suppose that $G$ is a subgroup of~$\Sym(K)$. Then the distinguished subcomplexes of~$K$ are divided into $G$-orbits. We say that the $G$-orbit of a distinguished subcomplex~$J$ is \textit{admissible} if any two subcomplexes~$g_1J$ and~$g_2J$ in this orbit either coincide or have disjoint interiors. In this case, we can simultaneously perform all the triple flips corresponding to the distinguished subcomplexes~$gJ$ with $g\in G$. This operation will be called a \textit{$G$-equivariant triple flip}. The combinatorial manifold resulting from such a flip will be again $G$-invariant and will be PL homeomorphic to~$K$. Two $G$-equivariant triple flips $\mathfrak{t}\colon K_1\rightsquigarrow K_2$ and $\mathfrak{t'}\colon K'_1\rightsquigarrow K'_2$, where $G$ is a subgroup of both~$\Sym(K_1)$ and~$\Sym(K_1')$, are said to be \textit{weakly $G$-isomorphic} if there exists a weak $G$-isomorphism $K_1\to K_1'$ that takes the $G$-orbit of distinguished subcomplexes associated with~$\mathfrak{t}$ to the $G$-orbit of distinguished subcomplexes associated with~$\mathfrak{t}'$. A $G$-equivariant triple flip~$\mathfrak{t}$ is called \textit{self-inverse} if it is weakly $G$-isomorphic to its inverse $G$-equivariant triple flip~$\mathfrak{t}^{-1}$.
\end{defin}

\begin{defin}
Let now $G$ be a subgroup of~$\rS_{15}$.
We conveniently construct the following \textit{$G$-equivariant triple flip graph}~$\CG_G$:
 \begin{itemize}
  \item vertices of~$\CG_G$ are weak $G$-isomorphism classes of $G$-invariant $15$-vertex combinatorial $8$-manifolds like the quaternionic projective plane on the vertex set $\{1,\ldots,15\}$,
  \item edges of~$\CG_G$ are weak $G$-isomorphism classes of $G$-equivariant triple flips.
 \end{itemize}
 \end{defin}

 Certainly, the graphs~$\CG_G$ corresponding to conjugate subgroups $G\subset \rS_{15}$ are isomorphic to each other. From Proposition~\ref{propos_class} it follows that the graphs~$\CG_G$ can be
 non-empty for only $11$ (up to conjugation) subgroups of~$\rS_{15}$, namely, for the groups~$\rA_5$, $\rA_4$, $\rC_7$, $\rS_3$, $\rC_6$, $\rC_5$, $\rC_3$, $\rC_2$, and~$\rC_1$ that are realized as subgroups of~$\rS_{15}$ as in Table~\ref{table_main}, and also for the group~$\rC_2\times\rC_2$, which can be a subgroup of~$\Sym(K)$, though cannot coincide with it. For all the subgroups~$G$ with $|G|>3$, Proposition~\ref{propos_class} provides a complete list of all isomorphism classes of $G$-invariant $15$-vertex combinatorial manifolds like the quaternionic projective plane. This allows to describe explicitly the graph~$\CG_G$ in each of these cases, using the same program \texttt{triple\_flip\_graph}. The result is as follows.

 \begin{enumerate}
  \item The graph~$\CG_{\rA_5}$ consists of one vertex~$\HP^2_{15}(\rA_5)$ and one self-inverse loop. Indeed, the Brehm--K\"uhnel triangulation~$\HP^2_{15}(\rA_5)$ contains exactly $5$ distinguished subcomplexes, which have pairwise disjoint interiors and form one $\rA_5$-orbit; if we perform all the $5$ triple flips simultaneously, then the resulting combinatorial manifold is again isomorphic to~$\HP^2_{15}(\rA_5)$, see~\cite[Section~3]{BrKu92}.
  \item The graph~$\CG_{\rA_4}$ is shown in Fig.~\ref{fig_A4}. Let us explain why the Brehm--K\"uhnel triangulations~$\HP^2_{15}(\rA_5)$ and~$\HP^2_{15}(\rA_4,1)$ are connected by two edges. The $\rA_5$-orbit of the $5$ distinguished subcomplexes of~$\HP^2_{15}(\rA_5)$ is divided into two $\rA_4$-orbits, one consisting of one and the other of four distinguished subcomplexes. If we perform either of the corresponding two $\rA_4$-equivariant triple flips, then we obtain the combinatorial manifold isomorphic to~$\HP^2_{15}(\rA_4,1)$, see~\cite[Section~3]{BrKu92}.
  \item The graph~$\CG_{\rC_6\times\rC_2}$ consists of one vertex~$\HP^2_{15}(\rC_6\times\rC_2)$ without any edges. The combinatorial manifold~$\HP^2_{15}(\rC_6\times\rC_2)$ contains $24$ distinguished subcomplexes, which are divided into two $(\rC_6\times\rC_2)$-orbits. Nevertheless, neither of these two orbits is admissible.
  \item The graph~$\CG_{\rC_7}$ has two connected components, see Fig.~\ref{fig_C7}. In this figure we denote each combinatorial manifold~$\HP^2_{15}(\rC_7,k)$ simply by~$k$.
  \item The graph~$\CG_{\rS_3}$ has $5$ connected components, see Fig.~\ref{fig_S3}. In this figure we denote each combinatorial manifold~$\HP^2_{15}(\rS_3,k)$ simply by~$k$ and $\rA_5$ stands for the Brehm--K\"uhnel manifold~$\HP^2_{15}(\rA_5)$ with the symmetry group $\rA_5\supset\rS_3$.
  \item The graph~$\CG_{\rC_6}$ has $4$ connected components, see Fig.~\ref{fig_C6}. In this figure we denote each combinatorial manifold~$\HP^2_{15}(\rC_6,k)$ simply by~$k$. The three vertices of~$\CG_{\rC_6}$ denoted by~$\rC_6\times\rC_2$ correspond to the three pairs~$\bigl(\HP^2_{15}(\rC_6\times\rC_2),H_i\bigr)$, where $H_1$, $H_2$, and~$H_3$ are the three subgroups of~$\rC_6\times\rC_2$ that are isomorphic to~$\rC_6$. Note that, endowing the combinatorial manifold $\HP^2_{15}(\rC_6\times\rC_2)$  with the three structures of $\rC_6$-simplicial complex corresponding to the subgroups~$H_1$, $H_2$, and~$H_3$, we get three $\rC_6$-simplicial complexes that are not weakly $\rC_6$-isomorphic to each other. So they give three different vertices of~$\CG_{\rC_6}$. This phenomenon does not occur in the case of the vertex~$\HP^2_{15}(\rA_5)$ in~$\CG_{\rA_4}$ (respectively, $\CG_{\rS_3}$ or $\CG_{\rC_5}$), since all subgroups of~$\rA_5$ that are isomorphic to~$\rA_4$ (respectively, $\rS_3$ or~$\rC_5$) are conjugate to each other.
  \item The graph~$\CG_{\rC_5}$ has two connected components, see Fig.~\ref{fig_C5}. One of them consists of one vertex~$\HP^2_{15}(\rA_5)$ and one self-inverse loop, and the other has $26$ vertices $\HP^2_{15}(\rC_5,1),\ldots,\HP^2_{15}(\rC_5,26)$, which we denote in the figure simply by~$1,\ldots,26$. It is interesting to compare the positions of vertices in the graph~$\CG_{\rC_5}$ with the link types of the $\rC_5$-orbits in the corresponding combinatorial manifolds, cf.~Remark~\ref{remark_3links}. So for each of the combinatorial manifolds~$\HP^2_{15}(\rC_5,k)$, we specify in Fig.~\ref{fig_C5} the links of the three $4$-simplices $\alpha_1$, $\alpha_2$, and~$\alpha_3$ that are $\rC_5$-orbits. Here $L_1,\ldots,L_9,M_1,\ldots,M_4$ are the simplicial $3$-spheres from Tables~\ref{table_L} and~\ref{table_M}.
  \item The graph~$\CG_{\rC_2\times\rC_2}$ coincides with the disjoint union of the graphs~$\CG_{\rA_4}$ and~$\CG_{\rC_6\times\rC_2}$.
\end{enumerate}

\begin{figure}[t]
\begin{tikzpicture}
 \tiny \tikzmath{\minrad=.65cm; \innsep=0pt;\step=2.1;}
 \draw (0, 0) node (1) [circle, inner sep = \innsep, minimum size=\minrad, draw] {$\rA_5$};
 \draw (\step, 0) node (2) [circle, inner sep = \innsep, minimum size=\minrad, draw] {$\rA_4,\!1$};
 \draw (2*\step, 0) node (3) [circle, inner sep = \innsep, minimum size=\minrad, draw] {$\rA_4,\!2$};
 \draw (2)--(3);
 \draw [double, double distance = 4pt] (1)--(2);
\end{tikzpicture}
 \caption{The graph~$\CG_{\rA_4}$}\label{fig_A4}
% \end{figure}
%
% \begin{figure}[t]
\vspace{15mm}

 \tiny
 \begin{tikzpicture}
  \tikzmath{\minrad=.4cm; \innsep=0pt;\step=1.5;\ydist=-2.5;}
  \begin{scope}
  \draw (-1.414*\step,0) node (1) [circle, inner sep = \innsep, minimum size=\minrad, draw] {$1$};
  \draw (-.707*\step,.707*\step) node (2) [circle, inner sep = \innsep, minimum size=\minrad, draw] {$2$};
  \draw (-.707*\step,-.707*\step) node (3) [circle, inner sep = \innsep, minimum size=\minrad, draw] {$3$};
  \draw (0,0) node (4) [circle, inner sep = \innsep, minimum size=\minrad, draw] {$4$};
  \draw (\step,0) node (5) [circle, inner sep = \innsep, minimum size=\minrad, draw] {$5$};
  \draw (2*\step,0) node (6) [circle, inner sep = \innsep, minimum size=\minrad, draw] {$6$};
  \draw (3*\step,0) node (7) [circle, inner sep = \innsep, minimum size=\minrad, draw] {$7$};
  \draw (4*\step,0) node (8) [circle, inner sep = \innsep, minimum size=\minrad, draw] {$8$};
  \draw (4.707*\step,.707*\step) node (9) [circle, inner sep = \innsep, minimum size=\minrad, draw] {$9$};
  \draw (4.707*\step,-.707*\step) node (10) [circle, inner sep = \innsep, minimum size=\minrad, draw] {$10$};
  \foreach \from/\to in {1/2,1/3,2/4,3/4,4/5,5/6,6/7,7/8,8/9,8/10}
    \draw (\from) -- (\to);
  \end{scope}
  \begin{scope}
    \draw (-\step,\ydist) node (11) [circle, inner sep = \innsep, minimum size=\minrad, draw] {$11$};
    \draw (0,\ydist) node (12) [circle, inner sep = \innsep, minimum size=\minrad, draw] {$12$};
    \draw (\step,\ydist) node (13) [circle, inner sep = \innsep, minimum size=\minrad, draw] {$13$};
    \draw (2*\step,\ydist) node (14) [circle, inner sep = \innsep, minimum size=\minrad, draw] {$14$};
    \draw (3*\step,\ydist) node (15) [circle, inner sep = \innsep, minimum size=\minrad, draw] {$15$};
    \draw (4*\step,\ydist) node (16) [circle, inner sep = \innsep, minimum size=\minrad, draw] {$16$};
    \draw (0,\ydist-\step) node (17) [circle, inner sep = \innsep, minimum size=\minrad, draw] {$17$};
    \draw (\step,\ydist-\step) node (18) [circle, inner sep = \innsep, minimum size=\minrad, draw] {$18$};
    \draw (2*\step,\ydist-\step) node (19) [circle, inner sep = \innsep, minimum size=\minrad, draw] {$19$};
    \foreach \from/\to in {11/12,12/13,13/14,14/15,15/16,12/17,13/18,14/19,17/18,18/19}
    \draw (\from) -- (\to);
  \end{scope}
 \end{tikzpicture}
 \caption{The graph~$\CG_{\rC_7}$}\label{fig_C7}
% \end{figure}
%
% \begin{figure}[t]
\vspace{15mm}

 \tiny
 \begin{tikzpicture}
  \tikzmath{\minrad=.4cm; \innsep=0pt;\step=1.5;}
  \draw (0, 2*\step) node (A5) [circle, inner sep = \innsep, minimum size=\minrad, draw] {$\rA_5$};
  \draw (\step, 2*\step) node (1) [circle, inner sep = \innsep, minimum size=\minrad, draw] {$1$};
  \draw (2*\step, 2*\step) node (2) [circle, inner sep = \innsep, minimum size=\minrad, draw] {$2$};
  \draw (\step, \step) node (3) [circle, inner sep = \innsep, minimum size=\minrad, draw] {$3$};
  \draw (2*\step, \step) node (4) [circle, inner sep = \innsep, minimum size=\minrad, draw] {$4$};
  \draw (\step, 0) node (6) [circle, inner sep = \innsep, minimum size=\minrad, draw] {$6$};
  \draw (2*\step, 0) node (5) [circle, inner sep = \innsep, minimum size=\minrad, draw] {$5$};

  \draw (3.5*\step, 2*\step) node (7) [circle, inner sep = \innsep, minimum size=\minrad, draw] {$7$};
  \draw (4.5*\step, 2*\step) node (8) [circle, inner sep = \innsep, minimum size=\minrad, draw] {$8$};
  \draw (5.5*\step, 2*\step) node (9) [circle, inner sep = \innsep, minimum size=\minrad, draw] {$9$};
  \draw (3.5*\step, \step) node (10) [circle, inner sep = \innsep, minimum size=\minrad, draw] {$10$};
  \draw (4.5*\step, \step) node (11) [circle, inner sep = \innsep, minimum size=\minrad, draw] {$11$};
  \draw (3.5*\step, 0) node (12) [circle, inner sep = \innsep, minimum size=\minrad, draw] {$12$};

  \foreach \from/\to in {1/2, 3/4, 3/6, 4/5, 5/6, 7/8, 8/9, 10/11}
    \draw (\from) -- (\to);
    \draw [double, double distance = 3pt] (A5) -- (1);
 \end{tikzpicture}
 \caption{The graph~$\CG_{\rS_3}$}\label{fig_S3}
% \end{figure}
%
% \begin{figure}
\vspace{15mm}

\tiny
 \begin{tikzpicture}
  \tikzmath{\minrad=.4cm; \innsep=0pt;\step=1.5;}
  \draw (0, 2*\step) node (121) [circle, inner sep = \innsep, minimum size=\minrad, draw] {$\rC_6\!\!\times\!\!\rC_2$};
  \draw (\step, 2*\step) node (1) [circle, inner sep = \innsep, minimum size=\minrad, draw] {$1$};
  \draw (2*\step, 2*\step) node (2) [circle, inner sep = \innsep, minimum size=\minrad, draw] {$2$};
  \draw (0, \step) node (3) [circle, inner sep = \innsep, minimum size=\minrad, draw] {$3$};
  \draw (\step, \step) node (4) [circle, inner sep = \innsep, minimum size=\minrad, draw] {$4$};
  \draw (2*\step, \step) node (5) [circle, inner sep = \innsep, minimum size=\minrad, draw] {$5$};
  \draw (\step, 0) node (6) [circle, inner sep = \innsep, minimum size=\minrad, draw] {$6$};
  \draw (2*\step, 0) node (7) [circle, inner sep = \innsep, minimum size=\minrad, draw] {$7$};

  \draw (3.5*\step, 1.6*\step) node (122) [circle, inner sep = \innsep, minimum size=\minrad, draw] {$\rC_6\!\!\times\!\!\rC_2$};
  \draw (4.5*\step, 1.6*\step) node (8) [circle, inner sep = \innsep, minimum size=\minrad, draw] {$8$};
  \draw (5.5*\step, 1.6*\step) node (9) [circle, inner sep = \innsep, minimum size=\minrad, draw] {$9$};
  \draw (3.5*\step, .4*\step) node (123) [circle, inner sep = \innsep, minimum size=\minrad, draw] {$\rC_6\!\!\times\!\!\rC_2$};
  \draw (4.5*\step, .4*\step) node (10) [circle, inner sep = \innsep, minimum size=\minrad, draw] {$10$};

  \draw (6.7*\step, \step) node (11) [circle, inner sep = \innsep, minimum size=\minrad, draw] {$11$};
  \draw (7.407*\step, 1.707*\step) node (12) [circle, inner sep = \innsep, minimum size=\minrad, draw] {$12$};
  \draw (7.407*\step, .293*\step) node (13) [circle, inner sep = \innsep, minimum size=\minrad, draw] {$13$};
  \draw (8.114*\step, \step) node (14) [circle, inner sep = \innsep, minimum size=\minrad, draw] {$14$};
  \foreach \from/\to in {121/1,1/2,121/3,1/4,2/5,3/4,4/5,4/6,5/7,6/7,122/8,8/9,123/10, 11/12,11/13,12/13,12/14,13/14}
    \draw (\from) -- (\to);
    \draw [dashed] (11) .. controls (6.7*\step-.75*\step,\step-.3*\step) and (6.7*\step-.75*\step,\step+.3*\step) .. (11);
    \draw [dashed] (14) .. controls (8.114*\step+.75*\step,\step-.3*\step) and (8.114*\step+.75*\step,\step+.3*\step) .. (14);
  \end{tikzpicture}
 \caption{The graph~$\CG_{\rC_6}$}\label{fig_C6}
 \end{figure}

 \begin{figure}[p]
\tiny
 \begin{tikzpicture}
  \tikzmath{\minrad=.4cm; \innsep=.5pt;\step=3;\dx=1; \dy=.75;}

  \draw(1.5,4.2) node (A5) [regular polygon, regular polygon sides = 5, inner sep = \innsep, minimum size=\minrad, draw] {$A_5$};
  \draw [dashed] (A5) .. controls (2.7,4.7) and (2.7,3.7) .. (A5);

 \draw (0, \step) node (1) [double,circle, inner sep = \innsep, minimum size=\minrad, draw] {$1$};
  \draw (\step, \step) node (2) [circle, inner sep = \innsep, minimum size=\minrad, draw] {$2$};
  \draw (2*\step, \step) node (3) [circle, inner sep = \innsep, minimum size=\minrad, draw] {$3$};
  \draw (3*\step, \step) node (4) [circle, inner sep = \innsep, minimum size=\minrad, draw] {$4$};
  \draw (4*\step, \step) node (5) [double, circle, inner sep = \innsep, minimum size=\minrad, draw] {$5$};
  \draw (0, 0) node (6) [double,circle, inner sep = \innsep, minimum size=\minrad, draw] {$6$};
  \draw (\step, 0) node (7) [circle, inner sep = \innsep, minimum size=\minrad, draw] {$7$};
  \draw (2*\step, 0) node (8) [circle, inner sep = \innsep, minimum size=\minrad, draw] {$8$};
  \draw (3*\step, 0) node (9) [circle, inner sep = \innsep, minimum size=\minrad, draw] {$9$};
  \draw (4*\step, 0) node (10) [double, circle, inner sep = \innsep, minimum size=\minrad, draw] {$10$};
    \draw (-\dx, \step-\dy) node (11) [double,rectangle, inner sep = \innsep, minimum size=\minrad, draw] {$11$};
  \draw (\step-\dx, \step-\dy) node (12) [rectangle, inner sep = \innsep, minimum size=\minrad, draw] {$12$};
  \draw (2*\step-\dx, \step-\dy) node (13) [rectangle, inner sep = \innsep, minimum size=\minrad, draw] {$13$};
  \draw (3*\step-\dx, \step-\dy) node (14) [rectangle, inner sep = \innsep, minimum size=\minrad, draw] {$14$};

    \draw (\step+\dx, \dy) node (15) [rectangle, inner sep = \innsep, minimum size=\minrad, draw] {$15$};
  \draw (2*\step+\dx, \dy) node (16) [rectangle, inner sep = \innsep, minimum size=\minrad, draw] {$16$};
  \draw (3*\step+\dx, \dy) node (17) [rectangle, inner sep = \innsep, minimum size=\minrad, draw] {$17$};
  \draw (4*\step+\dx, \dy) node (18) [double,rectangle, inner sep = \innsep, minimum size=\minrad, draw] {$18$};

  \draw (-\dx, -\dy) node (19) [double,diamond, inner sep = \innsep, minimum size=\minrad, draw] {$19$};
  \draw (\step-\dx, -\dy) node (20) [diamond, inner sep = \innsep, minimum size=\minrad, draw] {$20$};
  \draw (2*\step-\dx, -\dy) node (21) [diamond, inner sep = \innsep, minimum size=\minrad, draw] {$21$};
  \draw (3*\step-\dx, -\dy) node (22) [diamond, inner sep = \innsep, minimum size=\minrad, draw] {$22$};

  \draw (\step+\dx, \step+\dy) node (23) [diamond, inner sep = \innsep, minimum size=\minrad, draw] {$23$};
  \draw (2*\step+\dx, \step+\dy) node (24) [diamond, inner sep = \innsep, minimum size=\minrad, draw] {$24$};
  \draw (3*\step+\dx, \step+\dy) node (25) [diamond, inner sep = \innsep, minimum size=\minrad, draw] {$25$};
  \draw (4*\step+\dx, \step+\dy) node (26) [double,diamond, inner sep = \innsep, minimum size=\minrad, draw] {$26$};
  \foreach \from/\to in {1/2, 2/3, 3/4, 4/5, 6/7, 7/8, 8/9, 9/10, 11/12, 12/13, 13/14, 15/16, 16/17, 17/18, 19/20, 20/21, 21/22, 23/24, 24/25, 25/26, 1/6, 2/7, 3/8, 4/9, 5/10, 11/19, 12/20, 13/21, 14/22, 15/23, 16/24, 17/25, 18/26, 1/11, 2/12, 2/23, 3/13, 3/24, 4/14, 4/25, 5/26, 6/19, 7/15, 7/20, 8/16, 8/21, 9/17, 9/22, 10/18}
    \draw (\from) -- (\to);

   \normalsize
   \draw (0.3,-2.04) node {Link types:};
   \footnotesize
   \draw (2.2, -2) node [circle, minimum size=\minrad, draw] {};
   \draw (3.6, -2.04) node  {$L_2,L_6,M_1$};
   \draw (2.2, -2.6) node [double,circle, minimum size=\minrad, draw] {};
   \draw (3.6, -2.64) node  {$L_2,L_6,M_2$};
   \draw (5.7, -2) node [rectangle, minimum size=\minrad, draw] {};
   \draw (7.1, -2.04) node  {$L_2,L_4,M_3$};
   \draw (5.7, -2.6) node [double,rectangle, minimum size=\minrad, draw] {};
   \draw (7.1, -2.64) node  {$L_2,L_5,M_3$};
   \draw (9.2, -2) node [diamond, minimum size=\minrad, draw] {};
   \draw (10.6, -2.04) node  {$L_4,L_6,L_6$};
   \draw (9.2, -2.6) node [double,diamond, minimum size=\minrad, draw] {};
   \draw (10.6, -2.64) node  {$L_5,L_6,L_6$};
   \draw (2.2, -3.2) node [regular polygon, regular polygon sides = 5, minimum size=\minrad, draw] {};
   \draw (3.6, -3.24) node  {$L_2,L_6,M_4$};
  \end{tikzpicture}
  \vspace{-2mm}
 \caption{The graph~$\CG_{\rC_5}$}\label{fig_C5}
  \vspace{.5cm}
\end{figure}

\subsection{The graphs~$\CG_{\rC_3}$ and~$\CG_{\rC_2}$}
We do not have a complete list of $15$-vertex combinatorial $8$-manifolds like the quaternionic projective plane with the symmetry group~$\Sym(K)$ either~$\rC_3$ or~$\rC_2$. So we are not able to describe explicitly the graphs~$\CG_{\rC_3}$ and~$\CG_{\rC_2}$. Nevertheless, we can describe all connected components of these graphs that contain at least one vertex~$K$ with $|\Sym(K)|>3$. The result of the computation (made using the same program \texttt{triple\_flip\_graph}) is given in the following two propositions.

\begin{propos}\label{propos_graph_C3}
The graph~$\CG_{\rC_3}$ has two connected components that contain at least one vertex~$K$ with $|\Sym(K)|>3$.
\begin{itemize}
 \item The first connected component has $8$ vertices $\HP^2_{15}(\rA_5)$, $\HP^2_{15}(\rA_4,1)$,  $\HP^2_{15}(\rA_4,2)$, $\HP^2_{15}(\rS_3,1)$, $\HP^2_{15}(\rS_3,2)$, $\HP^2_{15}(\rC_3,1)$, $\HP^2_{15}(\rC_3,2)$, and~$\HP^2_{15}(\rC_3,3)$. It is shown in Fig.~\ref{fig_C3}.
 \item The second connected component has $4639$ vertices, namely,
 $$\HP^2_{15}(\rC_6\times\rC_2),\HP^2_{15}(\rS_3,3),\ldots,\HP^2_{15}(\rS_3,12),\HP^2_{15}(\rC_6,1),\ldots,\HP^2_{15}(\rC_6,14),$$
 and $4614$ triangulations with symmetry group exactly~$\rC_3$.
 \end{itemize}
\end{propos}

\begin{figure}[p]

 \begin{tikzpicture} \tiny
 \tikzmath{\minrad=.65cm; \innsep=0pt;\step=1.9;\dx=.8;\dy=.35;\dz=.6;}
  \draw (0, 2*\step) node (1) [circle, inner sep = \innsep, minimum size=\minrad, draw] {$\rA_5$};
  \draw (\step, \step) node (3) [circle, inner sep = \innsep, minimum size=\minrad, draw] {$\rA_4,\!1$};
  \draw (\step, 0) node (6) [circle, inner sep = \innsep, minimum size=\minrad, draw] {$\rC_3,\!1$};
  \draw (2*\step, \step) node (5) [circle, inner sep = \innsep, minimum size=\minrad, draw] {$\rC_3,\!2$};
  \draw (2*\step, 0) node (8) [circle, inner sep = \innsep, minimum size=\minrad, draw] {$\rA_4,\!2$};

  \draw (\step, 2*\step) node (2) [circle, inner sep = \innsep, minimum size=\minrad, draw] {$\rS_3,\!1$};
  \draw (2*\step, 2*\step) node (4) [circle, inner sep = \innsep, minimum size=\minrad, draw] {$\rC_3,\!3$};
  \draw (3*\step, 2*\step) node (7) [circle, inner sep = \innsep, minimum size=\minrad, draw] {$\rS_3,\!2$};

  \foreach \from/\to in {1/2,1/3,2/3,2/4,3/5,3/6,4/5,4/7,5/8,6/8}
    \draw (\from) -- (\to);
    \draw [dashed] (3) .. controls (.57*\step,.57*\step) and (.3917*\step,\step) .. (3);
    \draw [dashed] (6) .. controls (.57*\step,-.43*\step) and (.3917*\step,0) .. (6);
 \end{tikzpicture}
\vspace{-2mm}
 \caption{The connected components of~$\HP^2_{15}(\rA_5)$ in~$\CG_{\rC_3}$}\label{fig_C3}
 \vspace{.5cm}
 \end{figure}

 \begin{figure}[p]
 \begin{tikzpicture} \tiny
 \tikzmath{\minrad=.65cm; \innsep=0pt;\step=1.9;\dx=.8;\dy=.35;\dz=.6;}
  \draw (0, 3*\step) node (1) [circle, inner sep = \innsep, minimum size=\minrad, draw] {$\rA_5$};
  \draw (\step, 3*\step) node (3) [circle, inner sep = \innsep, minimum size=\minrad, draw] {$\rA_4,\!1$};
  \draw (2.5*\step, 3*\step) node (7) [circle, inner sep = \innsep, minimum size=\minrad, draw] {$\rC_2,\!1$};
  \draw (2.5*\step, 2*\step) node (5) [circle, inner sep = \innsep, minimum size=\minrad, draw] {$\rC_2,\!2$};
  \draw (3.5*\step, 3*\step) node (11) [circle, inner sep = \innsep, minimum size=\minrad, draw] {$\rA_4,\!2$};

  \draw (\step, 2*\step) node (2) [circle, inner sep = \innsep, minimum size=\minrad, draw] {$\rS_3,\!1$};
  \draw (\step, 0) node (9) [circle, inner sep = \innsep, minimum size=\minrad, draw] {$\rC_2,\!3$};
  \draw (2.5*\step, 0) node (12) [circle, inner sep = \innsep, minimum size=\minrad, draw] {$\rS_3,\!2$};

  \draw (\step-\dx, \step-\dy) node (4) [circle, inner sep = \innsep, minimum size=\minrad, draw] {$\rC_2,\!4$};
  \draw (\step+\dx, \step+\dy) node (6) [circle, inner sep = \innsep, minimum size=\minrad, draw] {$\rC_2,\!5$};
  \draw (2.5*\step-\dx, \step-\dy) node (8) [circle, inner sep = \innsep, minimum size=\minrad, draw] {$\rC_2,\!7$};
  \draw (2.5*\step+\dx, \step+\dy) node (10) [circle, inner sep = \innsep, minimum size=\minrad, draw] {$\rC_2,\!6$};

  \foreach \from/\to in {1/2,1/3,2/3,2/4,2/5,2/6,3/7,4/6,4/8,4/9,5/7,5/8,5/10,6/9,6/10,7/11,8/12,9/12,10/12}
    \draw (\from) -- (\to);
    \draw [dashed] (2) .. controls (.57*\step,1.57*\step) and (.3917*\step,2*\step) .. (2);
    \draw [dashed] (9) .. controls (.57*\step,-.43*\step) and (.3917*\step,0) .. (9);

 \end{tikzpicture}
 \vspace{-2mm}
 \caption{The connected components of~$\HP^2_{15}(\rA_5)$ in~$\CG_{\rC_2}$}\label{fig_C2}
\end{figure}
\afterpage{\clearpage}

For the graph~$\CG_{\rC_2}$ we have the same phenomenon as for~$\CG_{\rC_6}$. Namely, the group $\rC_6\times\rC_2$ contains three subgroups~$Q_1$, $Q_2$, and~$Q_3$ isomorphic to~$\rC_2$ and $\bigl(\HP^2_{15}(\rC_6\times\rC_2),Q_i\bigr)$ with $i=1,2,3$ are three different vertices of~$\CG_{\rC_2}$.

\begin{propos}\label{propos_graph_C2}
The graph~$\CG_{\rC_2}$ has $9$ connected components that contain at least one vertex~$K$ with $|\Sym(K)|>3$.
\begin{itemize}
 \item The first connected component has $12$ vertices $\HP^2_{15}(\rA_5)$, $\HP^2_{15}(\rA_4,1)$,  $\HP^2_{15}(\rA_4,2)$, $\HP^2_{15}(\rS_3,1)$, $\HP^2_{15}(\rS_3,2)$, $\HP^2_{15}(\rC_2,1),\ldots,\HP^2_{15}(\rC_2,7)$. It is shown in Fig.~\ref{fig_C2}.
 \item The second connected component has $1288$ vertices, namely,
 $$\HP^2_{15}(\rS_3,3),\ldots,\HP^2_{15}(\rS_3,6),$$ and $1284$ triangulations with symmetry group exactly~$\rC_2$.
 \item The third connected component has $271$ vertices, namely,
 $\HP^2_{15}(\rS_3,7)$, $\HP^2_{15}(\rS_3,8)$, $\HP^2_{15}(\rS_3,9)$, and $268$ triangulations with symmetry group exactly~$\rC_2$.
 \item The $4$th connected component has $480$ vertices, namely,
 $\HP^2_{15}(\rS_3,10)$, $\HP^2_{15}(\rS_3,11)$, and $478$ triangulations with symmetry group exactly~$\rC_2$.
 \item The $5$th connected component has $660$ vertices, namely, $\HP^2_{15}(\rS_3,12)$ and $659$ triangulations with symmetry group exactly~$\rC_2$.
 \item The $6$th connected component has $589$ vertices, namely,
 $$\bigl(\HP^2_{15}(\rC_6\times\rC_2),Q_1\bigr), \HP^2(\rC_6,1),\ldots,\HP^2(\rC_6,7),$$ and $581$  triangulations with symmetry group exactly~$\rC_2$.
 \item The $7$th connected component has $47$ vertices, namely, $$\bigl(\HP^2_{15}(\rC_6\times\rC_2),Q_2\bigr), \HP^2(\rC_6,8), \HP^2(\rC_6,9),$$ and $44$ triangulations with symmetry group exactly~$\rC_2$.
 \item The $8$th connected component has $6$ vertices, namely, $$\bigl(\HP^2_{15}(\rC_6\times\rC_2),Q_3\bigr), \HP^2(\rC_6,10),$$ and $4$ triangulations with symmetry group exactly~$\rC_2$.
 \item The $9$th connected component has $24$ vertices, namely, $$\HP^2(\rC_6,11),\ldots,\HP^2(\rC_6,14),$$ and $20$ triangulations with symmetry group exactly~$\rC_2$.
 \end{itemize}
\end{propos}

\begin{cor}\label{cor_5_man_>3}
 Each of the $75$ \ $15$-vertex combinatorial $8$-manifolds~$K$ like the quaternionic projective plane such that $|\Sym(K)|>3$ can be obtained by a sequence of triple flips from one of the five combinatorial manifolds
 \begin{equation}\label{eq_5_man}
 \HP^2_{15}(\rA_5),\ \HP^2_{15}(\rC_6\times\rC_2),\ \HP^2_{15}(\rC_7,1),\ \HP^2_{15}(\rC_7,11),\  \HP^2_{15}(\rC_5,1).
 \end{equation}
\end{cor}

\begin{cor}\label{cor_5_man_2,3}
 There are at least $4617$ (respectively, at least~$3345$) $15$-vertex combinatorial $8$-manifolds~$K$ like the quaternionic projective plane such that $\Sym(K)$ isomorphic to~$\rC_3$ (respectively, $\rC_2$) and $K$ can be obtained by a sequence of triple flips from one of the five combinatorial manifolds~\eqref{eq_5_man}.
\end{cor}

\subsection{Random walk and triangulations with trivial symmetry groups}\label{subsection_random}

When we try to use the program \texttt{triple\_flip\_graph} to calculate connected components of the non-equivariant triple flip graph~$\CG$, we encounter the following two difficulties:
\begin{enumerate}
 \item Every time we generate a new triangulation, we need to check whether it is isomorphic to one of the previously constructed ones. This check is time consuming.
 \item Every triangulation has $490$ maximal simplices, so it takes about 2kB of memory to store it. So it is memory consuming to store all the generated triangulations.
\end{enumerate}

The first difficulty is partially overcome in the following way. Suppose that $K$ is an arbitrary $15$-vertex combinatorial $8$-manifold like the quaternionic projective plane. By Corollary~\ref{cor_fvect} we have that the $f$-vector of~$K$ is given by~\eqref{eq_fvect_BK}. In particular, $K$ has exactly $4230$~\ $6$-simplices. The link of every $6$-simplex $\rho\in K$ is a triangulation of circle with at least $3$ and at most~$8$ vertices; we denote the number of vertices of $\link(\rho,K)$ by~$s(\rho)$. Equivalently, $s(\rho)$ is the number of $8$-simplices $\sigma\in K$ such that $\sigma\supset\rho$. Now, we can study the distribution of the function~$s(\rho)$, that is, compute the number~$m_s$ of $6$-simplices $\rho$ with each given value~$s$ of~$s(\rho)$, cf.~\cite[Section~3]{Gai22}. We always have
\begin{align}\label{eq_cert1}
m_3+m_4+\cdots+m_8&=4230,\\
\label{eq_cert2}
3m_3+4m_4+\cdots+8m_8&=36f_8(K)=36\cdot 490=17640.
\end{align}
However, the distributions $\mathbf{m}(K)=(m_3,\ldots,m_8)$ are typically different for different~$K$. It is technically convenient to organize the distribution~$\mathbf{m}(K)$ into a single integer
$$
\cert(K) =m_4+2^{12}\cdot m_5+2^{24}\cdot m_6 + 2^{36}\cdot m_7+2^{48}\cdot m_8,
$$
which we call the \textit{certificate} of~$K$. (Using equations~\eqref{eq_cert1} and~\eqref{eq_cert2}, one can easily show that~$\mathbf{m}(K)$ can be uniquely reconstructed from~$\cert(K)$.)
Whenever we want to know whether two triangulations~$K_1$ and~$K_2$ are isomorphic, we first compare their certificates and only if they are equal, we run a (rather slow) procedure of directly searching for isomorphism.

Nevertheless, if we try to use the program \texttt{triple\_flip\_graph} to calculate the connected components of~$\CG$ (except for the connected component~$\CG_0$, which is rather small), the task turns out to be beyond the capabilities of the computer. Even if we do not set the task of a complete description of the connected component of~$\CG$, but only want to generate more different triangulations using triple flips, the program does this too slowly. To generate triangulations more quickly, we use a random walk on the graph~$\CG$. The main advantage is that we do not need to store all the intermediate triangulations. Instead of storing the triangulations we only store their certificates, and compute the number of different certificates of triangulations with each symmetry group. This easy algorithm is implemented in the program \texttt{triple\_flip\_graph\_random} at~\cite{Gai-prog}. Certainly, the number of different certificates is typically smaller than the number of different generated triangulations. Nevertheless, for the trivial symmetry group, even the number of different certificates turns out to be huge.

\begin{propos}\label{propos_5_man_1}
 There are more than $670000$ \ $15$-vertex combinatorial $8$-manifolds~$K$ like the quaternionic projective plane with pairwise different certificates such that $K$ has trivial symmetry group and $K$ can be obtained by a sequence of triple flips from one of the five combinatorial manifolds~\eqref{eq_5_man}.
\end{propos}

\section{Topological type of the constructed combinatorial manifolds}\label{section_p1}

By Theorem~\ref{thm_BK} any $15$-vertex combinatorial $8$-manifold~$K$ that is not homeomorphic to~$S^8$ is a manifold like a projective plane. The construction of $8$-manifolds like a projective plane that are not homeomorphic to~$\HP^2$ goes back to Milnor~\cite{Mil56} and Eells and Kuiper~\cite{EeKu62}. This construction is as follows. Consider the sphere~$S^4$ and let $u\in H^4(S^4;\Z)$ be the standard generator. For each $k\in\Z$, there is a unique up to isomorphism oriented $4$-dimensional real vector bundle~$\xi_k$ over~$S^4$ such that $e(\xi_k)=u$ and $p_1(\xi_k)=2(2k-1)u$, where $e$ and~$p_1$ denote the Euler class and the first Pontryagin class, respectively. Let $D(\xi_k)$ be the total space of the associated disk bundle. Then $D(\xi_k)$ is a $8$-dimensional  smooth manifold with boundary and $\partial D(\xi_k)$ is PL homeomorphic (though not necessarily diffeomorphic) to~$S^7$. Hence, attaching to~$D(\xi_k)$ a cone over its boundary, we obtain a compact PL $8$-manifold without boundary, which we denote by~$Y_k^8$. Eells and Kuiper~\cite{EeKu62} showed that $Y_k^8$ is a manifold like a projective plane, i.\,e., admits a PL Morse function with three critical points. Simultaneous reversion of the orientations of~$S^4$ and~$\xi_k$ transforms~$\xi_k$ to~$\xi_{1-k}$, so $Y_k^8=Y_{1-k}^8$ for all~$k$. The manifolds~$Y_k^8$ with different~$k\ge 1$ are distinguished by their Pontryagin numbers, namely,
$$
p_1^2\bigl[Y^8_k\bigr]=4(2k-1)^2,
$$
where the orientation of~$Y^8_k$ is chosen so that the signature of~$Y^8_k$ is~$1$.
The classification of all $8$-dimensional manifolds like a projective plane was obtained by Kramer~\cite{Kra03}.

\begin{theorem}[Kramer~\cite{Kra03}]\label{thm_Kramer}
Suppose that $X$ is a $8$-dimensional PL manifold such that $H_*(X;\Z)\cong H_*(\HP^2;\Z)$. Then $X$ is PL homeomorphic to one of the manifolds~$Y^8_k$.
\end{theorem}

It is well known that $p_1^2[\HP^2]=4$.

\begin{cor}
Suppose that $X$ is a $8$-dimensional PL manifold such that $$H_*(X;\Z)\cong H_*(\HP^2;\Z).$$ Then $X$ is PL homeomorphic to~$\HP^2$ if and only if $p_1^2[X]=4$.
\end{cor}

\begin{remark}
In fact, Kramer proved a classification theorem for $8$-dimensional manifolds that are like a projective plane in topological category rather than in PL category. His proof uses hard Kirby--Siebenmann theory and the obtained classification result involves not only the Pontryagin number~$p_1^2[X]$ but also the Kirby--Siebenmann number $\mathrm{ks}^2[X]\in\Z/2\Z$. Namely, in topological category there are two series of $8$-dimensional manifolds like a projective plane with $\mathrm{ks}^2[X]=0$ and~$\mathrm{ks}^2[X]=1$, respectively. The manifolds in the latter series do not admit any PL structure. Nevertheless, one can repeat  Kramer's proof in PL category. Then the Kirby--Siebenmann number does not occur and one arrives at Theorem~\ref{thm_Kramer} avoiding the use of any hard results on PL structures on topological manifolds.
\end{remark}

If $K$ is a combinatorial $8$-manifold like a projective plane, then one can use an explicit combinatorial formula due to the author~\cite{Gai04} (see also~\cite{Gai05,Gai10,GaGo19}) to calculate the first rational Pontryagin class $p_1(K)\in H^4(K;\Q)$ and hence the number $p_1^2[K]$, and decide  whether $K$ is PL homeomorphic to~$\HP^2$. Using this strategy, Gorodkov~\cite{Gor16,Gor19} proved that the three Brehm--K\"uhnel combinatorial manifolds, $\HP^2_{15}(\rA_5)$, $\HP^2_{15}(\rA_4, 1)$, and $\HP^2_{15}(\rS_3,1)$ in our notation, are indeed PL homeomorphic to~$\HP^2$. The program he wrote to compute the first rational Pontryagin class can be found at~\cite{Gor-prog}. By my request Denis Gorodkov used the same program to calculate the first rational Pontryagin class for several new $15$-vertex combinatorial $8$-manifolds like a projective plane that are constructed in the present paper. The result is as follows.

\begin{propos}[Gorodkov]\label{propos_Gorodkov}
Suppose that $K$ is one of the four combinatorial manifolds $\HP^2_{15}(\rC_6\times\rC_2)$, $\HP^2_{15}(\rC_7,1)$, $\HP^2_{15}(\rC_7,11)$, or~$\HP^2_{15}(\rC_5,1)$. Then $p_1^2[K]=4$ and hence $K$ is PL homeomorphic to the quaternionic projective plane~$\HP^2$.
\end{propos}

If two combinatorial manifolds are obtained from each other by a triple flip, then they are PL homeomorphic. So, combining Proposition~\ref{propos_Gorodkov} with Corollaries~\ref{cor_5_man_>3} and~\ref{cor_5_man_2,3} and Proposition~\ref{propos_5_man_1}, we obtain the following corollary.

\begin{cor}\label{cor_HP2}
 All combinatorial manifolds constructed in this paper (see Table~\ref{table_main}) are PL homeomorphic to~$\HP^2$.
\end{cor}

\section{Concluding remarks and open questions}\label{section_conclude}

We would like to start with attracting the reader's attention to the following conjecture whose various versions appeared in the literature starting from~\cite{BrKu92} and~\cite{ArMa91}.

{\sloppy
\begin{conj}
 \begin{enumerate}
  \item Suppose that $K$ is a $15$-vertex $\Z$-homology $8$-manifold such that $H_*(K;\Z)$ is not isomorphic to~$H_*(S^8;\Z)$. Then $K$ is a combinatorial manifold that is PL homeomorphic to~$\HP^2$.
  \item Suppose that $K$ is a $27$-vertex $\Z$-homology $16$-manifold such that $H_*(K;\Z)$ is not isomorphic to~$H_*(S^{16};\Z)$. Then $K$ is a combinatorial manifold that is PL homeomorphic to the octonionic projective plane~$\OP^2$.
 \end{enumerate}
\end{conj}}

One can also consider a version of this conjecture with weaker assumptions on~$K$, for instance, for $\F_p$-homology manifolds, where $p$ is a prime, or even for arbitrary pseudomanifodls satisfying complimentarity. It is also interesting whether the classification result of the present paper can be extended to the case of $\F_p$-homology manifolds or pseudomanifodls with complimentarity.

\begin{question}
 \begin{enumerate}
  \item Suppose that $K$ is a $15$-vertex $\F_p$-homology $8$-manifold such that $H_*(K;\F_p)\ncong H_*(S^8;\F_p)$ and $|\Sym(K)|>3$. Is $K$ necessarily isomorphic to one of the $75$ triangulations of~$\HP^2$ from Table~\ref{table_main}?
  \item Suppose that $K$ is a $15$-vertex $8$-pseudomanifold such that $K$ satisfies complimentarity and $|\Sym(K)|>3$. Is $K$ necessarily isomorphic to one of the $75$ triangulations of~$\HP^2$ from Table~\ref{table_main}?
 \end{enumerate}
\end{question}

Theorem~\ref{thm_main} shows that the total number  of $15$-vertex triangulations of~$\HP^2$ (up to isomorphism) is large. However, it is completely unclear how large it is.

\begin{problem}
 Find reasonable lower and upper bounds for the number of different $15$-vertex triangulations of~$\HP^2$.
\end{problem}

It is also interesting to explore the properties of the triple flip graph~$\CG$.

\begin{question}
 \begin{enumerate}
  \item How many connected components does the graph~$\CG$ have?
  \item What are the diameters of the connected components of~$\CG$?
  \item Are there other `small' connected components of~$\CG$ besides the component~$\CG_0$ in Fig.~\ref{fig_G_component}?
 \end{enumerate}
\end{question}

In the $16$-dimensional case, among the $27$-vertex combinatorial $16$-manifolds like a projective plane constructed in~\cite{Gai22}, there are two (denoted by~$K_1$ and~$K_4$ in~\cite{Gai22}) that do not contain distinguished subcomplexes. It is an open question whether the same phenomenon can occur in dimension~$8$.

\begin{question}
 Is there a $15$-vertex triangulation of~$\HP^2$ (or a $8$-manifold like a projective plane) that contains no distinguished subcomplexes? Equivalently, does the graph~$\CG$ have an isolated vertex?
\end{question}

It is especially interesting to find out whether it is possible to connect by a sequence of triple flips two $15$-vertex triangulations of~$\HP^2$ with symmetry groups that are not related to each other, say, $\rC_7$ and~$\rC_6$. One result of this type was obtained by chance while studying the random walk on the graph~$\CG$.

\begin{propos}\label{propos_strange_path}
 The triangulations~$\HP^2_{15}(\rC_7,5)$ and~$\HP^2_{15}(\rS_3,6)$ lie in the same component of the graph~$\CG$.
\end{propos}

This proposition was obtained in the following way. A random walk trajectory that started at~$\HP^2_{15}(\rC_7,5)$ reached~$\HP^2_{15}(\rS_3,6)$ after about $10^5$ triple flips, which took about two days of computer calculation. Certainly, such result is rather unsatisfactory, since it does not give us any path of reasonable length from~$\HP^2_{15}(\rC_7,5)$ to~$\HP^2_{15}(\rS_3,6)$.

Combining Proposition~\ref{propos_strange_path} with Proposition~\ref{propos_graph_C3} and the description of the graph~$\CG_{\rC_7}$ (see Fig.~\ref{fig_C7}), we obtain the following corollary.

\begin{cor}
The $35$ triangulations
\begin{gather*}
\HP^2_{15}(\rC_6\times\rC_2),\HP^2_{15}(\rS_3,3),\ldots,\HP^2_{15}(\rS_3,12),\HP^2_{15}(\rC_6,1),\ldots,\HP^2_{15}(\rC_6,14),\\
\HP^2_{15}(\rC_7,1),\ldots, \HP^2_{15}(\rC_7,10)
\end{gather*}
lie in the same connected component~$\CG_1$ of~$\CG$.
\end{cor}

From the decriptions of the graphs~$\CG_{\rC_7}$ and~$\CG_{\rC_5}$ (see Figs.~\ref{fig_C7} and~\ref{fig_C5}, respectively) it follows that the $9$ triangulations
$$
\HP^2_{15}(\rC_7,11),\ldots, \HP^2_{15}(\rC_7,19)
$$
lie in the same connected component~$\CG_2$ of~$\CG$ and the $26$ triangulations
$$
\HP^2_{15}(\rC_5,1),\ldots, \HP^2_{15}(\rC_5,26)
$$
lie in the same connected component~$\CG_3$ of~$\CG$. Thus, the $75$ triangulations~$K$ in Table~\ref{table_main} with $|\Sym(K)|>3$ are distributed among at most $4$ connected components~$\CG_0$, $\CG_1$, $\CG_2$, and~$\CG_3$. Nevertheless, the following question is open.

\begin{question}
 Are the connected components~$\CG_1$, $\CG_2$, and~$\CG_3$ different?
\end{question}

\begin{problem}
For each pair of triangulations from Table~\ref{table_main} that lie in the same connected component of~$\CG$, find a sequence of triple flips of reasonable length transforming the first triangulation to the second one.
\end{problem}

For each of the constructed $15$-vertex triangulations of~$\HP^2$, the following numerical characteristics can be calculated:
\begin{itemize}
 \item the distribution $\mathbf{m}(K)=(m_3,\ldots,m_8)$ of $6$-simplices with respect to the number of adjacent $8$-simplices, see Subsection~\ref{subsection_random},
 \item the number $t$ of distinguished subcomplexes in~$K$.
\end{itemize}
In Table~\ref{table_numeric}, we provide the result of the calculation of these characteristics for all the $75$ triangulations~$K$ with $|\Sym(K)|>3$ and also for all triangulations in the connected component~$\CG_0$ of~$\CG$. Looking at these data, several observations can be made.

\begin{table}[p]
\caption{Distributions~$\mathbf{m}(K)$  and numbers~$t$ for different $15$-vertex triangulations of~$\HP^2$}\label{table_numeric}
\scriptsize
\begin{minipage}[t]{8.2cm}
 \begin{tabular}{|c|r|rrrrrr|r|}
 \hline
 \multirow{2}{*}{$\boldsymbol{G}$} & \multirow{2}{*}{\tiny\bf{}\!No\!} & \multicolumn{6}{c|}{\textbf{Distribution~$\mathbf{m}\boldsymbol{(K)}$}} & \multicolumn{1}{c|}{\multirow{2}{*}{$\boldsymbol{t}$}} \\
 &  & \tiny$\boldsymbol{m_3}$ & \tiny$\boldsymbol{m_4}$ & \tiny$\boldsymbol{m_5}$ & \tiny$\boldsymbol{m_6}$ & \multicolumn{1}{l}{\tiny$\lefteqn{\!\boldsymbol{m_7}}{}$} & \multicolumn{1}{l|}{\tiny$\lefteqn{\!\boldsymbol{m_8}}$} &  \\
 \hline
 $\rA_5$ &  &  1170 & 1740 & 870 & 360 & 60 & 30 & 5\\
 \hline
 \multirow{2}{*}{$\rA_4$} & 1 & 1206 & 1668 & 894 & 384 & 48 & 30 & 9\\
                          & 2 & 1080 & 1896 & 822 & 348 & 54 & 30 & 4\\
 \hline
 $\rC_6\!\times\!\rC_2$ & & 1248 & 1518 & 1044 & 342 & 72 & 6 & 24\\
 \hline
 \multirow{19}{*}{$\rC_7$} & 1 & 1281 & 1435 & 1071 & 399 & 44 & 0 & 21\\
                      & 2 & 1239 & 1554 & 966 &420 & 51 & 0 & 21 \\
                      & 3 & 1225 & 1526 & 1050 & 392 & 37 & 0 & 21\\
                      & 4 & 1316 & 1379 & 1071 & 427 & 37 & 0 & 28 \\
                      & 5 & 1225 & 1575 & 980 & 392 & 51 & 7 & 21 \\
                      & 6 & 1148 & 1673 & 1022 & 322 & 58 & 7 & 14 \\
                      & 7 & 1246 & 1519 & 1029 & 371 & 65 & 0 &  21 \\
                      & 8 & 1281 & 1449 & 1071 & 357 & 72 & 0 &   28 \\
                      & 9 & 1190 & 1603 & 987 & 427 & 23 & 0 &   14 \\
                      & 10 & 1120 & 1764 & 952 & 308 & 72 & 14 & 14 \\
                      & 11 & 1197 & 1666 & 882 & 420 & 65 & 0 & 14 \\
                      & 12 & 1309 & 1393 & 1071 & 413 & 44 & 0  & 28 \\
                      & 13 & 1204 & 1582 & 1008 & 392 & 44 & 0 & 21 \\
                      & 14 & 1302 & 1386 & 1127 & 357 & 51 & 7 & 28 \\
                      & 15 & 1232 & 1568 & 966 & 406 & 58 & 0 &  21 \\
                      & 16 & 1169 & 1638 & 1008 & 364 & 51 & 0 & 14 \\
                      & 17 & 1232 & 1470 & 1134 & 364 & 30 & 0 & 21  \\
                      & 18 & 1162 & 1617 & 1106 & 266 & 72 & 7 & 21  \\
                      & 19 & 1246 & 1512 & 1050 & 350 & 72 & 0 &  21 \\

 \hline
 \multirow{12}{*}{$\rS_3$} & 1 & 1224 & 1632 & 906 & 396 & 42 & 30 & 11\\
 & 2 & 1047 & 1968 & 780 & 348 & 57 & 30 & 6\\
 & 3 & 1116 & 1683 & 1053 & 351 & 27 & 0 & 14\\
 & 4 & 1206 & 1551 & 1071 & 351 & 51 & 0 & 20 \\
 & 5 & 1212 & 1521 & 1107 & 345 & 45 & 0 & 22 \\
 & 6 & 1152 & 1629 & 1053 & 369 & 27 & 0 & 16 \\
 & 7 & 1158 & 1674 & 987 & 345 & 63 & 3 & 14 \\
 & 8 & 1200 & 1548 & 1077 & 375 & 27 & 3 & 20\\
 & 9 & 1131 & 1695 & 1011 & 345 & 42 & 6 & 14 \\
 & 10 & 1167 & 1632 & 1062 & 282 & 87 & 0 & 14 \\
 & 11 & 1173 & 1617 & 1083 & 267 & 84 & 6 & 16 \\
 & 12 & 1137 & 1668 & 1056 & 312 & 51 & 6 & 14 \\

\hline
 \multirow{14}{*}{$\rC_6$}
  & 1 & 1200 & 1602 & 1023 & 327 & 69 & 9 & 18\\
  & 2 & 1128 & 1701 & 1047 & 273 & 69 & 12 & 12\\
  & 3 & 1146 & 1713 & 978 & 300 & 84 & 9 & 12\\
  & 4 & 1272 & 1458 & 1083 & 345 & 69 & 3 & 24\\
  & 5 & 1290 & 1461 & 1005 & 417 & 57 & 0 & 24 \\
  & 6 & 1125 & 1686 & 1065 & 285 & 66 & 3 & 12\\
  & 7 & 1239 & 1503 & 1077 & 351 & 60 & 0 & 18 \\
  & 8 & 1188 & 1632 & 990 & 348 & 66 & 6 & 18\\
  & 9 & 1038 & 1935 & 876 & 282 & 78 & 21 & 6 \\
  & 10 & 1134 & 1746 & 936 & 330 & 78 & 6 & 12\\
  & 11 & 1128 & 1704 & 1023 & 309 & 57 & 9 & 13\\
  & 12 & 1248 & 1491 & 1089 & 333 & 63 & 6 & 19 \\
  & 13 & 1242 & 1509 & 1071 & 339 & 63 & 6 & 19 \\
  & 14 & 1248 & 1497 & 1071 & 351 & 57 & 6 & 25 \\
\hline
\end{tabular}
\end{minipage}%
\begin{minipage}{7.4cm}
\begin{tabular}{|c|r|rrrrrr|r|}
\hline
 \multirow{2}{*}{$\boldsymbol{G}$} & \multirow{2}{*}{\tiny\bf{}\!No\!} & \multicolumn{6}{c|}{\textbf{Distribution~$\mathbf{m}\boldsymbol{(K)}$}} & \multicolumn{1}{c|}{\multirow{2}{*}{$\boldsymbol{t}$}} \\
 &  & \tiny$\boldsymbol{m_3}$ & \tiny$\boldsymbol{m_4}$ & \tiny$\boldsymbol{m_5}$ & \tiny$\boldsymbol{m_6}$ & \multicolumn{1}{l}{\tiny$\lefteqn{\!\boldsymbol{m_7}}{}$} & \multicolumn{1}{l|}{\tiny$\lefteqn{\!\boldsymbol{m_8}}$} &  \\
\hline
 \multirow{26}{*}{$\rC_5$}
 & 1 & 1190 & 1630 & 985 & 355 & 65 & 5 & 15 \\
 & 2 & 1245 & 1465 & 1150 & 300 & 65 & 5 & 25 \\
 & 3 & 1260 & 1430 & 1160 & 330 & 40 & 10 & 25 \\
 & 4 & 1235 & 1485 & 1145 & 295 & 60 & 10 & 25 \\
 & 5 & 1135 & 1705 & 995 & 335 & 50 & 10 & 15 \\
 & 6 & 1200 & 1615 & 975 & 375 & 65 & 0 & 15 \\
 & 7 & 1250 & 1470 & 1110 & 340 & 60 & 0 & 25\\
 & 8 & 1255 & 1475 & 1065 & 395 & 40 & 0 & 25 \\
 & 9 & 1220 & 1480 & 1165 & 325 & 35 & 5  & 25 \\
 & 10 & 1140 & 1695 & 1005 & 315 & 75 & 0  & 15 \\
 & 11 & 1135 & 1705 & 1000 & 320 & 65 & 5 & 15 \\
 & 12 & 1150 & 1660 & 1045 & 305 & 65 & 5 & 20 \\
 & 13 & 1170 & 1620 & 1045 & 345 & 45 & 5 & 20\\
 & 14 & 1105 & 1710 & 1070 & 290 & 45 & 10 & 15 \\
 & 15 & 1155 & 1665 & 1005 & 345 & 60 & 0 & 15 \\
 & 16 & 1175 & 1615 & 1035 & 355 & 50 & 0 & 20\\
 & 17 & 1165 & 1640 & 1020 & 350 & 55 & 0 & 20\\
 & 18 & 1140 & 1685 & 1035 & 285 & 85 & 0 & 15\\
 & 19 & 1145 & 1675 & 1030 & 310 & 65 & 5 & 15\\
 & 20 & 1155 & 1645 & 1065 & 285 & 80 & 0 & 20\\
 & 21 & 1165 & 1640 & 1020 & 350 & 55 & 0 & 20\\
 & 22 & 1090 & 1685 & 1150 & 260 & 40 & 5 & 15\\
 & 23 & 1140 & 1690 & 1015 & 315 & 65 & 5 & 15\\
 & 24 & 1170 & 1600 & 1100 & 300 & 50 & 10 & 20\\
 & 25 & 1175 & 1650 & 1010 & 310 & 75 & 10 & 20\\
 & 26 & 1130 & 1695 & 1050 & 280 & 60 & 15 & 15\\
 \hline
 \multirow{3}{*}{$\rC_3$} & 1 & 1161 & 1746 & 876 & 366 & 51 & 30 & 7\\
 & 2 & 1101 & 1851 & 849 & 345 & 54 & 30 & 5\\
 & 3 & 1119 & 1815 & 861 & 357 & 48 & 30 & 7\\
 \hline
 \multirow{7}{*}{$\rC_2$} & 1 & 1127 & 1804 & 864 & 352 & 53 & 30 & 6\\
         & 2 & 1145 & 1768 & 876 & 364 & 47 & 30 &  8\\
         & 3 & 1088 & 1874 & 840 & 346 & 52 & 30 &  7\\
         & 4 & 1134 & 1793 & 861 & 363 & 49 & 30 & 7\\
         & 5 & 1134 & 1793 & 861 & 363 & 49 & 30 & 7\\
         & 6 & 1077 & 1899 & 825 & 345 & 54 & 30 & 6\\
         & 7 & 1077 & 1899 & 825 & 345 & 54 & 30 & 6\\
 \hline
 \multirow{7}{*}{$\rC_1$} & 1 & 1179 & 1710 & 888 & 378 & 45 & 30 & 9 \\
 & 2 & 1145 & 1768 & 876 & 364 & 47 & 30 & 8 \\
 & 3 & 1085 & 1883 & 831 & 349 & 52 & 30 & 6 \\
 & 4 & 1062 & 1931 & 807 & 345 & 55 & 30 & 6 \\
 & 5 & 1100 & 1856 & 840 & 352 & 52 & 30 & 6 \\
 & 6 & 1111 & 1831 & 855 & 353 & 50 & 30 & 7 \\
 & 7 & 1111 & 1831 & 855 & 353 & 50 & 30 & 7 \\
 \hline
 \multicolumn{2}{c}{\phantom{7}} \\
 \multicolumn{2}{c}{\phantom{7}} \\
 \multicolumn{2}{c}{\phantom{7}} \\
 \multicolumn{2}{c}{\phantom{7}} \\
 \multicolumn{2}{c}{\phantom{7}} \\
 \multicolumn{2}{c}{\phantom{7}} \\
\end{tabular}
\end{minipage}
\end{table}
\afterpage{\clearpage}

\begin{observe}
In Table~\ref{table_numeric} there are four pairs of \textit{twins}, i.\,e.,  pairs $\{K_1,K_2\}$ such that
\begin{itemize}
 \item $\Sym(K_1)\cong\Sym(K_2)$,
 \item $\mathbf{m}(K_1)=\mathbf{m}(K_2)$,
 \item $t(K_1)=t(K_2)$,
 \item but $K_1\ncong K_2$.
\end{itemize}
These are the pairs
\begin{align*}
&\bigl\{\HP^2_{15}(\rC_5,17),\HP^2_{15}(\rC_5,21)\bigr\},&&\bigl\{\HP^2_{15}(\rC_2,4),\HP^2_{15}(\rC_2,5)\bigr\},\\
&\bigl\{\HP^2_{15}(\rC_2,6),\HP^2_{15}(\rC_2,7)\bigr\}, &&\bigl\{\HP^2_{15}(\rC_1,6),\HP^2_{15}(\rC_1,7)\bigr\}.
\end{align*}
The last three pairs of twins belong to the connected component~$\CG_0$ shown in Fig.~\ref{fig_G_component}.
Moreover, the graph~$\CG_0$ has an automorphism that swaps the vertices corresponding to every pair of twins and fixes every  other vertex. In addition, $\mathcal{G}_0$ contains a pair of \textit{almost twins}
$$
\bigl\{\HP^2_{15}(\rC_2,2),\HP^2_{15}(\rC_1,2)\bigr\}.
$$
The triangulations in this pair have equal distribution vectors~$\mathbf{m}$ and equal numbers of distinguished triples~$t$ but different symmetry groups.
\end{observe}

It would be nice to understand the reasons for these phenomena.

\begin{observe}
\begin{enumerate}
 \item All triangulations in~$\CG_0$ have $m_8=30$ and have rather small number~$t$ of distinguished subcomplexes, with the largest value~$11$ achieved at $\HP^2_{15}(\rS_3,1)$.
 \item All triangulations (in Table~\ref{table_numeric}) that are not in~$\CG_0$ have $m_8<30$ and typically have rather large number~$t$. Namely, if we exclude the triangulation~$\HP^2_{15}(\rC_6,9)$, which seems to be exceptional, then we always have $m_8\le 15$ and $t\ge 12$.
 \item We have also computed the number~$m_8$ for all triangulations with symmetry groups~$\rC_3$ and~$\rC_2$ from Propositions~\ref{propos_graph_C3} and~\ref{propos_graph_C2}, respectively, as well as for all triangulations that we have met during random walks in~$\CG$. The largest value of~$m_8$ among all considered triangulations that are not in~$\CG_0$ is~$23$. It is achieved at a triangulation with trivial symmetry group. Among triangulations with nontrivial symmetry groups, the largest value $21$ is achieved at~$\HP^2_{15}(\rC_6,9)$ and also at $8$ triangulations with symmetry group~$\rC_3$.
\end{enumerate}
\end{observe}

\begin{question}
 Is it true that $m_8<30$ for all $15$-vertex combinatorial $8$-manifolds~$K$ like a projective plane that are not in~$\CG_0$? What is the largest possible value of~$m_8$ for such combinatorial manifolds?
\end{question}

Another numerical characteristic for which the question of its extreme values is interesting is the number~$m_3$ of $6$-simplices that are contained in exactly three $8$-simplices. We have made the calculation for the $75$ triangulations with $|\Sym(K)|>3$, all triangulations with symmetry groups~$\rC_3$ and~$\rC_2$ from Propositions~\ref{propos_graph_C3} and~\ref{propos_graph_C2}, respectively, and all triangulations that we have met during random walks in~$\CG$.

\begin{observe}
Among all constructed triangulations:
\begin{itemize}
 \item The smallest value $m_3=1036$ is achieved at a triangulation (not belonging to~$\CG_0$) with trivial symmetry group. Among triangulations with non-trivial symmetry groups the smallest value $m_3=1038$ is achieved at~$\HP^2_{15}(\rC_6,9)$.
 \item The largest value $m_3=1335$ is achieved at a triangulation (not belonging to~$\CG_0$) with symmetry group~$\rC_3$.
\end{itemize}
\end{observe}

\begin{question}
What are the smallest and the largest possible values of~$m_3$ for $15$-vertex combinatorial $8$-manifolds like a projective plane?
\end{question}

Also, one can ask a similar question for the number~$t$ of distinguished subcomplexes.

\begin{question}
What are the smallest and the largest possible values of~$t$ for $15$-vertex combinatorial $8$-manifolds like a projective plane?
\end{question}

Among the triangulations in Table~\ref{table_numeric}, the smallest value $t=4$ is achieved at $\HP^2_{15}(\rA_4,2)$, while the largest value $t=28$ is achieved at four different triangulations with symmetry group~$\rC_7$.

Most triangulations in Table~\ref{table_numeric} have a $8$-simplex that belongs to two different distinguished subcomplexes. This means that the corresponding triple flips conflict with each other: by performing one of them, we destroy the second distinguished subcomplex. Interestingly, this phenomenon does not occur with~$\CP^2_9$ or any of the $27$-vertex combinatorial $16$-manifolds like a projective plane constructed in~\cite{Gai22}.

\begin{question}
 Does there exist a $27$-vertex combinatorial $16$-manifold like a projective plane that contains a $16$-simplex belonging to two different distinguished subcomplexes?
\end{question}

Finally, note that the triangulation~$\HP^2_{15}(\rC_6,9)$ seems to be  extremal or close to extremal from several viewpoints. Namely, it pretends to have almost the smallest~$m_3$ among all triangulations and also the smallest~$t$ and almost the largest~$m_8$ among all triangulations that are  not in~$\CG_0$. (Moreover, it seems likely that $m_3$ and~$m_8$ are truly the smallest and the largest, respectively, under the condition that the symmetry group is non-trivial.) Perhaps this means that this triangulation deserves more detailed study.

\medskip

The author is  grateful to Denis Gorodkov for performing the calculations of the first Pontryagin class and to Taras Panov and Vasilii Rozhdestvenskii for useful discussions and remarks.

\end{document}